\documentclass[10pt, notitlepage]{article}

\usepackage{mystyle}

\usepackage{ccfonts,eulervm}
\usepackage[T1]{fontenc}

\usepackage [english]{babel}
\usepackage [autostyle, english = american]{csquotes}
\MakeOuterQuote{"}

\usepackage[page,title,titletoc,header]{appendix}
\usepackage{relsize}
\usepackage{stmaryrd}
\def\Bast{\mathop{\mathlarger{\mathlarger{\mathlarger{\ast}}}}}
\begin{document}
\pagestyle{headings}
\title{Shift-minimal groups, fixed price 1, and the unique trace property}
\author{{\let\thefootnote\relax\let\thefootnote\relax}Robin D. Tucker-Drob\footnote{email:rtuckerd@caltech.edu}\\ Caltech\footnote{Department of Mathematics; California Institute of Technology; Pasadena, CA 91125}}
\date{December 21, 2012}
\begin{abstract}
%We introduce an ergodic theoretic property of countable groups called shift-minimality:
A countable group $\Gamma$ is called shift-minimal if every non-trivial measure preserving action of $\Gamma$ weakly contained in the Bernoulli shift $\Gamma \cc ([0,1]^\Gamma , \lambda ^\Gamma )$ is free. We show that any group $\Gamma$ whose reduced $C^*$-algebra $C^*_r(\Gamma )$ admits a unique tracial state is shift-minimal, and that any group $\Gamma$ admitting a free measure preserving action of cost$>1$ contains a finite normal subgroup $N$ such that $\Gamma /N$ is shift-minimal. Any shift-minimal group in turn is shown to have trivial amenable radical. Recurrence arguments are used in studying invariant random subgroups of a wide variety of shift-minimal groups. We also examine continuity properties of cost in the context of infinitely generated groups and equivalence relations. A number of open questions are discussed which concern cost, shift-minimality, $C^*$-simplicity, and uniqueness of tracial state on $C^*_r(\Gamma )$.
\end{abstract}
\maketitle
\setcounter{tocdepth}{3}
\tableofcontents

\section{Introduction}

The \emph{Bernoulli shift} of a countable discrete group $\Gamma$, denoted by $\bm{s}_\Gamma$, is the measure preserving action $\Gamma \cc ^s ([0,1]^\Gamma ,\lambda ^\Gamma )$ (where $\lambda$ denotes Lebesgue measure on $[0,1]$) of $\Gamma$ given by
\[
(\gamma ^s \cdot f)(\delta ) = f(\gamma ^{-1}\delta )
\]
for $\gamma ,\delta \in \Gamma$ and $f\in [0,1]^\Gamma$. If $\Gamma$ is infinite, then the Bernoulli shift may be seen as the archetypal free measure preserving action of $\Gamma$. This point of view is supported by Ab\'{e}rt and Weiss's result \cite{AW11} that $\bm{s}_\Gamma$ is weakly contained in every free measure preserving action of $\Gamma$. Conversely, it is well known that any measure preserving action weakly containing a free action must itself be free. A measure preserving action is therefore free if and only if it exhibits approximate Bernoulli behavior.

Inverting our point of view, the approximation properties exhibited by $\bm{s}_\Gamma$ itself have been shown to reflect the group theoretic nature of $\Gamma$. One example of this is Schmidt's characterization \cite{Sc81} of amenable groups as exactly those groups $\Gamma$ for which $\bm{s}_\Gamma$ admits a non-trivial sequence of almost invariant sets. An equivalent formulation in the language of weak containment is that $\Gamma$ is amenable if and only if $\bm{s}_\Gamma$ weakly contains an action that is not ergodic. In addition, a direct consequence of Foreman and Weiss's work \cite{FW04} is that amenability of $\Gamma$ is equivalent to \emph{every} measure preserving action of $\Gamma$ being weakly contained in $\bm{s}_\Gamma$. That each of these properties of $\bm{s}_\Gamma$ is necessary for amenability of $\Gamma$ is essentially a consequence of the Ornstein-Weiss Theorem \cite{OW80}, while sufficiency of these properties may be reduced to the corresponding representation theoretic characterizations of amenability due to Hulanicki and Reiter (see \cite{Hu64, Hu66}, \cite[7.1.8]{Zi84}, \cite[Appendix G.3]{BHV08}): a group $\Gamma$ is amenable if and only if its left regular representation $\lambda _\Gamma$ weakly contains the trivial representation if and only if $\lambda _\Gamma$ weakly contains every unitary representation of $\Gamma$.

This paper investigates further the extent to which properties of a group may be detected by its Bernoulli action. Roughly speaking, it is observed that even when a group is non-amenable, the manifestation (or lack thereof) of certain behaviors in the Bernoulli shift has implications for the extent of that group's non-amenability. Central to this investigation is the following definition.

\begin{definition}\label{def:sm}
A countable group $\Gamma$ is called \emph{shift-minimal} if every non-trivial measure preserving action weakly contained in $\bm{s}_\Gamma$ is free.
\end{definition}

\noindent The reader is referred to \cite{Ke10} for background on weak containment of measure preserving actions. Note that by definition the trivial group $\{ e \}$ is shift-minimal.

Shift-minimality, as with the above-mentioned ergodic theoretic characterizations of amenability, takes its precedent in the theory of unitary representations of $\Gamma$. It is well known that $\Gamma$ is \emph{$C^*$-simple} (i.e., its reduced $C^*$-algebra $C^*_r(\Gamma )$ is simple) if and only if every non-zero unitary representation of $\Gamma$ weakly contained in the left-regular representation $\lambda _\Gamma$ is actually weakly equivalent to $\lambda_\Gamma$ \cite{Ha07}.  Using the Ab\'{e}rt-Weiss characterization of freeness it is apparent that $\Gamma$ is shift-minimal if and only if every non-trivial m.p.\ action of $\Gamma$ weakly contained in $\bm{s}_\Gamma$ is in fact weakly equivalent to $\bm{s}_\Gamma$.  Apart from analogy, the relationship between shift-minimality and $C^*$-simplicity in general is unclear. However, we show in Theorem \ref{thm:UTsm} that shift-minimality follows from a property closely related to $C^*$-simplicity. A group $\Gamma$ is said to have the \emph{unique trace property} if there is a unique tracial state on $C^*_r(\Gamma )$.

\begin{theorem}\label{thm:introUT}
Let $\Gamma$ be a countable group. If $\Gamma$ has the unique trace property then $\Gamma$ is shift-minimal.
\end{theorem}

In addition, a co-induction argument (Proposition \ref{prop:noamensub}) shows that shift-minimal groups have no non-trivial normal amenable subgroups, i.e., they have trivial amenable radical. This places shift-minimality squarely between two other properties whose general equivalence with $C^*$-simplicity remains an open problem. Indeed, it is open whether there are any general implications between $C^*$-simplicity and the unique trace property; in all concrete examples these two properties coincide. Furthermore, while the amenable radical of any $C^*$-simple group is known to be trivial \cite{PS79}, it is an open question -- asked explicitly by Bekka and de la Harpe \cite{BH00} -- whether conversely, a group which is not $C^*$-simple always contains a non-trivial normal amenable subgroup. For shift-minimality in place of $C^*$-simplicity, a stochastic version of this question is shown to have a positive answer (Theorem \ref{thm:smequiv}).

\begin{theorem}\label{thm:smiff}
A countable group $\Gamma$ is shift-minimal if and only if there is no non-trivial amenable invariant random subgroup of $\Gamma$ weakly contained in $\bm{s}_\Gamma$.
\end{theorem}
Here an \emph{invariant random subgroup} (IRS) of $\Gamma$ is a Borel probability measure on the compact space $\mbox{Sub}_\Gamma$ of subgroups of $\Gamma$ that is invariant under the conjugation action $\Gamma \cc \mbox{Sub}_\Gamma$ of $\Gamma$.  It is called \emph{amenable} if it concentrates on the amenable subgroups of $\Gamma$. Invariant random subgroups generalize the notion of normal subgroups: if $N$ is a normal subgroup of $\Gamma$ then the Dirac measure $\updelta _N$ on $\mbox{Sub}_\Gamma$ is conjugation invariant. It is shown in \cite{AGV12} that the invariant random subgroups of $\Gamma$ are precisely those measures on $\mbox{Sub}_\Gamma$ that arise as the stabilizer distribution of some measure preserving action of $\Gamma$ (see \S\ref{subsec:IRS}).

Theorem \ref{thm:smiff} is not entirely satisfactory since it still seems possible that shift-minimality of $\Gamma$ is equivalent to $\Gamma$ having no non-trivial amenable invariant random subgroups whatsoever (see Question \ref{question:UIRS}).  In fact, the proof of Theorem \ref{thm:introUT} in $\S$\ref{sec:UT} shows that this possibly stronger property is a consequence of the unique trace property.

\begin{theorem}
Let $\Gamma$ be a countable group. If $\Gamma$ has the unique trace property then $\Gamma$ has no non-trivial amenable invariant random subgroups.
\end{theorem}

The known general implications among all of the notions introduced thus far are expressed in Figure \ref{fig:rel}. %In what follows we say that a group $\Gamma$ has the \emph{unique trace property} if there is a unique tracial state on $C^*_r(\Gamma )$.
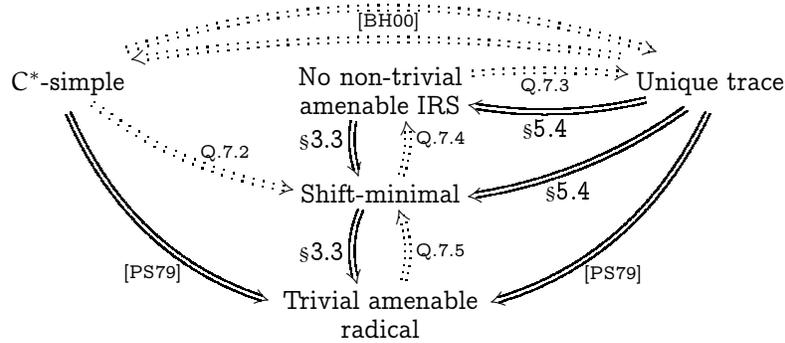
\begin{figure}[h!]\label{fig:rel}
\[
\xymatrix{
C^*\mbox{-simple}
    \ar @<-1ex> @/_2pc/ @{=>}[ddrr]_{\mbox{\scriptsize{\cite{PS79}}}}
    \ar @<1ex> @/^2pc/ @{:>}[rrrr]
    \ar @/_1pc/ @{:>}[drr]^{\mbox{\scriptsize{Q.\ref{Q:CS->SM}}}}
%    \ar @<-1ex>@/^1pc/ @{:>}[rr]
    &
    &
 \txt{No non-trivial \\ amenable IRS}
 \ar @<-.25ex> @/^/ @{:>}[rr]_{\mbox{\scriptsize{Q.\ref{Q:UIRS->UT}}}}
 \ar @<-1ex> @/_/ @{=>}[d]_{\S\ref{sec:airs}}
% \ar @<-1ex>@/^1pc/ @{>}[rr]_{??}
%  \ar @<-1ex>@/^1pc/ @{:>}[ll]_{??}
    &
    &
\mbox{Unique trace}
    \ar @<-1ex> @/_1pc/ @{:>}[llll]_{\mbox{\scriptsize{\cite{BH00}}}}
    \ar @<-.25ex>@/^1pc/ @{=>}[ll]^{\S\ref{sec:UT}}
    \ar @<.5ex> @/^1pc/ @{=>}[dll]^{\S\ref{sec:UT}}
    \ar @<1ex> @/^2pc/ @{=>}[ddll]^{\mbox{\scriptsize{\cite{PS79}}}}
    \\
  &
  &  \mbox{Shift-minimal}
  \ar @<-1ex> @/_/ @{=>}[d]_{\S\ref{sec:airs}}
  \ar @<-1ex> @/_/ @{:>}[u]_{\mbox{\scriptsize{Q.\ref{question:UIRS}}}}
  &
  &
  \\
  &
  &  \txt{Trivial amenable \\ radical}
  \ar @<-1ex> @/_/ @{:>}[u]_{\mbox{\scriptsize{Q.\ref{question:AReSM}}}}
  &
  &
}
\]
\caption{The solid lines indicate known implications and the dotted lines indicate open questions discussed in \S\ref{sec:Questions}. Any implication which is not addressed by the diagram is open in general. However, these properties all coincide for large classes of groups, e.g., linear groups (see \S\ref{sec:linear}).}
\end{figure}
%
%\noindent The solid lines indicate known implications and the dotted lines indicate open questions discussed in \S\ref{sec:Questions}. %(see questions \ref{Q:SM->UT}, \ref{Q:CS->SM}, \ref{question:UIRS} and \ref{question:AReSM}).

Our starting point in studying shift-minimality is the observation that if $\Gamma \cc ^a (X,\mu )$ is a m.p.\ action that is weakly contained in $\bm{s}_\Gamma$ then every non-amenable subgroup of $\Gamma$ acts ergodically. We call this property of a m.p.\ action \emph{NA-ergodicity}. We show in Theorem \ref{thm:NAerg} that when a m.p.\ action of $\Gamma$ is NA-ergodic then the stabilizer of almost every point must be amenable.

\S\ref{sec:hered} deals with permanence properties of shift-minimality by examining situations in which freeness of a m.p.\ action $\Gamma \cc ^a (X,\mu )$ may be deduced from freeness of some acting subgroup. Many of the proofs in this section appeal to some form of the Poincar\'{e} Recurrence Theorem. %Combining the results here with the analysis of NA-ergodicity we show the following in Corollaries \ref{cor:NAalmostasc} and \ref{cor:NAcor}.
%
%\begin{theorem}$\ $
%Let $\Gamma$ be a countable group.
%\begin{enumerate}
%\item Suppose that $\mbox{\emph{AR}}_\Gamma = \{ e\}$. If $\Gamma$ contains a non-trivial shift-minimal almost ascendant subgroup $L$ then $\Gamma$ is itself shift-minimal.
%\item Suppose that $\Gamma$ contains an ascendant subgroup $L$ such that $L$ is shift-minimal and $\mbox{\emph{AR}}_{C_\Gamma (L)} = \{ e \}$. Then $\Gamma$ is shift-minimal. In particular, if both $L$ and $C_\Gamma (L)$ are shift-minimal then so is $\Gamma$.
%\end{enumerate}
%\end{theorem}
%\noindent Here $\mbox{AR}_\Gamma$ denotes the amenable radical of $\Gamma$. See Appendix \ref{app:ar} for the definition of ascendant and almost ascendant subgroups. Also note that by definition the trivial group $\{ e \}$ is shift-minimal.

A wide variety of groups are known to have the unique trace property and Theorem \ref{thm:introUT} shows that all such groups are shift-minimal. Among these are all non-abelian free groups (\cite{Pow75}), all Powers groups and weak Powers groups (\cite{Ha85}, \cite{BN88}), groups with property $\mbox{P}_{\mbox{\tiny{nai}}}$ \cite{BCH94}, all ICC relatively hyperbolic groups (\cite{AM07}), and all ICC groups with a minimal non-elementary convergence group action \cite{MOY11}. In \S\ref{sec:examples} we observe that all of these groups share a common paradoxicality property we call (BP), abstracted from M. Brin and G. Picioroaga's proof that all weak Powers groups contain a free group (see \cite[following Question 15]{Ha07}).  It is shown in Theorem \ref{thm:BPirs} that any non-trivial ergodic invariant random subgroup of a group with property (BP) must contain a non-abelian free group almost surely. Recurrence once again plays a key role in the proof.

\S\ref{sec:cost} studies the relationship between cost, weak containment, and invariant random subgroups. Kechris shows in \cite[Corollary 10.14]{Ke10} that if $\bm{a}$ and $\bm{b}$ are free measure preserving actions of a finitely generated group $\Gamma$ then $\bm{a}\prec \bm{b}$ implies $C(\bm{b})\leq C(\bm{a})$ where $C(\bm{a})$ denotes the cost of $\bm{a}$ (i.e., the cost of the orbit equivalence relation generated by $\bm{a}$). This is deduced from the stronger fact \cite[Theorem 10.13]{Ke10} that the cost function $C: \mbox{FR}(\Gamma ,X,\mu ) \ra \R$, $\bm{a}\mapsto C(\bm{a})$, is upper semi-continuous for finitely generated $\Gamma$. In \S\ref{sec:infgen} we obtain a generalization which holds for arbitrary countable groups (Theorem \ref{thm:KecOpen} below). The consequences of this generalization are most naturally stated in terms of an invariant we call \emph{pseudocost}.

If $E$ is a m.p.\ countable Borel equivalence relation on $(X,\mu )$ then the pseudocost of $E$ is defined as $PC_\mu (E) := \inf _{(E_n)}\liminf _n C_\mu (E_n)$, where $(E_n)_{n\in \N}$ ranges over all increasing sequences $E_0\subseteq E_1\subseteq \cdots$, of Borel subequivalence relations of $E$ such that $\bigcup _n E_n =E$. The pseudocost of an action and of a group is then defined in analogy with cost (see Definition \ref{def:pseudo}). It is immediate that $PC_\mu (E)\leq C_\mu (E)$, and while the pseudocost and cost coincide in most cases, including whenever $E$ is treeable or whenever $C_\mu (E)<\infty$ (Corollary \ref{cor:PC=C}), it is unclear whether equality holds in general.

One of the main motivations for introducing pseudocost is the following useful continuity property (Corollary \ref{cor:AFrPC}):

\begin{theorem}\label{thm:weakconPC}
Let $\bm{a} = \Gamma \cc ^a (X,\mu )$ and $\bm{b} =\Gamma \cc ^b(Y,\nu )$ be measure preserving actions of a countable group $\Gamma$. Assume that $\bm{a}$ is free. If $\bm{a}\prec \bm{b}$ then $PC(\bm{b})\leq PC(\bm{a})$.
\end{theorem}

Combining Theorem \ref{thm:weakconPC} and \cite[Theorem 1]{AW11} it follows that, among all free m.p.\ actions of $\Gamma$, the Bernoulli shift $\bm{s}_\Gamma$ has the maximum pseudocost. Since pseudocost and cost coincide for m.p.\ actions of finitely generated groups, this generalizes the result of \cite{AW11} that $\bm{s}_\Gamma$ has the greatest cost among free actions of a finitely generated group $\Gamma$. In Corollary \ref{cor:weakcon} we use Theorem \ref{thm:weakconPC} to deduce general consequences for cost:

\begin{theorem}
Let $\bm{a}$ and $\bm{b}$ be m.p.\ actions of a countably infinite group $\Gamma$. Assume that $\bm{a}$ is free and $\bm{a}\prec \bm{b}$.
\begin{enumerate}
\item If $C (\bm{b})<\infty$ then $C (\bm{b})\leq C (\bm{a})$.
\item If $E_b$ is treeable then $C(\bm{b})\leq C(\bm{a})$.
\item If $C(\bm{a})=1$ then $C(\bm{b})=1$.
\end{enumerate}
\end{theorem}

This leads to a characterization of countable groups with fixed price $1$ as exactly those groups whose Bernoulli shift has cost $1$. This characterization was previously shown for finitely generated groups in \cite{AW11}.

\begin{theorem}
Let $\Gamma$ be a countable group. The following are equivalent:
\begin{enumerate}
\item[(1)] $\Gamma$ has fixed price $1$
\item[(2)] $C(\bm{s}_\Gamma ) = 1$
\item[(3)] $C(\bm{a}) = 1$ for some m.p.\ action $\bm{a}$ weakly equivalent to $\bm{s}_\Gamma$.
\item[(4)] $PC(\bm{a})=1$ for some m.p.\ action $\bm{a}$ weakly equivalent to $\bm{s}_\Gamma$.
\item[(5)] $\Gamma$ is infinite and $C(\bm{a})\leq 1$ for some non-trivial m.p.\ action $\bm{a}$ weakly contained in $\bm{s}_\Gamma$.
\end{enumerate}
\end{theorem}

We use this characterization to obtain a new class of shift-minimal groups in \S\ref{sec:fp1sm}. In what follows, $\mbox{AR}_\Gamma$ denotes the amenable radical of $\Gamma$ (see Appendix \ref{app:ar}). Gaboriau \cite[Theorem 3]{Ga00} showed that if $\Gamma$ does not have fixed price $1$ then $\mbox{AR}_\Gamma$ is finite. We now have:

\begin{theorem}\label{thm:fpcost>1}
Let $\Gamma$ be a countable group that does not have fixed price $1$. Then $\mbox{\emph{AR}}_\Gamma$ is finite and $\Gamma /\mbox{\emph{AR}}_\Gamma$ is shift-minimal.
\end{theorem}

In Theorem \ref{thm:gencost1} of \S\ref{sec:costIRS} it is shown that if the hypothesis of Theorem \ref{thm:fpcost>1} is strengthened to $C(\Gamma )>1$, i.e., if \emph{all} free m.p.\ actions of $\Gamma$ have cost ${}>1$, then the conclusion may be strengthened considerably. The following is an analogue of Bergeron and Gaboriau's result \cite[\S 5]{BG05} (see also \cite[Corollary 1.6]{ST10}) in which the statement is shown to hold for the first $\ell ^2$-Betti number in place of cost.

\begin{theorem}\label{thm:BGintro}
Suppose that $C(\Gamma )>1$. Let $\Gamma \cc ^a (X,\mu )$ be an ergodic measure-preserving action of $\Gamma$ on a non-atomic probability space. Then exactly one of the following holds:
\begin{enumerate}
\item Almost all stabilizers are finite;
\item Almost every stabilizer has infinite cost, i.e., $C(\Gamma _x) = \infty$ almost surely.
\end{enumerate}
In particular, $\mbox{\emph{AR}}_\Gamma$ is finite and $\Gamma /\mbox{\emph{AR}}_\Gamma$ has no non-trivial amenable invariant random subgroups.
\end{theorem}

The analysis of pseudocost in \S\ref{sec:infgen} is used in \S\ref{sec:generic} to study the cost of generic actions in the Polish space $A(\Gamma ,X,\mu )$ of measure preserving actions of $\Gamma$. For any group $\Gamma$ there is a comeager subset of $A(\Gamma ,X,\mu )$, consisting of free actions, on which the cost function $C:A(\Gamma ,X,\mu ) \ra \R\cup \{\infty \}$ takes a constant value $C_{\mbox{\tiny{gen}}}(\Gamma )\in \R$ \cite{Ke10}. Likewise, the pseudocost function $\bm{a}\mapsto PC(\bm{a})$ must be constant on a comeager set of free actions, and we denote this constant value by $PC_{\mbox{\tiny{gen}}}(\Gamma )$. Kechris shows in \cite{Ke10} that $C_{\mbox{\tiny{gen}}}(\Gamma ) =C(\Gamma )$ for finitely generated $\Gamma$ and Problem 10.11 of \cite{Ke10} asks whether $C_{\mbox{\tiny{gen}}}(\Gamma ) =C(\Gamma )$ in general. The following is proved in Corollaries \ref{cor:Gdelta} and \ref{cor:Cgen}.

\begin{theorem}
Let $\Gamma$ be a countably infinite group. Then
\begin{enumerate}
\item The set $\{ \bm{a}\in \mbox{A}(\Gamma ,X,\mu )\csuchthat \bm{a}\mbox{ is free and }PC(\bm{a})=PC(\Gamma ) \}$ is dense $G_\delta$ in $A(\Gamma ,X,\mu )$.
\item $PC_{\mbox{\tiny{gen}}}(\Gamma )=PC(\Gamma )$.
\item Either $C_{\mbox{\tiny{gen}}}(\Gamma ) = C(\Gamma )$ or $C_{\mbox{\tiny{gen}}}(\Gamma )= \infty$.
\item If $PC(\Gamma ) =1$ then $C_{\mbox{\tiny{gen}}}(\Gamma ) = C(\Gamma ) =1$.
\item The set
\begin{align*}
\big\{\bm{b}\in A(\Gamma ,X,\mu ) \csuchthat &\bm{b}\mbox{ is free and }\exists\mbox{aperiodic Borel subequivalence relations }\\
&E_0\subseteq E_1\subseteq E_2\subseteq \cdots \mbox{ of } E_b, \mbox{ with } E_b = \bigcup _n E_n \mbox{ and }\lim _n C_\mu (E_n)= C(\Gamma ) \big\}
\end{align*}
is dense $G_\delta$ in $A(\Gamma ,X,\mu )$.
\item If all free actions of $\Gamma$ have finite cost then $\{ \bm{b}\in A(\Gamma ,X,\mu )\csuchthat \bm{b}\mbox{ is free and }C(\bm{b})=C(\Gamma ) \}$ is dense $G_\delta$ in $A(\Gamma ,X,\mu )$.
\end{enumerate}
\end{theorem}

The only possible exception to the equality $C_{\mbox{\tiny{gen}}}(\Gamma )= C(\Gamma )$ would be a group $\Gamma$ with $C(\Gamma ) < \infty$ such that the set $\{ \bm{a}\in A(\Gamma ,X,\mu ) \csuchthat \bm{a}\mbox{ is free, }C(\bm{a}) = \infty\mbox{ and }E_a \mbox{ is not treeable}\}$ comeager in $A(\Gamma ,X,\mu )$.

A number of questions are discussed in \S\ref{sec:Questions}. The paper ends with two appendices, the first clarifying the relationship between invariant random subgroups and subequivalence relations. The second contains relevant results about the amenable radical of a countable group.

{\bf Acknowledgments}. I would like to thank Miklos Ab\'{e}rt, Lewis Bowen, Clinton Conley, Yair Glasner, Alexander Kechris, and Jesse Peterson for valuable comments and suggestions. The research of the author was partially supported by NSF Grant DMS-0968710.

\section{Preliminaries}

Throughout, $\Gamma$ denotes a countable discrete group. The identity element of $\Gamma$ is denoted by $e_\Gamma$, or simply $e$ when $\Gamma$ is clear from the context. All countable groups are assumed to be equipped with the discrete topology; a \emph{countable group} always refers to a countable discrete group.

\subsection{Group theory}
\noindent {\bf Subgroups.} Let $\Delta$ and $\Gamma$ be countable groups. We write $\Delta \leq \Gamma$ to denote that $\Delta$ is a subgroup of $\Gamma$ and we write $\Delta \triangleleft \Gamma$ to denote that $\Delta$ is a normal subgroup of $\Gamma$. The index of a subgroup $\Delta\leq \Gamma$ is denoted by $[\Gamma : \Delta ]$. The \emph{trivial} subgroup of $\Gamma$ is the subgroup $\{ e_\Gamma \}$ that contains only the identity element. For a subset $S\subseteq \Gamma$ we let $\langle S\rangle$ denote the subgroup generated by $S$. A group that is not finitely generated will be called \emph{infinitely generated}.

\medskip

\noindent {\bf Centralizers and normalizers.} Let $S$ be any subset of $\Gamma$ and let $H$ be a subgroup of $\Gamma$. The \emph{centralizer} of $S$ in $H$ is the set
\[
C_H (S) = \{ h \in H \csuchthat \forall s\in S \ (hs h^{-1} = s ) \}
\]
and the \emph{normalizer} of $S$ in $H$ is the set
\[
N_H (S) = \{ h \in H\csuchthat hSh^{-1} = S \} .
\]
Then $C_H (S)$ and $N_H (S)$ are subgroups of $H$ with $C_H (S)\triangleleft N_H (S)$. Clearly $C_H(S)= C_\Gamma (S) \cap H$ and $N_H(S) = N_\Gamma (S)\cap H$. The group $C_\Gamma (\Gamma )$ is called the \emph{center} of $\Gamma$ and is denoted by $Z(\Gamma )$. We say that a subset $T$ of $\Gamma$ \emph{normalizes} $S$ if $T\subseteq N_\Gamma (S)$. We call a subgroup $H$ \emph{self-normalizing} in $\Gamma$ if $H= N_\Gamma (H)$.

\medskip

\noindent {\bf Infinite conjugacy class (ICC) groups.} A group $\Gamma$ is called \emph{ICC} if every $\gamma \in \Gamma \setminus \{ e \}$ has an infinite conjugacy class. This is equivalent to $C_\Gamma (\gamma )$ having infinite index in $\Gamma$ for all $\gamma \neq e$. Thus, according to our definition, the trivial group $\{ e \}$ is ICC.

\medskip

\noindent {\bf The Amenable Radical.} We let $\mbox{AR}_\Gamma$ denote the \emph{amenable radical} of $\Gamma$. See Appendix \ref{app:ar} below.

\subsection{Ergodic theory}

\noindent {\bf Measure preserving actions.} A \emph{measure preserving (m.p.) action} of $\Gamma$ is a triple $(\Gamma , a , (X,\mu ))$, which we write as $\Gamma \cc ^ a (X,\mu )$, where $(X,\mu )$ is a standard probability space (possibly with atoms) and $a:\Gamma\times X\ra X$ is a Borel action of $\Gamma$ on $X$ that preserves the probability measure $\mu$. For $(\gamma ,x)\in \Gamma \times X$ we let $\gamma ^ax$ denote the image $a(\gamma ,x)$ of the pair $(\gamma ,x)$ under $a$. We write $\bm{a}$ for $\Gamma \cc ^a (X,\mu )$ when $\Gamma$ and $(X,\mu )$ are clear from the context. A measure preserving action of $\Gamma$ will also be called a \emph{$\Gamma$-system} or simply a \emph{system} when $\Gamma$ is understood.

For the rest of this subsection let $\bm{a}=\Gamma \cc ^a (X,\mu )$ and let $\bm{b}= \Gamma \cc ^b(Y,\nu )$.

\medskip

\noindent {\bf Isomorphism and factors.} If $\varphi : (X,\mu ) \ra Y$ is a measurable map then we let $\varphi _*\mu$ denote the pushforward measure on $Y$ given by $\varphi _*\mu (B) = \mu (\varphi ^{-1}(B))$ for $B\subseteq Y$ Borel. We say that $\bm{b}$ is a \emph{factor} of $\bm{a}$ (or that $\bm{a}$ \emph{factors onto} $\bm{b}$), written $\bm{b}\sqsubseteq \bm{a}$, if there exists a measurable map $\pi : X \ra Y$ with $\pi _*\mu = \nu$ and such that for each $\gamma \in \Gamma$ the equality $\pi (\gamma ^a x )= \gamma ^b\pi (x)$ holds for $\mu$-almost every $x\in X$. Such a map $\pi$ is called a \emph{factor map} from $\bm{a}$ to $\bm{b}$. The factor map $\pi$ is called an \emph{isomorphism} from $\bm{a}$ to $\bm{b}$ if there exists a co-null subset of $X$ on which $\pi$ is injective. We say that $\bm{a}$ and $\bm{b}$ are \emph{isomorphic}, written $\bm{a}\cong \bm{b}$, if there exists some isomorphism from $\bm{a}$ to $\bm{b}$.

\medskip

\noindent {\bf Weak containment of m.p.\ actions.} We write $\bm{a}\prec \bm{b}$ to denote that $\bm{a}$ is weakly contained in $\bm{b}$ and we write $\bm{a}\sim \bm{b}$ to denote that $\bm{a}$ and $\bm{b}$ are weakly equivalent. The reader is referred to \cite{Ke10} for background on weak containment of measure preserving actions.

\medskip

\noindent {\bf Product of actions.} The \emph{product} of $\bm{a}$ and $\bm{b}$ is the m.p.\ action $\bm{a}\times \bm{b} = \Gamma \cc ^{a\times b}(X\times Y,\mu \times \nu )$ where $\gamma ^{a\times b}(x,y)= (\gamma ^a x, \gamma ^b y )$ for each $\gamma \in \Gamma$ and $(x,y)\in X\times Y$.

\medskip

\noindent {\bf Bernoulli shifts.} Let $\Gamma \times T\ra T$, $(\gamma ,t)\mapsto \gamma \cdot t$ be an action of $\Gamma$ on a countable set $T$. The \emph{generalized Bernoulli shift} corresponding to this action is the system $\bm{s}_{\Gamma ,T} =\Gamma \cc ^s ([0,1]^T, \lambda ^T )$, where $\lambda$ is Lebesgue measure and where the action $s$ is given by $(\gamma ^sf)(t) = f(\gamma ^{-1}\cdot t)$ for $\gamma \in \Gamma$, $f\in [0,1]^T$, $t\in T$. We write $\bm{s}_\Gamma$ for $\bm{s}_{\Gamma ,\Gamma}$ when the action of $\Gamma$ on itself is by left translation. The system $\bm{s}_\Gamma$ is called the \emph{Bernoulli shift} of $\Gamma$.

\medskip

\noindent {\bf The trivial system.} We call $\bm{a}=\Gamma \cc ^a (X,\mu )$ \emph{trivial} if $\mu$ is a point mass.  Otherwise, $\bm{a}$ is called \emph{non-trivial}. Up to isomorphism, each group $\Gamma$ has a unique trivial measure preserving action, which we denote by $\bm{i}_\Gamma$ or simply $\bm{i}$ when $\Gamma$ is clear.

\medskip

\noindent {\bf Identity systems.} Let $\bm{\iota}_{\Gamma ,\mu} = \Gamma \cc ^{\iota} (X,\mu )$ denote the \emph{identity system} of $\Gamma$ on $(X,\mu )$ given by $\gamma ^{\iota} = \mbox{id}_X$ for all $\gamma \in \Gamma$. We write $\bm{\iota}_{\mu}$ when $\Gamma$ is clear. Thus if $\mu$ is a point mass then $\bm{\iota}_\mu \cong \bm{i}$.

\medskip

\noindent {\bf Strong ergodicity.} A system $\bm{a}$ is called \emph{strongly ergodic} if it is ergodic and does not weakly contain the identity system $\bm{\iota}_{\Gamma ,\lambda}$ on $([0,1],\lambda )$.

\medskip

\noindent {\bf Fixed point sets and free actions.} For a subset $C\subseteq \Gamma$ we let
\[
\mbox{Fix}^b(C) = \{ y\in Y \csuchthat \forall \gamma \in C \ \, \gamma ^by =y \} .
\]
We write $\mbox{Fix} ^b(\gamma )$ for $\mbox{Fix} ^b(\{ \gamma \} )$. The \emph{kernel} of the system $\bm{b}$ is the set $\mbox{ker}(\bm{b}) = \{ \gamma \in \Gamma \csuchthat \nu (\mbox{Fix}^b(\gamma ) ) = 1 \}$. It is clear that $\mbox{ker}(\bm{b})$ is a normal subgroup of $\Gamma$. The system $\bm{b}$ is called \emph{faithful} if $\mbox{ker}(\bm{b}) = \{ e \}$, i.e., $\nu (\mbox{Fix}^b(\gamma ))<1$ for each $\gamma \in \Gamma \setminus \{ e\}$. The system $\bm{b}$ is called \emph{(essentially) free} if the stabilizer of $\nu$-almost every point is trivial, i.e., $\nu (\mbox{Fix}^b(\gamma )) = 0$ for each $\gamma \in \Gamma \setminus \{ e \}$.
%
\begin{comment}
Note that for any $\gamma ,\delta \in \Gamma$, $C\subseteq \Gamma$, and $\bm{b}=\Gamma \cc ^b(Y,\nu )$ we have
\begin{align*}
\nu (\mbox{Fix}^b(C)) &= \theta _{\bm{b}} ( \{ H \csuchthat C \subseteq H \} ) \\
\delta ^b\cdot \mbox{Fix}^b(\gamma ) &= \mbox{Fix}^b(\delta \gamma \delta ^{-1})\\
\mbox{Fix}^b(\gamma )&= \mbox{Fix}^b(\gamma ^{-1})\\
\mbox{Fix}^b(\gamma )\cap \mbox{Fix}^b(\delta ) &\subseteq \mbox{Fix}^b(\gamma\delta )
\end{align*}
\end{comment}

\subsection{Invariant random subgroups}\label{subsec:IRS}

\noindent {\bf The space of subgroups.} We let $\mbox{Sub}_\Gamma \subseteq 2^\Gamma$ denote the compact space of all subgroup of $\Gamma$ and we let $c:\Gamma \times \mbox{Sub}_\Gamma \ra \mbox{Sub}_\Gamma$ be the continuous action of $\Gamma$ on $\mbox{Sub}_\Gamma$ by conjugation.

\medskip

\noindent {\bf Invariant random subgroups.} An \emph{invariant random subgroup} (IRS) of $\Gamma$ is a conjugation-invariant Borel probability measures on $\mbox{Sub}_\Gamma$. The point mass $\updelta _N$ at a normal subgroup $N$ of $\Gamma$ is an example of an invariant random subgroup. Let $\mbox{IRS}_\Gamma$ denote the space of invariant random subgroups of $\Gamma$. We associate to each $\theta \in \mbox{IRS}_{\Gamma}$ the measure preserving action $\Gamma \cc ^c (\mbox{Sub}_\Gamma , \theta )$. We also denote this system by $\bm{\theta}$.

\medskip

\noindent {\bf Stabilizer distributions.} Each measure preserving action $\bm{b} = \Gamma \cc ^b (Y,\nu )$ of $\Gamma$ gives rise to and invariant random subgroup $\theta _{\bm{b}}$ of $\Gamma$ as follows. The \emph{stabilizer} of a point $y\in Y$ under the action $b$ is the subgroup $\Gamma _y$ of $\Gamma$ defined by
\[
\Gamma _y = \{ \gamma \in \Gamma \csuchthat \gamma ^b y = y \} .
\]
The group $\Gamma _y$ of course depends on the action $b$. The \emph{stabilizer map} associated to $b$ is the map $\mbox{stab}_b : Y \ra \mbox{Sub}_\Gamma$ given by $\mbox{stab}_b(y) = \Gamma _y$. The \emph{stabilizer distribution} of $\bm{b}$, which we denote by $\theta _{\bm{b}}$ or $\mbox{type}({\bm{b}})$, is the measure $(\mbox{stab}_b)_*\nu$ on $\mbox{Sub}_\Gamma$. It is clear that $\theta _{\bm{b}}$ is an invariant random subgroup of $\Gamma$.  In \cite{AGV12} it is shown that for any invariant random subgroup $\theta$ of $\Gamma$, there exists a m.p.\ action $\bm{b}$ of $\Gamma$ such that $\theta _{\bm{b}} = \theta$. Moreover, if $\bm{\theta}$ is ergodic then $\bm{b}$ can be taken to be ergodic as well. See \cite{CP12}.

\medskip

\noindent {\bf Group theoretic properties of invariant random subgroups.} Let $\theta$ be an invariant random subgroup of $\Gamma$. We say that a given property $\Es{P}$ of subgroups of $\Gamma$ holds for $\theta$ if $\Es{P}$ holds almost everywhere. For example, $\theta$ is called \emph{amenable} (or \emph{infinite index}) if $\theta$ concentrates on the amenable (respectively, infinite index) subgroups of $\Gamma$.

\medskip

\noindent {\bf The trivial IRS.} By the \emph{trivial} invariant random subgroup we mean the point mass at the trivial subgroup $\{ e \}$ of $\Gamma$. We write $\updelta _e$ instead of $\updelta _{\{ e \}}$ for the trivial invariant random subgroup. An invariant random subgroup not equal to $\updelta _e$ is called \emph{non-trivial}.

\begin{remark}
We will often abuse terminology and confuse an invariant random subgroup $\theta$ with the measure preserving action $\bm{\theta} = \Gamma \cc ^c (\mbox{Sub}_\Gamma ,\theta )$ it defines, stating, for example, that $\theta$ is ergodic or is weakly contained in $\bm{s}_\Gamma$ to mean that $\bm{\theta}$ is ergodic or is weakly contained in $\bm{s}_\Gamma$. When there is a potential for ambiguity we will make sure to distinguish between an invariant random subgroup and the measure preserving system to which it gives rise. We emphasise that "$\theta$ is non-trivial" will always mean that $\theta$ is not equal to the trivial IRS $\updelta _e$, whereas "$\bm{\theta}$ is non-trivial" will always mean that $\theta$ is not a point mass (at \emph{any} subgroup).
%
%The underlying point here is category theoretic. An invariant random subgroup $\theta$ is an object in the category $\bm{\mbox{IRS}_\Gamma}$ of invariant random subgroups of $\Gamma$, while the system $\bm{\theta}$ is as an object in the category $\bm{\mbox{MPA}_\Gamma}$ of measure preserving actions of $\Gamma$. When dealing with invariant random subgroups \emph{qua} invariant random subgroups (i.e., as objects in $\bm{\mbox{IRS}_\Gamma}$) the notion of equivalence employed in this paper is equality: two invariant random subgroups $\theta _1$ and $\theta _2$ are considered the same if and only if they are equal. On the other hand, the notion of equivalence between measure preserving actions most relevant to this paper is isomorphism; we will often not distinguish between two measure preserving actions $\bm{a}$ and $\bm{b}$ which are isomorphic.
\end{remark}

\section{Shift-minimality}\label{sec:smdef}

\subsection{Seven characterizations of shift-minimality}

It will be useful to record here the main theorem of \cite{AW11} which was already mentioned several times in the introduction.

\begin{theorem}[\cite{AW11}]\label{thm:AW}
Let $\Gamma$ be a countably infinite group. Then the Bernoulli shift $\bm{s}_\Gamma$ is weakly contained in every free measure preserving action of $\Gamma$.
\end{theorem}

We let $\Aut (X,\mu )$ denote the Polish group of measure preserving transformations of $(X,\mu )$, and we let $A(\Gamma ,X,\mu )$ denote the Polish space of measure preserving actions of $\Gamma$ on the measure space $(X,\mu )$. See \cite{Ke10} for information on these two spaces.
In the following proposition, let $[\bm{a}]$ denote the weak equivalence class of a measure preserving action $\bm{a}$ of $\Gamma$. Denote by $\bm{s}_{\Gamma ,2}$ the full $2$-shift of $\Gamma$, i.e., the shift action of $\Gamma$ on $( 2^\Gamma , \rho ^\Gamma )$ where we identify $2$ with $\{ 0,1\}$ and where $\rho ( \{ 0 \} ) = \rho (\{ 1 \} ) = 1/2$.
%
%!!!!!!!!!!!!!!!!!!!!!!!!!!!Add the condition below!!!!!!!!!!!!!!!!!!!!!!!!!!!!!!!!!!!!!!!
%
%$\Gamma$ is shift-minimal iff every non-trivial $\bm{a}$ weakly contained in $\bm{s}_{\Gamma ,2} := \Gamma \cc ^s (2^\Gamma ,\nu ^\Gamma )$ is weakly equivalent to $\bm{s}_{\Gamma ,2}$. This is nice since no need to distinguish finite and infinite case.
%
%
\begin{proposition}\label{prop:minimal}
Let $\Gamma$ be a countable group and let $(X,\mu )$ be a standard non-atomic probability space. Then the following are equivalent.
\begin{enumerate}
\item $\Gamma$ is shift-minimal, i.e., every non-trivial m.p.\ action weakly contained in $\bm{s}_\Gamma$ is free.
\item Every non-trivial m.p.\ action weakly contained in $\bm{s}_{\Gamma ,2}$ is free.
\item Among non-trivial m.p.\ actions of $\Gamma$, $[\bm{s}_{\Gamma ,2}]$ is minimal with respect to weak containment.
\item Either $\Gamma = \{ e \}$ or, among non-trivial m.p.\ actions of $\Gamma$, $[\bm{s}_\Gamma ]$ is minimal with respect to weak containment.
\item %Within the set of weak equivalence classes of non-atomic m.p.\ actions,
    Among non-atomic m.p.\ actions of $\Gamma$, $[\bm{s}_\Gamma ]$ is minimal with respect to weak containment.
\item The conjugation action of the Polish group $\mbox{\emph{Aut}}(X,\mu  )$ on the Polish space $A_{\bm{s}} (\Gamma ,X,\mu  ) = \{ \bm{a} \in A(\Gamma ,X,\mu  ) \csuchthat \bm{a} \prec \bm{s}_\Gamma \}$ is minimal, i.e., every orbit is dense.
\item For some (equivalently: every) non-principal ultrafilter $\mc{U}$ on the the natural numbers $\N$, every non-trivial factor of the ultrapower $(\bm{s}_\Gamma )_{\mc{U}}$ is free.
\end{enumerate}
\end{proposition}

\begin{proof}
The equivalence (7)$\IFF$(1) follows from \cite[Theorem 1]{CKT-D12}. For the remaining equivalences, first note that if $\Gamma$ is a finite group then $\bm{s}_\Gamma$ factors onto $\bm{\iota}_{\mu}$, so if $\Gamma\neq \{ e \}$ then $\Gamma$ does not satisfy (1), (4), (5) or (6). In addition, for $\Gamma\neq \{ e \}$ finite, $\bm{s}_{\Gamma ,2}$ factors onto a non-trivial identity system, which shows that $\Gamma$ does not satisfy (2) or (3) either. This shows that the trivial group $\Gamma = \{ e \}$ is the only finite group that satisfies any of the properties (1)-(6), and it is clear the trivial group satisfies all of these properties. We may therefore assume for the rest of the proof that $\Gamma$ is infinite.

(1)$\Ra$(2): This implication is clear since $\bm{s}_{\Gamma ,2}$ is a factor of $\bm{s}_\Gamma$.

(2)$\Ra$(3): Suppose that (2) holds. By hypothesis any $\bm{a}$ weakly contained in $\bm{s}_{\Gamma ,2}$ is free and thus weakly contains $\bm{s}_{\Gamma}$ by Theorem \ref{thm:AW}. (3) follows since $\bm{s}_{\Gamma 2}$ is a factor of $\bm{s}_\Gamma$.

(3)$\Ra$(4): Since we are assuming $\Gamma$ is infinite, Theorem \ref{thm:AW} implies $[\bm{s}_{\Gamma}]= [\bm{s}_{\Gamma ,2}]$, and this implication follows. %Suppose $\Gamma$ is shift-minimal. If $\Gamma$ is finite then $\bm{s}_\Gamma$ factors onto the identity system $\bm{\iota}_\mu$ where $\mu$ is non-atomic, so $\bm{\iota}_\mu$ is free by shift-minimality, and thus $\Gamma = \{ e \}$. If $\Gamma$ is infinite then every free action of $\Gamma$ is non-atomic, so $[\bm{s}_\Gamma ]$ is minimal by the Ab\'{e}rt-Weiss characterization of free non-atomic m.p.\ actions as those that weakly contain $\bm{s}_\Gamma$.
%
%(4)$\Ra$(1): Suppose (4) holds. Then every non-trivial m.p.\ action weakly contained in $\bm{s}_\Gamma$ is in fact weakly equivalent to $\bm{s}_\Gamma$ and is thus free, so $(1)$ holds.
(4)$\Ra$(5) is clear.

(5)$\Ra$(6): Suppose (5) holds. By \cite[Proposition 10.1]{Ke10} the $\Aut (X,\mu )$-orbit closure of any $\bm{a}\in  A(\Gamma ,X,\mu )$ is equal to $\{ \bm{b}\in A (\Gamma ,X,\mu )\csuchthat \bm{b}\prec \bm{a} \}$. Thus, if $\bm{a}$ is weakly equivalent to $\bm{s}_\Gamma$, then the orbit of $\bm{a}$ is dense in $A_{\bm{s}}(\Gamma ,X,\mu )$. Since $[\bm{s}_\Gamma ]$ is minimal with respect to weak containment, every element of $A_{\bm{s}}(\Gamma ,X,\mu )$ is weakly equivalent to $\bm{s}_\Gamma$, so has dense orbit in $A_{\bm{s}}(\Gamma ,X,\mu )$. Thus, the action $\Aut (X,\mu ) \cc A_{\bm{s}}(\Gamma ,X,\mu )$ is minimal.

(6)$\Ra$(1): Suppose that every $\bm{a}\in A_{\bm{s}}(\Gamma ,X,\mu )$ has dense orbit. If $\bm{\iota}_\mu \in A_{\bm{s}}(\Gamma ,X,\mu )$ then, since $\bm{\iota}_\mu$ is a fixed point for the $\Aut (X,\mu )$ action, $\bm{\iota}_\mu = \bm{s}_\Gamma$ and thus $\Gamma = \{ e \}$. Otherwise, if $\bm{\iota}_\mu \not\prec \bm{s}_\Gamma$ then the system $\bm{s}_\Gamma$ is strongly ergodic and the group $\Gamma$ is therefore non-amenable. Let $\bm{b} =\Gamma \cc ^b(Y,\nu )$ be any non-trivial m.p.\ action of $\Gamma$ weakly contained in $\bm{s}_\Gamma $. Then $\bm{b}\times \bm{b}$ is weakly contained in $\bm{s}_\Gamma \times \bm{s}_\Gamma \cong \bm{s}_\Gamma$ and therefore $\bm{b}\times \bm{b}$ is strongly ergodic since strong ergodicity is downward closed under weak containment (see e.g., \cite[Proposition 5.6]{CKT-D12}). In particular $\bm{b}\times \bm{b}$ is ergodic and it follows that the probability space $(Y,\nu )$ is non-atomic. The action $\bm{b}$ is then isomorphic to some action $\bm{a}$ on the non-atomic space $(X,\mu )$, and $\bm{a} \in A_{\bm{s}}(\Gamma ,X,\mu )$ since $\bm{b}\prec \bm{s}_\Gamma$. By hypothesis $\bm{a}$ has dense orbit in $A_{\bm{s}}(\Gamma ,X,\mu )$ so that $\bm{s}_\Gamma \sim \bm{a}$ by \cite[Proposition 10.1]{Ke10} and hence $\bm{a}$ is free, and thus $\bm{b}$ is free as well. %, $\bm{s}_\Gamma \sim \bm{b}$. This shows that $[\bm{s}_\Gamma]$ is $\prec$-minimal among non-trivial m.p.\ actions.
\qedhere
\end{proof}

Two more characterizations of shift-minimality are given in terms of amenable invariant random subgroups in Theorem \ref{thm:smequiv} below.
\subsection{NA-ergodicity}\label{sec:NAerg}

\begin{definition}
Let $\bm{a}$ be a measure preserving action of a countable group $\Gamma$. We say that $\bm{a}$ is \emph{NA-ergodic} if the restriction of $\bm{a}$ to every non-amenable subgroup of $\Gamma$ is ergodic. We say that $\bm{a}$ is \emph{strongly NA-ergodic} if the restriction of $\bm{a}$ to every non-amenable subgroup of $\Gamma$ is strongly ergodic.
\end{definition}

\begin{example}\label{ex:NAerg}
The central example of an NA-ergodic (and in fact, strongly NA-ergodic) action is the Bernoulli shift action $\bm{s}_\Gamma$; if $H\leq \Gamma$ is non-amenable then $\bm{s}_\Gamma |H \cong \bm{s}_H$ is strongly ergodic. More generally, if $\Gamma$ acts on a countable set $T$ and the stabilizer of every $t\in T$ is amenable then the generalized Bernoulli shift $\bm{s}_{T} = \Gamma \cc ^{s_T} ([0,1]^T , \lambda ^T )$ is strongly NA-ergodic (see e.g., \cite{KT09}).
\end{example}

\begin{example}\label{ex:NAergSL2}
The action $\mbox{SL}_2(\Z ) \cc (\T ^2 ,\lambda ^2 )$ by matrix multiplication, where $\lambda ^2$ is Haar measure on $\T ^2$, is another example of a strongly NA-ergodic action. A proof of this is given in \cite[5.(B)]{Ke07}.
\end{example}

\begin{example}\label{ex:NAerg2}
I would like to thank L. Bowen for bringing my attention to this example. Let $\Gamma$ be a countable group and let $f$ be an element of the integral group ring $\Z \Gamma$. The left translation action of $\Gamma$ on the discrete abelian group $\Z\Gamma /{\Z \Gamma} f$ is by automorphisms, and this induces an action of $\Gamma$ by automorphisms on the dual group $\widehat{\Z \Gamma / {\Z \Gamma} f}$, which is a compact metrizable abelian group so that this action preserves normalized Haar measure $\mu _f$. Bowen has shown that if the function $f$ has an inverse in $\ell ^1 (\Gamma )$ then the system $\Gamma \cc ( \widehat{\Z \Gamma /{\Z \Gamma} f}  , \mu _f )$ is weakly contained in $\bm{s}_\Gamma$ and is therefore strongly NA-ergodic by Proposition \ref{prop:NAerg} (\cite[\S 5]{Bo11a}; note that the hypothesis that $\Gamma$ is residually finite is not used in that section so that this holds for arbitrary countable groups $\Gamma$).
\end{example}

\begin{remark} \label{rem:temp}
The actions from Examples \ref{ex:NAerg}, \ref{ex:NAergSL2}, and \ref{ex:NAerg2} share a common property: they are \emph{tempered} in the sense of \cite{Ke07}. A measure preserving action $\bm{a}=\Gamma \cc ^a (X,\mu )$ is called \emph{tempered} if the Koopman representation $\kappa ^{\bm{a}} _0$ on $L^2_0 (X,\mu ) = L^2(X,\mu ) \ominus \C 1_X$ is weakly contained in the regular representation $\lambda _\Gamma$ of $\Gamma$. Any tempered action $\bm{a}$ of a non-amenable group $\Gamma$ has \emph{stable spectral gap} in the sense of \cite{Pop08} (this means $\kappa ^{\bm{a}}_0\otimes \kappa ^{\bm{a}}_0$ does not weakly contain the trivial representation), and this implies in turn that the product action $\bm{a}\times \bm{b}$ is strongly ergodic relative to $\bm{b}$ for every measure preserving action $\bm{b}$ of $\Gamma$ (see \cite{Io10}). In particular (taking $\bm{b} = \bm{i}_\Gamma$) a tempered action $\bm{a}$ of a non-amenable group is itself strongly ergodic. Since the restriction of a tempered action to a subgroup is still tempered it follows that every tempered action is strongly NA-ergodic. In \cite{Ke07} it is shown that the converse holds for any action on a compact Polish group $G$ by automorphisms (such an action necessarily preserves Haar measure $\mu _G$):

\begin{theorem}[Theorem 4.6 of \cite{Ke07}]\label{thm:Kec05} Let $\Gamma$ be a countably infinite group acting by automorphisms on a compact Polish group $G$. Let $\hat{G}$ denote the (countable) set of all isomorphism classes of irreducible unitary representations of $G$ and let $\hat{G}_0 = \hat{G}\setminus \{ \hat{1}_G \}$. Then the following are equivalent:
\begin{enumerate}
\item The action $\Gamma \cc (G, \mu _G)$ is tempered;
\item Every stabilizer of the associated action of $\Gamma$ on $\hat{G}_0$ is amenable.
\item The action $\Gamma \cc (G, \mu _G)$ is NA-ergodic.
\item The action $\Gamma \cc (G, \mu _G)$ is strongly NA-ergodic.
\end{enumerate}
\end{theorem}

Condition (2) of Theorem \ref{thm:Kec05} should be compared with part (ii) of Lemma \ref{lem:am} below, although Lemma \ref{lem:am} deals with general NA-ergodic actions. It follows from \cite[Proposition 10.5]{Ke10} that any measure preserving action weakly contained in $\bm{s}_\Gamma$ is tempered. I do not know however whether the converse holds, although Example \ref{ex:NAerg2} and Theorem \ref{thm:Kec05} suggest that this may be the case for actions by automorphisms on compact Polish groups.

\begin{question}\label{Q:CompactAut}
Let $\Gamma$ be a countable group acting by automorphisms on a compact Polish group $G$ and assume the action is tempered. Does it follow that the action is weakly contained in $\bm{s}_\Gamma$? As a special case, is it true that the action $\mbox{SL}_2(\Z ) \cc (\T ^2 ,\lambda ^2 )$ is weakly contained in $\bm{s}_{\mbox{\scriptsize{SL}}_2(\Z )}$?
\end{question}
\end{remark}

We now establish some properties of general NA-ergodic actions.

\begin{proposition}\label{prop:NAerg} Any factor of an NA-ergodic action is NA-ergodic. Any action weakly contained in a strongly NA-ergodic action is strongly NA-ergodic.
\end{proposition}

\begin{proof}
The first statement is clear and the second is a consequence of strong ergodicity being downward closed under weak containment (see \cite[Proposition 5.6]{CKT-D12}).
\end{proof}

Part (ii) of the following lemma is one of the key facts about NA-ergodicity. %The following lemma is %is an analogue of \cite[Theorem 7.4]{CP12} and \cite[Lemma 3.1]{T-D12b}.
\begin{lemma}\label{lem:am}
Let $\bm{b} = \Gamma \cc ^b (Y,\nu )$ be any non-trivial NA-ergodic action of a countable group $\Gamma$.
\begin{enumerate}
\item[(i)] Suppose that $C \subseteq \Gamma$ is a subset of $\Gamma$ such that $\nu ( \{ y\in Y\csuchthat C\subseteq \Gamma _y \} ) > 0 $. Then the subgroup $\langle C \rangle$ generated by $C$ is amenable.
\item[(ii)] The stabilizer $\Gamma _y$ of $\nu$-almost every $y\in Y$ is amenable.
\end{enumerate}
\end{lemma}

\begin{proof}%[Proof of Lemma \ref{lem:am}]
We begin with part (i). The hypothesis tells us that $\nu (\mbox{Fix}^b (C)) >0$. Since $\nu$ is not a point mass there is some $B\subseteq \mbox{Fix}^b(C)$ with $0<\nu (B) <1$. Then $B$ witnesses that $\bm{b}\resto \langle C \rangle$ is not ergodic, so $\langle C\rangle$ is amenable by NA-ergodicity of $\bm{b}$.

For (ii), let $\Es{F}$ denote the collection of finite subsets $F$ of $\Gamma$ such that $\langle F \rangle$ is non-amenable and let $\mbox{NA} = \{ y\in Y\csuchthat \Gamma _y \mbox{ is non-amenable} \}$. Then
\[
\mbox{NA} = \bigcup _{F\in \Es{F}} \{ y\in Y \csuchthat F\subseteq \Gamma _y \}  .
\]
By part (i), $\nu (\{ y\in Y \csuchthat F\subseteq \Gamma _y \} ) = 0$ for each $F \in \Es{F}$. Since $\Es{F}$ is countable it follows that $\nu ( \mbox{NA} ) = 0$.
\qedhere[Lemma]
\end{proof}

The function $N:\mbox{Sub}_\Gamma \ra \mbox{Sub}_\Gamma$ sending a subgroup $H\leq \Gamma$ to its normalizer $N(H)$ in $\Gamma$ is equivariant for the conjugation action $\Gamma \cc ^c \mbox{Sub}_\Gamma$. In \cite[\S 2.4]{Ve12} Vershik examines the following transfinite iterations of this function.

\begin{definition}
Define $N^\alpha :\mbox{Sub}_\Gamma \ra \mbox{Sub}_\Gamma$ by transfinite induction on ordinals $\alpha$ as follows.
\begin{align*}
N^0(H) & =H,  \\
N^{\alpha +1}(H) &= N(N^{\alpha}(H)) \, \mbox{ is the normalizer of } N^\alpha (H)  \\
N^\lambda (H) &= \bigcup _{\alpha <\lambda} N^\alpha (H) \, \mbox{ when }\lambda\mbox{ is a limit ordinal.}
\end{align*}
\end{definition}
%Note that the maps $N^\alpha$ are not generally continuous. For example, let $\Gamma$ be the free group on the two generators $a$, $b$ and let $H_n$ be the cyclic group generated by $a^n b a^{-n}$. Then $H_n\cap H_m = \{ e \}$ for all $n\neq m$ so that $H_n\ra \{ e \}$ as $n\ra \infty$. However $N(H_n) = H_n \ra _{n\ra \infty} \{ e \} \neq \Gamma = N( \{ e \} )$.

Each $N^\alpha$ is equivariant with respect to conjugation. %This is clear at successor ordinals. Assume now that this is true for all $\alpha <\lambda$ where $\lambda$ is a limit ordinal. Then
%\begin{align*}
%N^\lambda (\gamma H \gamma ^{-1} )
%&= \bigcup _{\alpha <\lambda} N^\alpha (\gamma H\gamma ^{-1})
%= \bigcup _{\alpha <\lambda}\gamma N^\alpha (H)\gamma ^{-1} = \gamma N^\lambda (H)\gamma ^{-1}.
%\end{align*}
%
For each $H$ the sets $H, N(H),\dots , N^\alpha (H), N^{\alpha +1}(H),\dots$ form an increasing ordinal-indexed sequence of subsets of $\Gamma$. The least ordinal $\alpha _H$ such that $N^{\alpha _H +1}(H)=N^{\alpha _H}(H)$ is therefore countable. If $\theta \in \mbox{IRS}_\Gamma$ then we let $\theta ^\alpha = (N^\alpha )_*\theta$ for each countable ordinal $\alpha <\omega _1$. The net $\{ \theta ^\alpha \} _{\alpha <\omega _1}$ is increasing in the sense of \cite[\S 3.5]{CP12} (see also the paragraphs preceding Theorem \ref{thm:IRPleq} below), so by \cite[Theorem 3.12]{CP12} there is a weak${}^*$-limit $\theta ^\infty$ such that $\theta ^\alpha \leq \theta ^\infty$ for all $\alpha$. Since $\mbox{IRS}_\Gamma$ is a second-countable topological space there is a countable ordinal $\alpha$ such that $\theta ^\beta = \theta ^\infty$ for all $\beta \geq \alpha$. Thus $N_*\theta ^\infty = \theta ^\infty$, and it follows from \cite[Proposition 4]{Ve12} that $\theta ^\infty$ concentrates on the self-normalizing subgroups of $\Gamma$.

\begin{theorem}\label{thm:NAerg}
Let $\bm{a}=  \Gamma \cc ^a (X, \nu )$ be a non-trivial measure preserving action of the countable group $\Gamma$. Suppose that $\bm{a}$ is NA-ergodic. Then the stabilizer $\Gamma _x$ of $\mu$-almost every $x\in X$ is amenable. In addition, at least one of the following is true:
\begin{itemize}
\item[(1)] There exists a normal amenable subgroup $N\triangleleft \Gamma$ such that the stabilizer of $\mu$-almost every $x\in X$ is contained in $N$.
\item[(2)] $\theta _{\bm{a}}^{\infty}$ is a non-atomic, self-normalizing, infinitely generated amenable invariant random subgroup, where $\theta _{\bm{a}}$ denotes the stabilizer distribution of $\bm{a}$.
\end{itemize}
\end{theorem}

\begin{proof}%[Proof of Theorem \ref{thm:NAerg}]
%Let $\bm{a}$ be a non-trivial NA-ergodic action as in the statement of Theorem \ref{thm:NAerg}, and
Let $\theta = \theta _{\bm{a}}$. It is enough to show that either (1) or (2) is true. We may assume that $\Gamma$ is non-amenable. There are two cases to consider.

{\bf Case 1:} There is some ordinal $\alpha$ such that the measure $\theta ^{\alpha}$ has an atom.  Let $\alpha _0$ be the least such ordinal. Then $\bm{\theta }^{\alpha _0}$ is NA-ergodic, being a factor of $\bm{a}$, and thus the restriction of $\bm{\theta}^{\alpha _0}$ to every finite index subgroup of $\Gamma$ is ergodic since $\Gamma$ is non-amenable.  Thus, $\theta ^{\alpha _0}$ having an atom implies that it is a point mass, so let $N\leq \Gamma$ be such that $\theta ^{\alpha _0} = \updelta _N$. Then $N$ is a normal subgroup of $\Gamma$ and we show that $N$ is amenable so that alternative (1) holds in this case. %If $\alpha _0 = 0$ this means $\theta _{\bm{a}}(\{ H\csuchthat N\subseteq H \} ) >0$ and so part (i) implies $N$ is amenable.
By definition of $\alpha _0$, $\bm{a}$ and each $\bm{\theta} ^{\alpha}$ for $\alpha <\alpha _0$ are non-trivial NA-ergodic actions. Lemma \ref{lem:am} then implies that the invariant random subgroups $\mbox{type}(\bm{a}) = \theta ^0$ and $\mbox{type}(\bm{\theta}^{\alpha}) = \theta ^{\alpha +1}$, for $\alpha <\alpha _0$, all concentrate on the amenable subgroups of $\Gamma$.  If $\alpha _0 =0$ or if $\alpha _0$ is a successor ordinal then we see immediately that $N$ is amenable. If $\alpha _0$ is a limit ordinal then $N$ is an increasing union of amenable groups and so is amenable in this case as well.

{\bf Case 2:} The other possibility is that $\theta ^{\infty}$ has no atoms. Thus $\bm{\theta }^\infty$ is a non-trivial NA-ergodic action with $\mbox{type}(\bm{\theta}^\infty )= N_*\theta ^\infty = \theta ^\infty$. This implies that $\theta ^\infty$ is amenable by Lemma \ref{lem:am}. Since $\theta ^\infty$ is non-atomic and there are only countably many finitely generated subgroups of $\Gamma$, $\theta ^\infty$ must concentrate on the infinitely-generated subgroups. This shows that (2) holds.
%$\theta ^{\alpha}$ is not a point mass for any $\alpha$. In this case $\bm{\theta} ^{\infty}$ is a non-trivial factor of $\bm{a}$, so by Lemma \ref{lem:am}, $\theta ^{\infty}$ concentrates on the amenable subgroups of $\Gamma$.
%{\bf Subcase 2.a:} $\theta ^{\infty}$ has an atom, say $\theta ^{\infty}( \{ H_0 \} )>0$. Then $H_0$ is amenable and self-normalizing. We have
%\[
%1\geq \theta ^{\infty}( \{ \gamma H_0\gamma ^{-1} \csuchthat \gamma \in \Gamma \} ) = \theta ^{\infty} (\{ H_0 \} )[\Gamma :N(H_0)]=\theta ^{\infty} (\{ H_0 \} )[\Gamma :H_0],
%\]
%which shows that $[\Gamma :H_0 ]<\infty$. It follows that $\Gamma$ is amenable and alternative (1) holds again.
%
%{\bf Subcase 2.b:}
%
%Thus, if $\theta ^{\infty} ( \{ H\leq \Gamma \csuchthat H \mbox{ is co-amenable in }\Gamma \} )> 0$ then $\Gamma$ itself is amenable since it contains an amenable, co-amenable subgroup, and so (1) holds again. On the other hand, if
\end{proof}

\subsection{Amenable invariant random subgroups}\label{sec:airs}

We record a corollary of Theorem \ref{thm:NAerg} which will be used in the proof of our final characterization of shift-minimality.

\begin{corollary}\label{cor:notsm}
Any group $\Gamma$ that is not shift-minimal either has a non-trivial normal amenable subgroup $N$, or has a non-atomic, self-normalizing, infinitely-generated, amenable invariant random subgroup $\theta$ such that the action $\bm{\theta} = \Gamma \cc ^c (\mbox{\emph{Sub}}_\Gamma ,\theta )$ is weakly contained in $\bm{s}_\Gamma$. %Moreover, in the latter case, we can find such a $\theta$ with $\bm{\theta}\prec \bm{s}_\Gamma$.
\end{corollary}

\begin{proof}
Let $\Gamma$ be a group that is not shift-minimal so that there exists some non-trivial $\bm{a}$ weakly contained in $\bm{s}_{\Gamma}$ which is not free. %If $\Gamma$ is amenable then we are done by taking $N=\Gamma$ (note that $\Gamma \neq \{ e \}$ by assumption), so we may assume that $\Gamma$ is non-amenable .
The action $\bm{a}$ is strongly NA-ergodic by \ref{ex:NAerg} and \ref{prop:NAerg}, so $\bm{a}$ satisfies the hypotheses of Theorem \ref{thm:NAerg}. If (1) of Theorem \ref{thm:NAerg} holds, say with witnessing normal amenable subgroup $N\leq \Gamma$, then $N$ is non-trivial since $\bm{a}$ is non-free. If alternative (2) of Theorem \ref{thm:NAerg} holds then taking $\theta = \theta _{\bm{a}}^\infty$ works.% $\bm{\theta} _{\bm{a}}^\infty$ is weakly contained in $\bm{s}_\Gamma$ since it is a factor of $\bm{a}$. This implies $\bm{\theta }_{\bm{a}}^\infty \times \bm{\theta}_{\bm{a}}^\infty$ is also weakly contained in $\bm{s}_\Gamma$, so, as $\bm{s}_\Gamma$ is strongly ergodic, $\bm{\theta }_{\bm{a}}^\infty \times \bm{\theta}_{\bm{a}}^\infty$ is ergodic, i.e., $\bm{\theta}_{\bm{a}}^\infty$ is weakly mixing.
\end{proof}

We also need

\begin{proposition}\label{prop:noamensub}
If $\Gamma$ is shift-minimal then $\Gamma$ has no non-trivial normal amenable subgroups.
\end{proposition}

\begin{proof}
Suppose that $\Gamma$ has a non-trivial normal amenable subgroup $N$. Amenability implies that $\bm{\iota}_N\prec \bm{s}_N$. Then since co-inducing preserves weak containment we have
\[
\bm{s}_{\Gamma ,\Gamma /N}\prec \mbox{CInd}_N^\Gamma (\bm{\iota}_N) \prec \mbox{CInd}_N^\Gamma (\bm{s}_N)\cong \bm{s}_\Gamma
\]
which shows that $\bm{s}_{\Gamma ,\Gamma /N}\prec \bm{s}_\Gamma$. The action $\bm{s}_{\Gamma ,\Gamma /N}$ is not free since $N\subseteq \mbox{ker}(\bm{s}_{\Gamma ,\Gamma /N})$. This shows that $\Gamma$ is not shift-minimal.
\end{proof}

The following immediately yields Theorem \ref{thm:smiff} from the introduction.

\begin{theorem}\label{thm:smequiv}
The following are equivalent for a countable group $\Gamma$:
\begin{enumerate}
\item[(1)] $\Gamma$ is not shift-minimal.
\item[(2)] There exists a non-trivial amenable invariant random subgroup $\theta$ of $\Gamma$ that is weakly contained in $\bm{s}_\Gamma$.
\item[(3)] Either $\mbox{\emph{AR}}_\Gamma$ is finite and non-trivial, or there exists an infinite amenable invariant random subgroup $\theta$ of $\Gamma$ that is weakly contained in $\bm{s}_\Gamma$.
\end{enumerate}
\end{theorem}

\begin{proof}
(1)$\Ra$(3): Suppose that $\Gamma$ is not shift-minimal. If the second alternative of Corollary \ref{cor:notsm} holds then we are done. Otherwise, the first alternative holds and so $\mbox{AR}_\Gamma$ is non-trivial. If $\mbox{AR}_\Gamma$ is finite then (3) is immediate, and if $\mbox{AR}_\Gamma$ is infinite then the point-mass at $\mbox{AR}_\Gamma$ shows that (3) holds.

(3)$\Ra$(2) is clear. Now let $\theta$ be as in (2) and we will show that $\Gamma$ is not shift-minimal. If $\theta$ is a point mass, say at $H\in \mbox{Sub}(\Gamma )$, then $H$ is normal and by hypothesis $H$ is non-trivial and amenable so (1) then follows from Proposition \ref{prop:noamensub}. If $\theta$ is not a point mass then $\Gamma \cc ^c (\mbox{Sub}_\Gamma , \theta )$ is a non-trivial and non-free measure preserving action of $\Gamma$ that is weakly contained in $\bm{s}_\Gamma$. This action then witnesses that $\Gamma$ is not shift-minimal.
\end{proof}

Any group with no non-trivial normal amenable subgroups is ICC (see \cite[Appendix J]{Ha07} for a proof), so Proposition \ref{prop:noamensub} also shows

\begin{proposition}
Shift-minimal groups are \emph{ICC}.
\end{proposition}

\section{Permanence properties}\label{sec:hered}%free measure preserving actions}

This section examines various circumstances in which shift-minimality is preserved. \S\ref{subsec:triv} establishes a lemma which will be used to show that, in many cases, shift-minimality passes to finite index subgroups. %The Poincar\'{e} Recurrence Theorem plays a central role in many of the proofs in this section.

\subsection{Invariant random subgroups with trivial intersection}\label{subsec:triv} For each invariant random subgroup $\theta$ of $\Delta$ define the set
\[
P_\theta = \{ \delta \in \Delta \csuchthat \theta (\{ H\csuchthat \delta \in H \} )>0 \} .
\]
We say that two invariant random subgroups $\theta$ and $\rho$ \emph{intersect trivially} if $P_\theta \cap P_\rho = \{ e \}$. This notion comes from looking at freeness of a product action.

\begin{lemma}\label{lem:trivially}
If $\bm{a} = \Delta \cc ^a (X,\mu )$ and $\bm{b} =\Delta \cc ^b (Y,\nu )$ are measure preserving actions of $\Delta$ then $\bm{a}\times \bm{b}$ is free if and only if $\theta _{\bm{a}}$ and $\theta _{\bm{b}}$ intersect trivially.
\end{lemma}

\begin{proof}
For each $\delta \in \Delta$ we have $\mbox{Fix}^{a\times b}(\delta ) =\mbox{Fix}^a(\delta )\times \mbox{Fix}^b(\delta )$, and this set is $(\mu \times \nu )$-null if and only if either $\mbox{Fix}^a(\delta )$ is $\mu$-null or $\mbox{Fix}^b(\delta )$ is $\nu$-null. The lemma easily follows.
\end{proof}

It is a straightforward group theoretic fact that if $L$ and $K$ are normal subgroups of $\Delta$ which intersect trivially then they commute. This generalizes to invariant random subgroups as follows.

\begin{lemma}\label{lem:FinIndCom}
Let $\Delta$ be a countable group. Let $\theta , \rho \in \mbox{\emph{IRS}}_\Delta$ and suppose that $\theta$ and $\rho$ intersect trivially. Suppose $L$ and $K$ are subgroups of $\Delta$ satisfying
\begin{align*}
\theta &( \{ H \in \mbox{\emph{Sub}}_\Delta \csuchthat L\leq H \} ) > \tfrac{1}{m} \\
\rho &( \{ H\in \mbox{\emph{Sub}}_\Delta \csuchthat K\leq H \} ) > \tfrac{1}{n}
\end{align*}
for some $n,m\in \N$. Then there exist commuting subgroups $L_0\leq L$ and $K_0\leq K$ with $[L:L_0] < n$ and $[K:K_0]<m$.
\end{lemma}

\begin{proof}
Define the sets
\begin{align*}
Q_L &= \{ l\in L \csuchthat \langle lKl^{-1}\cup K\rangle \subseteq P_\rho \} \\
Q_K &= \{ k\in K \csuchthat \langle kLk^{-1}\cup L\rangle \subseteq P_\theta \} .
\end{align*}
If $l\in Q_L$ then for any $k\in K$ we have $lkl^{-1}k^{-1} \in \langle lKl^{-1}\cup K\rangle \subseteq P_\rho$.  Similarly, if $k\in Q_K$ then for any $l\in L$ we have $lkl^{-1}k^{-1} \in \langle kLk^{-1} \cup L\rangle \subseteq P_\theta$. Thus, if $l\in Q_L$ and $k\in Q_K$ then $lkl^{-1}k^{-1} \in P_\rho \cap P_\theta = \{ e \}$ and so $l$ and $k$ commute. It follows that the groups $L_0 = \langle Q_L\rangle \leq L$ and $K_0= \langle Q_K\rangle \leq K$ commute.

Suppose for contradiction that $[L:L_0]\geq n$ and let $l_0,\dots ,l_{n-1}$ be elements of distinct left cosets of $L_0$ in $L$, with $l_0= e$. For each $i<n$ let $A_i = \{ H\in \mbox{Sub}_\Delta \csuchthat l_iKl_i^{-1} \leq H \}$ so that $\rho (A_i ) = \rho (l_i^c \cdot A_0 ) = \rho (A_0) >\frac{1}{n}$ by hypothesis. There must be some $0\leq i<j<n$ with $\rho (A_i\cap A_j) > 0$. Let $l=l_j^{-1}l_i$. Then $\rho (l^c \cdot  A_0 \cap A_0 ) = \rho (A_i\cap A_j)>0$ and $l^c \cdot  A_0 \cap A_0$ consists of those $H\in\mbox{Sub}_\Delta$ such that $lKl^{-1} \cup K \leq H$. This shows that $\langle lKl^{-1}\cup K\rangle \subseteq P_\rho$ and thus $l\in Q_L\subseteq L_0$. But this contradicts that $l=l_j^{-1}l_i$ and $l_iL_0\neq l_jL_0$. Therefore $[L:L_0]< n$. Similarly, $[K:K_0]<m$. \qedhere[Lemma \ref{lem:FinIndCom}]
\end{proof}

\begin{theorem}\label{thm:NormCom}
Let $\theta , \rho \in \mbox{\emph{IRS}}_\Delta$, $L, K\leq \Delta$, and $n,m\in \N$ be as in Lemma \ref{lem:FinIndCom}, and assume in addition that $L$ and $K$ are finitely generated. Then there exist commuting subgroups $N_L$ and $N_K$, both normal in $\Delta$, with $[L:L\cap N_L]<\infty$ and $[K:K\cap N_K ]<\infty$.
\end{theorem}

\begin{proof}
For a subgroup $H\leq \Delta$ and $i\in \N$ let $H(i)$ be the intersection of all subgroups of $H$ of index strictly less than $i$. Then $L(n)$ is finite index in $L$, and $K(m)$ is finite index in $K$, since $L$ and $K$ are finitely generated. By Lemma \ref{lem:FinIndCom} $L(n)$ and $K(m)$ commute. For any $\gamma ,\delta \in \Delta$ the groups $\gamma L\gamma ^{-1}$ and $\delta K \delta ^{-1}$ satisfy the hypotheses of Lemma \ref{lem:FinIndCom} hence the groups $(\gamma L\gamma ^{-1})(n) =\gamma L(n)\gamma ^{-1}$ and $(\delta K\delta ^{-1})(m)=\delta K(m)\delta ^{-1}$ commute. It follows that the normal subgroups $N_L= \langle \bigcup _{\delta \in \Delta } \delta L(n)\delta ^{-1}\rangle$ and $N_K=\langle \bigcup _{\delta \in \Delta } \delta K(m)\delta ^{-1}\rangle$ satisfy the conclusion of the theorem.
\end{proof}

\subsection{Finite index subgroups}\label{sec:finind}
\begin{comment}Let $K$ be a finite index subgroup of $\Gamma$. In this subsection we show that, under the mild hypothesis that $\Gamma$ is ICC, shift-minimality of $\Gamma$ and shift-minimality of $K$ are equivalent.
\end{comment}
%
%We first show that shift-minimality of $K$ implies shift-minimality of $\Gamma$.
%
%If $K$ is shift-minimal then one reason $\Gamma$ may not be shift-minimal is if $C_\Gamma (K)$ is non-trivial. Indeed, since $K$ is finite index, every element of $C_\Gamma (K)$ has a finite conjugacy class in $\Gamma$. The set of elements with finite conjugacy class forms a characteristic amenable subgroups of $\Gamma$ (see \cite[Appendix J]{Ha07}), so if $C_\Gamma (K)$ is nontrivial then $\Gamma$ by is not shift-minimal by Proposition \ref{prop:noamensub}. We will now show that this is the only situation in which shift-minimality does not pass from $K$ to $\Gamma$.

The following is an analogue of a theorem of \cite{Be91}, and its proof is essentially the same as \cite[Proposition 6]{BH00}.

\begin{proposition}\label{prop:norm}
Let $\bm{a}$ be a measure preserving action of a countable group $\Gamma$ and let $N$ be a normal subgroup of $\Gamma$.  If the restriction $\bm{a}\resto N$ of $\bm{a}$ to $N$ is free then $\mu (\mbox{\emph{Fix}}^a(\gamma )) = 0$ for any $\gamma \in \Gamma$ satisfying
\begin{equation}\label{eqn:infinite}
| \{ h\gamma h ^{-1}\csuchthat h\in N \} | =\infty .
\end{equation}
Thus, if \mbox{\emph{(\ref{eqn:infinite})}} holds for all $\gamma \not\in N$ then a m.p.\ action of $\Gamma$ is free if and only if its restriction to $N$ is free.
\end{proposition}
%More generally, suppose H is any subgroup of \Gamma with $\bm{a}\resto H$ free. Let $\gamma \in \Gamma$ and suppose that the set $\{ h\gamma h ^{-1}\csuchthat h\in \gamma H\gamma ^{-1}\cap H \}$ is infinite (i.e., $C_\Gamma (\gamma ) \cap \gamma H\gamma ^{-1}\cap H$ has infinite index in $\gamma H\gamma ^{-1} \cap H$), then $\mu (\mbox{Fix}^a (\gamma )) =0$.
%
%So, for example, if $C_\Gamma (H)= \{ e \}$ and $H$ is $s$-normal (i.e., $\gamma H\gamma ^{-1}\cap H$ infinite for all $\gamma\not\in H$), then $\bm{a}\resto H$ is free iff $\bm{a}$ itself is free.

For example, it is shown in \cite{Be91} that (\ref{eqn:infinite}) holds for all $\gamma \not\in N$ whenever $C_\Gamma (N) = \{ e\}$ and $N$ is ICC.

\begin{proof}[Proof of Proposition \ref{prop:norm}]
Suppose $\gamma \in \Gamma \setminus \{ e\}$ is such that $\mu (\mbox{Fix}^a (\gamma )) > 0$ and $\{ h\gamma h ^{-1}\csuchthat h\in N\}$ is infinite. It suffices to show that $\bm{a}\resto N$ is not free. The Poincar\'{e} recurrence theorem implies that there exist $h_0,h_1\in N$ with $h_0\gamma h_0 ^{-1}\neq h_1\gamma h_1 ^{-1}$ and $\mu (h_0^a\cdot \mbox{Fix}^a(\gamma ) \cap h_1^a \cdot \mbox{Fix}^a(\gamma )) > 0$. Let $h= h_1^{-1}h_0$ so that $h \in N$ and $h\gamma h^{-1}\neq \gamma$. Since $\mbox{Fix}^a(\gamma ) = \mbox{Fix}^a(\gamma ^{-1})$ we have
\[
h^a\cdot \mbox{Fix}^a(\gamma )\cap \mbox{Fix}^a(\gamma ) = \mbox{Fix}^a(h\gamma h^{-1}) \cap \mbox{Fix}^a (\gamma ^{-1}) \subseteq \mbox{Fix}^a(\gamma ^{-1}h\gamma h ^{-1}),
\]
which implies $\mu (\mbox{Fix}^a (\gamma ^{-1}h\gamma h ^{-1})) >0$. This shows $\bm{a}\resto N$ is not free since $e\neq \gamma ^{-1}(h\gamma h ^{-1}) = (\gamma ^{-1}h\gamma )h^{-1}\in N$ by our choice of $h$.
\end{proof}

\begin{proposition}\label{prop:findfr}
Let $K$ be a finite index subgroup of a countable \emph{ICC} group $\Gamma$, and let $\bm{a}$ be a measure preserving action of $\Gamma$. If $\bm{a}\resto K$ is free, then $\bm{a}$ is free.
\end{proposition}

\begin{proof}
Let $N = \bigcap _{\gamma \in \Gamma}\gamma K\gamma ^{-1}$ be the normal core of $K$ in $\Gamma$. Then $N$ is a normal finite index subgroup of $\Gamma$. Since $\Gamma$ is ICC, the group $C_\Gamma (\gamma )$ is infinite index in $\Gamma$ for any $\gamma \in \Gamma$, hence $C_\Gamma (\gamma ) \cap N$ is infinite index in $N$. In particular $\{ h\gamma h ^{-1} \csuchthat h\in N \}$ is infinite. If $\bm{a}$ is any m.p.\ action of $\Gamma$ whose restriction to $K$ is free, then the restriction of $\bm{a}$ to $N$ is free, so by Proposition \ref{prop:norm}, $\bm{a}$ is free.
\end{proof}

Proposition \ref{prop:findfr} can be used to characterize exactly when shift-minimality of $\Gamma$ may be deduced from shift-minimality of one of its finite index subgroups.

\begin{proposition}\label{prop:findex} Let $K$ be a finite index subgroup of the countable group $\Gamma$. Suppose that $K$ is shift-minimal. Then the following are equivalent.
\begin{enumerate}
\item $\Gamma$ is shift-minimal.
\item $\Gamma$ is \emph{ICC}.
\item $\Gamma$ has no non-trivial finite normal subgroups.
\item $C_\Gamma (N) = \{ e \}$ where $N=\bigcap _{\gamma \in \Gamma}\gamma K\gamma ^{-1}$.
\end{enumerate}
\end{proposition}

\begin{proof}
Since $K$ is shift-minimal, it is also ICC by Proposition \ref{prop:noamensub}. The equivalence of (2), (3), and (4) then follows from \cite[Proposition 6.3]{Pr\'{e}12}. It remains to show that (2)$\Ra$(1). Suppose that $\Gamma$ is ICC and that $\bm{a}\prec \bm{s}_\Gamma$ is non-trivial. Then $\bm{a}\resto K\prec \bm{s}_K$, so $\bm{a}\resto K$ is free by shift-minimality of $K$, and therefore $\bm{a}$ itself is free by Proposition \ref{prop:findfr}.
\end{proof}

Proposition \ref{prop:findex} shows that, except for the obvious counterexamples, shift-minimality is inherited from a finite index subgroup. It seems likely that, conversely, shift-minimality passes from a group to each of its finite index subgroups. By Proposition \ref{prop:findex} to show this it would be enough to show that shift-minimality passes to finite index \emph{normal} subgroups (see the discussion following Question \ref{Q:fi} in \S 7). Theorem \ref{thm:NormCom} can be used to give a partial confirmation of this. Recall that a group is \emph{locally finite} if each of its finitely generated subgroups is finite.
%Theorem \ref{thm:NormCom} and co-induction the following can be shown.

\begin{theorem}\label{thm:nolocfin}
Let $N$ be a normal finite index subgroup of a shift-minimal group $\Gamma$. Suppose that $N$ has no infinite locally finite invariant random subgroups that are weakly contained in $\bm{s}_N$. Then $N$ is shift-minimal.
\end{theorem}

\begin{corollary}\label{cor:torfree}
Let $\Gamma$ be a shift-minimal group. Then every finite index subgroup of $\Gamma$ which is torsion-free is shift-minimal.
\end{corollary}

\begin{proof}[Proof of Corollary \ref{cor:torfree}]
Let $K$ be a torsion-free finite index subgroup of $\Gamma$. Note that $K$ is ICC since the ICC property passes to finite index subgroups. The group $N:=\bigcap _{\gamma \in \Gamma}\gamma K \gamma ^{-1}$ is finite index in $\Gamma$ and torsion-free, and it is moreover normal in $\Gamma$. By Theorem \ref{thm:nolocfin}, $N$ is shift-minimal, whence $K$ is shift-minimal by Proposition \ref{prop:findex}.
\end{proof}

Theorem \ref{thm:nolocfin} will follow from:

\begin{theorem}\label{thm:compl}
Let $\Delta$ be a countable group with $\mbox{\emph{AR}}_\Delta = \{ e \}$. Let $\theta$ and $\rho$ be invariant random subgroups of $\Delta$ which are not locally finite. Suppose that $\bm{\rho}$ is \emph{NA}-ergodic. Then $\theta$ and $\rho$ have non-trivial intersection.
\end{theorem}
%
%\begin{remark}
%The proof of Theorem \ref{thm:comp1} below shows that the hypothesis that $\bm{\theta}$ and $\bm{\rho}$ are both NA-ergodic can weakened to the hypothesis that $\bm{\theta}$ is NA-ergodic and $\rho$ is faithful in the sense that $\rho (\{ H\csuchthat \delta \in H \} )<1$ for all $\delta \neq e$. This hypothesis is in fact weaker: the set $K_\rho = \{ \delta \csuchthat \rho (\{ H\csuchthat \delta \in H \} ) =1 \}$ is a normal subgroup of $\Gamma$ which acts trivially under $\bm{\rho}$, so if $\bm{\rho}$ is NA-ergodic then $K_\rho$ must be amenable, and, as $\mbox{AR}_\Delta = \{ e \}$, it follows that $K_\rho = \{ e \}$.
%\end{remark}

We first show how to deduce \ref{thm:nolocfin} from \ref{thm:compl}.

\begin{proof}[Proof of Theorem \ref{thm:nolocfin} from Theorem \ref{thm:compl}]
Let $\bm{a} = N \cc ^a (X,\mu )$ be a non-trivial m.p.\ action of $N$ weakly contained in $\bm{s}_N$. We will show that $\bm{a}$ is free.

The coinduced action $\bm{c}=\mbox{CInd}_N^\Gamma (\bm{a})$ is weakly contained in $\bm{s}_\Gamma$, so $\bm{c}$ is free by shift-minimality of $\Gamma$. Let $T= \{ t_0,\dots ,t_{n-1}\}$ be a transversal for the left cosets of $N$ in $\Gamma$. Then $\bm{c}\resto N \cong \prod _{0\leq i<n}\bm{a}^{t_i}$ where for $\bm{b}\in A(N ,X,\mu )$, $\bm{b}^t\in A(N,X,\mu )$ is given by $k^{b^t}= (t^{-1}kt)^b$ for each $k\in N$, $t\in T$  \cite[10.{\bf{(G)}}]{Ke10}. Observe that $\theta _{\bm{a}^t} = (\varphi _t)_*\theta _{\bm{a}}$ where $\varphi _t :\mbox{Sub}_N \ra \mbox{Sub}_N$ is the conjugation map $H\mapsto tHt^{-1}$. In particular, for each $t\in T$, $\bm{a}^t$ is free if and only if $\bm{a}$ is free. It is easy to see that $(\bm{s}_N )^t \cong \bm{s}_N$ for each $t\in T$, so it follows that $\bm{c}\resto N\cong \prod _{0\leq i<n}\bm{a}^{t_i} \prec \bm{s}_N$. For each $j<n$ let $\bm{c}_j= \prod _{j\leq i<n}\bm{a}^{t_i}$. We will show that $\bm{c}_j$ is free for all $0\leq j<n$, which will finish the proof since this will show that $\bm{c}_{n-1}=\bm{a}^{t_{n-1}}$ is free, whence $\bm{a}$ is free.

We know that $\bm{c}_0=\bm{c}\resto N$ is free. Assume for induction that $\bm{c}_{j-1}$ is free (where $j\geq 1$ is less than $n$) and we will show that $\bm{c}_j$ is free. Note the following:
\begin{itemize}
\item[(i)] $\theta _{\bm{a}^{t_{j-1}}}$ and $\theta _{\bm{c}_j}$ intersect trivially. This follows from Lemma \ref{lem:trivially} because $\bm{c}_{j-1}=\bm{a}^{t_{j-1}}\times \bm{c}_j$ is free.
\item[(ii)] Both $\bm{\theta}_{\bm{a}^{t_{j-1}}}$ and $\bm{\theta}_{\bm{c}_j}$ are NA-ergodic, since they are both weakly contained in $\bm{s}_N$.
\item[(iii)] $\mbox{AR}_N = \{ e \}$. This is because $\Gamma$ is shift-minimal, so that $\mbox{AR}_\Gamma = \{ e \}$ by Proposition \ref{prop:noamensub}, and $N$ is normal in $\Gamma$ so apply Proposition \ref{prop:ARbasic}.
\end{itemize}
Theorem \ref{thm:compl} along with (i), (ii), and (iii) imply that either $\theta _{\bm{a}^{t_{j-1}}}$ or $\theta _{\bm{c}_j}$ is locally finite. But $N$ has no infinite locally finite invariant random subgroups weakly contained in $\bm{s}_N$ by hypothesis, and since $\mbox{AR}_N = \{ e \}$, $N$ actually has no \emph{non-trivial} locally finite invariant random subgroups weakly contained in $\bm{s}_N$. It follows that either $\theta _{\bm{a}^{t_{j-1}}}$ or $\theta _{\bm{c}_j}$ is trivial. If $\theta _{\bm{c}_j}$ is trivial then $\bm{c}_j$ is free, which is what we wanted to show. If $\theta _{\bm{a}^{t_{j-1}}}$ is trivial then $\bm{a}^{t_{j-1}}$ is free, so $\bm{a}^{t_{i}}$ is free for all $i<n$, and therefore $\bm{c}_j$ is free all the same.
\end{proof}

\begin{proof}[Proof of Theorem \ref{thm:compl}]
Suppose toward a contradiction that $\theta$ and $\rho$ intersect trivially. By hypothesis $\theta$ is not locally finite, so the set of $H\in \mbox{Sub}_\Delta$ that contain an infinite finitely generated subgroup is $\theta$-non-null. As there are only countably many infinite finitely generated subgroups of $\Delta$, there must be at least one -- call it $L$ -- for which $\theta (\{ H\csuchthat L\subseteq H \} ) >0$. Similarly, there is an infinite finitely generated $K\leq \Delta$ with $\rho (\{ H\csuchthat K \leq H \} )>0$. Then $\theta$, $\rho$, $L$ and $K$ satisfy the hypotheses of Theorem \ref{thm:NormCom} (for some $n$ and $m$), so there exist normal subgroups $N_L, N_K\leq \Delta$ which commute, with $[L:L\cap N_L]<\infty$ and $[K:K\cap N_K]<\infty$. Since $L$ and $K$ are infinite, neither $N_L$ nor $N_K$ is trivial, and since $\mbox{AR}_\Delta = \{ e \}$, both $N_L$ and $N_K$ are non-amenable.

Pick some $k\neq e$ with $k\in K\cap N_K$. Since $k\in K$, the set $\{ H\csuchthat k\in H \}$ has positive $\rho$-measure, and it is $N_L$-invariant since $N_L$ commutes with $k$. NA-ergodicity of $\bm{\rho}$ and non-amenability of $N_L$ then imply that $\rho (\{ H\csuchthat k\in H\} )=1$. On the other hand, the set
\[
M_\rho = \{ \delta \in \Delta \csuchthat \rho (\{ H\csuchthat \delta \in H \} ) =1 \}
\]
is a normal subgroup of $\Delta$ which acts trivially under $\bm{\rho}$, so NA-ergodicity of $\bm{\rho}$ implies $M_\rho$ is amenable, and as $\mbox{AR}_\Delta = \{ e \}$, we actually have $M_\rho = \{ e \}$, which contradicts that $k\in M_\rho$.
\end{proof}

Question \ref{Q:fi} below asks whether a finite index subgroup of a shift-minimal group is always shift-minimal.

\subsection{Direct sums}

\begin{proposition}\label{prop:product}
Let $(\Gamma _i )_{i\in I}$ be a sequence of countable \emph{ICC} groups and let $\bm{a}$ be a measure preserving action of $\Gamma =\bigoplus _{i\in I} \Gamma _i$. If $\bm{a}\resto \Gamma _i$ is free for each $i\in I$ then $\bm{a}$ is free. In particular, the direct sum of shift-minimal groups is shift-minimal.
\end{proposition}

\begin{proof}
We will show that if $\bm{a}$ is not free then $\bm{a}\resto \Gamma _i$ is not free for some $i\in I$. We give the proof for the case of the direct sum of two ICC groups -- say $\Gamma _1$ and $\Gamma _2$ -- since the proof for infinitely many groups is nearly identical. Let $\Gamma = \Gamma _1\times \Gamma _2$ and let $(\gamma ,\delta )\in \Gamma$ be such that $\mu \big( \mbox{Fix}^a((\gamma ,\delta ))\big) >0$ where $(\gamma ,\delta )\neq e_\Gamma$. Suppose that $\delta \neq e$ (the case where $\gamma \neq e$ is similar). Since $\Gamma _2$ is ICC we have that $C_{\Gamma _2}(\delta )$ is infinite index in $\Gamma _2$ so by Poincar\'{e} recurrence there exists $\alpha \in \Gamma _2$, $\alpha\not\in C_{\Gamma _2}(\delta )$ such that
\[
\mu \big( (e,\alpha )^a \cdot \mbox{Fix}^a ((\gamma ,\delta )) \cap \mbox{Fix}^a ((\gamma ,\delta ))\big) > 0 .
\]
Thus $\mu \big( \mbox{Fix}^a ( \langle (\gamma ,\alpha\delta \alpha ^{-1}), (\gamma ,\delta )\rangle )\big) >0$ and in particular $\mu \big( \mbox{Fix}^a ((e, \alpha \delta \alpha ^{-1}\delta ^{-1} ))\big) > 0$. Our choice of $\alpha$ implies that $\alpha\delta\alpha ^{-1}\delta ^{-1}\neq e$ and so $\bm{a}\resto \Gamma _2$ is non-free as was to be shown.
\end{proof}
\subsection{Other permanence properties}

\begin{proposition}\label{prop:CNicc}
Let $\bm{a}$ be a measure preserving action of $\Gamma $. Let $N$ be a normal subgroup of $\Gamma$. Suppose that both $N$ and $C_\Gamma (N)$ are \emph{ICC}. Suppose that $\bm{a}\resto N$ and $\bm{a}\resto C_\Gamma (N)$ are both free. Then $\bm{a}$ is free.
\end{proposition}

\begin{proof}
Let $K=C_\Gamma (N)N$. Then $K$ is normal in $\Gamma$ since both $N$ and $C_\Gamma (N)$ are normal. By hypothesis $C_\Gamma (N)\cap N = \{ e \}$ so $K\cong C_\Gamma (N)\times N$. It follows that $K$ is ICC, being a product of ICC groups. Proposition \ref{prop:product} then implies that $\bm{a}\resto K$ is free. Since $C_\Gamma (K)\leq C_\Gamma (C_\Gamma (N)) \cap C_\Gamma (N) = Z(C_\Gamma (N)) = \{ e \}$, Proposition \ref{prop:norm} implies that $\bm{a}$ is free.
\end{proof}

\begin{definition}\label{def:almostasc}
A subgroup $H$ of $\Gamma$ is called \emph{almost ascendant} in $\Gamma$ if there exists a well-ordered increasing sequence $\{ H_\alpha \} _{\alpha \leq \lambda}$ of subgroups of $\Gamma$, indexed by some countable ordinal $\lambda$, such that
\begin{enumerate}
\item[(i)] $H= H_0$ and $H_\lambda = \Gamma$.
\item[(ii)] For each $\alpha <\lambda$, either $H_{\alpha}$ is a normal subgroup of $H_{\alpha +1}$ or $H_\alpha$ is a finite index subgroup of $H_{\alpha +1}$.
\item[(iii)] $H_\beta = \bigcup _{\alpha < \beta}H_\alpha$ whenever $\beta$ is a limit ordinal.
\end{enumerate}
We call $\{ H_\alpha \} _{\alpha \leq \lambda}$ an \emph{almost ascendant series} for $H$ in $\Gamma$. If $H$ is almost ascendant in $\Gamma$ and if there exists an almost ascendant series $\{ H_\alpha \} _{\alpha \leq \lambda}$ for $H$ in $\Gamma$ such that $H_\alpha$ is normal in $H_{\alpha +1}$ for all $\alpha <\lambda$ then we say that $H$ is \emph{ascendant} in $\Gamma$ and we call $\{ H_\alpha \} _{\alpha \leq \lambda}$ an \emph{ascendant series} for $H$ in $\Gamma$.
\end{definition}

\begin{proposition}\label{prop:ICCalmostasc} Let $\bm{a}=\Gamma \cc ^a (X,\mu )$ be a measure preserving action of $\Gamma$.
\begin{enumerate}
\item Suppose that $L$ is an almost ascendant subgroup of $\Gamma$ that is \emph{ICC} and satisfies $C_\Gamma (L) = \{ e \}$. Then $\bm{a}$ is free if and only if $\bm{a}\resto L$ is free. Thus, if $L$ is shift-minimal then so is $\Gamma$.
\item Suppose that $L$ is an ascendant subgroup of $\Gamma$ such that $\mbox{\emph{AR}}_L = \mbox{\emph{AR}}_{C_\Gamma (L)} = \{ e \}$. Then $\bm{a}$ is free if and only if both $\bm{a}\resto L$ and $\bm{a}\resto C_\Gamma (L)$ are free.
\end{enumerate}
\end{proposition}

\begin{proof}
(1): Assume that $\bm{a}\resto L$ is free. Let $\{ L_\alpha \} _{\alpha \leq \lambda}$ be an almost ascendant series for $L$ in $\Gamma$. Then $C_\Gamma (L_\alpha )= \{ e \}$ for all $\alpha \leq \lambda$. By transfinite induction each $L_\alpha$ is ICC. Another induction shows that each $\bm{a}\resto L_\alpha$ is free: this is clear for limit $\alpha$, and at successors, $L_\alpha$ is either normal or finite index in $L_{\alpha +1}$, so assuming $\bm{a}\resto L_\alpha$ is free it follows that $\bm{a}\resto L_{\alpha +1}$ is free by applying either Proposition \ref{prop:CNicc} (Proposition \ref{prop:norm} also works) or Proposition \ref{prop:findfr}.

If now $L$ is shift-minimal and $\bm{a}$ is a non-trivial m.p.\ action of $\Gamma$ with $\bm{a}\prec \bm{s}_\Gamma$ then $\bm{a}\resto L\prec \bm{s}_L$ so that $\bm{a}\resto L$ is free and thus $\bm{a}$ is free.

(2): Assume that both $\bm{a}\resto L$ and $\bm{a}\resto C_\Gamma (L)$ are free. Let $\{ L_\alpha \} _{\alpha \leq \lambda}$ be an ascendant series for $L$ in $\Gamma$. Theorem \ref{thm:ascsubgrp} implies that $\mbox{AR}_{L_\alpha} = \mbox{AR}_{C_\Gamma (L_\alpha )} = \{ e \}$ for all $\alpha \leq \lambda$. For each $\alpha \leq \lambda$ we have
\[
\{ e \} = \mbox{AR}_{C_\Gamma (L_\alpha )}\cap L_{\alpha +1} = \mbox{AR}_{C_\Gamma (L_\alpha )}\cap C_{L_{\alpha +1}}(L_\alpha ) = \mbox{AR}_{C_{L_{\alpha +1}}(L_\alpha )}
\]
where the last equality follows from Corollary \ref{cor:aaAR} since the series $\{ C_{L_\beta}(L_\alpha ) \} _{\beta\leq \lambda}$ is ascendant in $C_\Gamma (L_\alpha )$. It is clear that $C_{L_{\alpha +1}}(L_\alpha ) \leq C_\Gamma (L)$, so by hypothesis $\bm{a}\resto C_{L_{\alpha+1}}(L_\alpha )$ is free for all $\alpha \leq \lambda$. We now show by transfinite induction on $\alpha\leq \lambda$ that $\bm{a}\resto L_\alpha$ is free. The induction is clear at limit stages. At successor stages, if we assume for induction that $\bm{a}\resto L_\alpha$ is free then all the hypotheses of Proposition \ref{prop:CNicc} hold and it follows that $\bm{a}\resto L_{\alpha +1}$ is free.
\end{proof}

\begin{proposition}\label{prop:NAnorm} Let $\bm{a}=\Gamma \cc ^a (X,\mu )$ be a measure preserving action of $\Gamma$. Let $K=\mbox{\emph{ker}}(\bm{a})$.
\begin{enumerate}
\item Suppose that there exists a normal subgroup $N$ of $\Gamma$ such that $\bm{a}\resto N$ is free and such that every finite index subgroups of $N$ acts ergodically. Then $\Gamma _x = K$ almost surely.
\item Suppose that $\bm{a}$ is NA-ergodic and there exists a non-amenable normal subgroup $N$ of $\Gamma$ such that $\bm{a}\resto N$ is free. Then $K$ is amenable and $\Gamma _x = K$ almost surely.
\end{enumerate}
\end{proposition}

\begin{proof}
We begin with (1). Note that, by Proposition \ref{prop:norm}, if $\gamma \in \Gamma$ is such that the set $\{ h\gamma h^{-1}\csuchthat h\in N \}$ is infinite, then $\mu (\mbox{Fix}^a(\gamma ))= 0$.  It therefore suffices to show that if $\mu (\mbox{Fix}^a(\gamma ))>0$ and $\{ h\gamma h^{-1}\csuchthat h\in N \}$ is finite, then $\gamma \in K$. This set being finite means that the group $H = C_\Gamma (\gamma ) \cap N$ is finite index in $N$, so $\bm{a}\resto H$ is ergodic by hypothesis. Since $H\leq C_\Gamma (\gamma )$, the set $\mbox{Fix}^a(\gamma )$ is $\bm{a}\resto H$-invariant, so if it is non-null then it must be conull, i.e., $\gamma \in K$, by ergodicity.

For (2), amenability of $K$ is immediate since $\bm{a}$ is non-trivial and NA-ergodic. NA-ergodicity also implies that every finite index subgroup of $N$ acts ergodically, so (1) applies and we are done.
\end{proof}

The following Corollary replaces the hypothesis in Proposition \ref{prop:ICCalmostasc}.(1) that $C_\Gamma (L)= \{ e \}$ with the hypotheses that $\mbox{AR}_\Gamma = \{ e \}$ and $\bm{a}$ is NA-ergodic.

\begin{corollary}\label{cor:NAalmostasc}
Suppose $\mbox{\emph{AR}}_\Gamma = \{ e \}$. Let $\bm{a}$ be any NA-ergodic action of $\Gamma$ and suppose that there exists a non-trivial almost ascendant subgroup $L$ of $\Gamma$ such that the restriction $\bm{a}\resto L$ of $\bm{a}$ to $L$ is free, then $\bm{a}$ itself is free.
\end{corollary}

\begin{proof}
Let $\{ L_\alpha \} _{\alpha \leq \lambda}$ be an almost ascendant series for $L$ in $\Gamma$. Since $\mbox{AR}_\Gamma = \{ e \}$, Corollary \ref{cor:aaAR} implies that $\mbox{AR}_{L_\alpha }= \{ e \}$ for each $\alpha \leq \lambda$. Suppose for induction that we have shown that $\bm{a}\resto L_\alpha$ is free for all $\alpha <\beta$. If $\beta$ is a limit then $L_\beta = \bigcup _{\alpha <\beta} L_\alpha$ so $\bm{a}\resto L_\beta$ is free as well. If $\beta =\alpha +1$ is a successor then $\bm{a}\resto L_\alpha$ is free and $L_\alpha$ is either finite index or normal in $L_\beta$. If $L_\alpha$ is finite index in $L_\beta$ then $\bm{a}\resto L_\beta$ is free by Proposition \ref{prop:findfr}. If $L_\alpha$ is normal in $L_\beta$ then $\bm{a}\resto L_\beta$ is free by Proposition \ref{prop:NAnorm}.(2). It follows by induction that $\bm{a}\resto \Gamma$ is free.
\end{proof}

\begin{corollary}\label{cor:NAcor}$\ $
\begin{enumerate}
\item Let $\Gamma$ be a countable group with $\mbox{\emph{AR}}_\Gamma = \{ e\}$. If $\Gamma$ contains a shift-minimal almost ascendant subgroup $L$ then $\Gamma$ is itself shift-minimal.
\item Suppose that $\Gamma$ is a countable group containing an ascendant subgroup $L$ such that $L$ is shift-minimal and $\mbox{\emph{AR}}_{C_\Gamma (L)}=\{ e \}$. Then $\Gamma$ is shift-minimal. In particular, if both $L$ and $C_\Gamma (L)$ are shift-minimal then so is $\Gamma$.
\end{enumerate}
\end{corollary}

\begin{proof}
Starting with (1), let $L$ be a shift-minimal almost ascendant subgroup of $\Gamma$. Let $\bm{a}$ be a non-trivial measure preserving action of $\Gamma$ weakly contained in $\bm{s}_\Gamma$. Then $\bm{a}$ is NA-ergodic and $\bm{a}\resto L$ is free, so $\bm{a}$ is free by Corollary \ref{cor:NAalmostasc}. Statement (2) is a special case of (1) since Theorem \ref{thm:ascsubgrp} shows that $\mbox{AR}_\Gamma = \{ e \}$. %since both $L$ and $C_\Gamma (L)$ have trivial amenable radical by Proposition \ref{prop:noamensub}.%
\end{proof}

\section{Examples of shift-minimal groups} \label{sec:examples}

Theorem \ref{thm:UTsm} below shows that if the reduced $C^*$-algebra of a countable group $\Gamma$ admits a unique tracial state then $\Gamma$ is shift-minimal. We can also often gain more specific information by giving direct ergodic theoretic proofs of shift-minimality. These proofs often rely on an appeal to some form of the Poincar\'{e} recurrence theorem (several proofs of which may be found in \cite{Ber96}).

\subsection{Free groups}

Since the argument is quite short it seems helpful to present a direct argument that free groups are shift-minimal.

\begin{theorem}\label{thm:freegroup}
Let $\Gamma$ be a non-abelian free group.
\begin{enumerate}
\item[(i)] If $\bm{a}=\Gamma \cc ^a (X,\mu )$ is any non-trivial measure preserving action of $\Gamma$ which is NA-ergodic then $\bm{a}$ is free.
\item[(ii)] $\Gamma$ is shift-minimal.
\end{enumerate}
\end{theorem}

\begin{proof}
For (i) we show that non-free actions of $\Gamma$ are never NA-ergodic. Suppose that $\bm{a}$ is non-free so that $\mu (\mbox{Fix}^a (\gamma )) > 0$ for some $\gamma \in \Gamma - \{ e \}$. Fix any $\delta \in \Gamma -\langle \gamma \rangle$. By the Poincar\'{e} recurrence theorem there exists an $n>0$ with $\mu (\delta ^n\cdot \mbox{Fix}^a(\gamma ) \cap \mbox{Fix}^a(\gamma ) ) >0$. The group $H$ generated by $\delta ^n \gamma \delta ^{-n}$ and $\gamma$ is free on these elements and $\alpha ^a \cdot x= x$ for every $\alpha \in H$ and $x\in \delta ^n \cdot \mbox{Fix}^a(\gamma ) \cap \mbox{Fix}^a(\gamma )$. In particular $\bm{a}\resto H$ is not ergodic, whence $\bm{a}$ cannot be NA-ergodic.

Statement (ii) now follows since any non-trivial action weakly contained in $\bm{s}_\Gamma$ is strongly NA-ergodic, hence free by (i).
\end{proof}

Another proof of part (i) of Theorem \ref{thm:freegroup} follows from Theorem \ref{thm:NAerg} (see also \cite[Lemma 24]{AGV12}). Indeed, alternative (2) of Theorem \ref{thm:NAerg} can never hold since a non-abelian free group has only countably many amenable subgroups. So if $\bm{a}$ is any non-trivial NA-ergodic action of a non-abelian free group $\Gamma$ then (1) of Theorem \ref{thm:NAerg} holds, and so $\bm{a}$ is free since the only normal amenable subgroup of $\Gamma$ is the trivial group $N = \{ e \}$.

\subsection{Property (BP)}

\begin{definition}
Let $\Gamma$ be a countable group.
\begin{enumerate}
%12/5/2012: corrected definition of Powers groups by adding "$\Gamma \neq \{ e \}$." Also, changed "$\delta _j$" to "$\alpha _j$" to be consistent with notation used below.
\item $\Gamma$ is said to be a \emph{Powers group} (\cite{Ha85}) if $\Gamma \neq \{ e \}$ and for every finite subset $F\subseteq \Gamma \setminus \{ e \}$ and every integer $N>0$ there exists a partition $\Gamma = C\sqcup D$ and elements $\alpha _1,\dots ,\alpha _N\in \Gamma$ such that
\begin{align*}
\gamma C\cap C &=\emptyset \ \mbox{ for all }\gamma \in F \\
\alpha _jD\cap \alpha _kD &= \emptyset \ \mbox{ for all }j,k \in \{ 1,\dots ,N \} , \ j\neq k .
\end{align*}
$\Gamma$ is said to be a \emph{weak Powers group} (\cite{BN88}) if $\Gamma$ satisfies all instances of the Powers property with $F$ ranging over finite subsets of mutually conjugate elements of $\Gamma \setminus \{ e \}$. We define $\Gamma$ to be a \emph{weak${}^*$ Powers group} if $\Gamma$ satisfies all instances of the Powers property with $F$ ranging over singletons in $\Gamma \setminus \{ e \}$.

\item $\Gamma$ has \emph{property $\mbox{P}_{\mbox{\tiny{nai}}}$} (\cite{BCH94}) if for any finite subset $F$ of $\Gamma$ there exists an element $\alpha \in \Gamma$ of infinite order such that for each $\gamma \in F$, the canonical homomorphism from the free product $\langle \gamma \rangle \ast \langle \alpha \rangle$ onto the subgroup $\langle \gamma ,\alpha \rangle$ of $\Gamma$ generated by $\gamma$ and $\alpha$ is an isomorphism.

    If $\Gamma$ satisfies the defining property of $\mbox{P}_{\mbox{\tiny{nai}}}$ but with $F$ only ranging over singletons, then we say that $\Gamma$ has \emph{property $\mbox{P}_{\mbox{\tiny{nai}}}^*$}.

\item $\Gamma$ is said to have property (PH) (\cite{Pro93}) if for all nonempty finite $F\subseteq \Gamma \setminus \{ e \}$ there exists some ordering $F= \{ \gamma _1,\dots ,\gamma _m \}$ of $F$ along with an increasing sequence $e\in Q_1\subseteq \cdots \subseteq Q_m$ of subsets of $\Gamma$ such that for all $i\leq m$, all nonempty finite $M\subseteq Q_i$ and all $n> 0$ we may find $\alpha _1,\dots ,\alpha _n \in Q_i$ and $T_1,\dots ,T_n$ pairwise disjoint such that
\[
(\alpha _j  \delta )\gamma _i (\alpha _j \delta ) ^{-1} (\Gamma \setminus T_j) \subseteq T_j
\]
for all $\delta \in M$ and $1\leq j\leq n$.
\end{enumerate}
\end{definition}

Examples of groups with these properties may be found in \cite{AM07, HP11, MOY11, PT11} along with the references given in the above definitions. For our purposes, what is important is a common consequence of these properties.

\begin{definition}
A countable group $\Gamma$ is said to have property (BP) if for all $\gamma \in \Gamma \setminus \{ e \}$ and $n\geq 2$ there exists $\alpha _1,\dots ,\alpha _n\in \Gamma$, a subgroup $H\leq \Gamma$, and pairwise disjoint subsets $T_1,\dots ,T_n\subseteq H$ such that
\[
\alpha _j \gamma \alpha _j ^{-1} (H\setminus T_j ) \subseteq T_j
\]
for all $j=1,\dots ,n$.
\end{definition}
%Note that we do not require $\alpha _j \gamma \alpha _j ^{-1}$ to all be distinct. So it is possible, e.g., that $T_1\sqcup T_2 =H$ and $\alpha _1=\alpha _2=e$ work, so that $\gamma T_1 \subseteq T_2$ and $\gamma T_2\subseteq T_1$.
%
%12/5/2012 : Change above definition to $n\geq 2$ instead of $n\geq 1$ so that we could remove the parenthetical "(excepting the trivial case $n=1$ and $T_1=H$)" in the following sentence. Also added "and $T_j\neq \emptyset$."
Note that when $\gamma$, $H$, $\alpha _1,\dots , \alpha _n$, and $T_1,\dots ,T_n$ are as above, then $\alpha _j\gamma \alpha _j ^{-1} \in H$ and $T_j\neq \emptyset$ for all $j\leq n$.

We show in Theorem \ref{thm:BPirs} that groups with property (BP) satisfy a strong form of shift-minimality. The definition of property (BP) (as well as its name) is motivated by an argument of M. Brin and G. Picioroaga showing that all weak Powers groups contain a free group. Their proof appears in \cite{Ha07} (see the remark following Question 15 in that paper), though we also present a version of their proof in Theorem \ref{thm:BP} since we will need it for Theorem \ref{thm:BPirs}. %The property (BP) is a hybrid of (weak versions of) the Powers property and Promislow's property (PH) \cite{Pro93}.

\begin{theorem}[Brin, Picioroaga \cite{Ha07}]
\label{thm:BP}
$\ $
\begin{enumerate}
\item[(1)] All weak${}^*$ Powers groups have property \emph{(BP)}.
\item[(2)] Property $\mbox{\emph{P}}_{\mbox{\tiny{\emph{nai}}}}^*$ implies property \emph{(BP)}.
\item[(3)] Property \emph{(PH)} implies property \emph{(BP)}.
\item[(4)] Groups with property \emph{(BP)} contain a free group.
\end{enumerate}
\end{theorem}

\begin{proof}
(1): given $\gamma \in \Gamma\setminus \{ e \}$ and $n\geq 1$ by the weak${}^*$ Powers property there exists $\alpha _1 ,\dots , \alpha _n$ and a partition $\Gamma = C\sqcup D$ of $\Gamma$ with $\gamma C\cap C = \emptyset$ and $\alpha _iD\cap \alpha _j D = \emptyset$ for all $1\leq i,j \leq n$, $i\neq j$. Take $H=\Gamma$ and for each $1\leq j \leq n$ let $T_j = \alpha _j D$ so that the sets $T_1,\dots ,T_n$ are pairwise disjoint and
\[
\alpha _j\gamma \alpha _j ^{-1} (\Gamma \setminus T_j ) = \alpha _j \gamma (\Gamma \setminus D) = \alpha _j \gamma C \subseteq \alpha _j(\Gamma \setminus C) = \alpha _j D = T_j
\]
thus verifying (BP).

(2): Let $\gamma \in \Gamma\setminus \{ e \}$. By property $\mbox{P}_{\mbox{\tiny{nai}}}^*$ there exists an element $\alpha\in \Gamma$ of infinite order such that the subgroup $H = \langle \gamma , \alpha \rangle$ of $\Gamma$ is canonically isomorphic to the free product $\langle \gamma \rangle \ast \langle \alpha \rangle$. Let $T_n$ denote the set of elements of $H$ whose reduced expression starts with $\alpha ^n\gamma ^k$ for some $k\in \Z$ with $\gamma ^k\neq e$. Then the sets $T_n$, $n\in \N$, are pairwise disjoint and $\alpha ^n\gamma \alpha ^{-n}(H\setminus T_n) \subseteq T_n$.

(3): %is merely a matter of unpacking the definition of property (PH). We do this here for the convenience of the reader.
%For statement (2), unpacking the definition of property (PH) we see that
%\begin{quote}
%$\Gamma$ has property (PH) if and only if for all nonempty finite $F\subseteq \Gamma \setminus \{ e \}$ there exists some ordering $F= \{ \gamma _1,\dots ,\gamma _n \}$ of $F$ along with an increasing sequence $e\in Q_1\subseteq \cdots \subseteq Q_n$ of subsets of $\Gamma$ such that for all $i\leq n$, all nonempty finite $M\subseteq Q_i$ and all $n> 0$ we may find $\alpha _1,\dots ,\alpha _n \in Q_i$ and $T_1,\dots ,T_n$ pairwise disjoint such that
%\[
%(\alpha _j  \delta )\gamma _i (\alpha _j \delta ) ^{-1} (\Gamma \setminus T_j) \subseteq T_j
%\]
%for all $\delta \in M$ and $1\leq j\leq n$.
%\end{quote}
Assume that $\Gamma$ has property (PH) and fix any $\gamma \in \Gamma \setminus \{ e \}$ and $n\geq 1$ toward the aim of verifying property (BP). Taking $F = \{ \gamma \}$ we obtain a set $Q = Q_1 \subseteq \Gamma$ from the above definition of (PH) with $e\in Q$. Taking $M = \{ e \}$, the defining property of $Q$ produces $\alpha _1,\dots ,\alpha _n \in Q$ and pairwise disjoint $T_1,\dots ,T_n\subseteq \Gamma$ with
\[
\alpha _j \gamma \alpha _j ^{-1} (\Gamma \setminus T_j)\subseteq T_j ,
\]
so taking $H=\Gamma$ confirms this instance of property (BP).

Statement (4) is a consequence of the following Lemma, which will be used in Theorem \ref{thm:BPirs} below.

\begin{lemma}[Brin, Picioroaga]\label{lem:BPfgrp}
Suppose that $x_1,\dots x_4$ are elements of a group $H$ and that $T_1,\dots ,T_4$ are pairwise disjoint subsets of $H$ such that
\[
x_j (H\setminus T_j ) \subseteq T_j
\]
for each $j \in \{ 1,\dots ,4 \}$. Then the group elements $u=x_1x_2$ and $v=x_3x_4$ freely generate a non-abelian free subgroup of $H$.
\end{lemma}

\begin{proof}[Proof of Lemma \ref{lem:BPfgrp}]
%For $j\in \{ 1,\dots ,5 \}$ let $x_j = \alpha _j \gamma \alpha _j ^{-1}$. %By hypothesis $\Gamma \setminus x_j T_j = x_j (\Gamma \setminus T_j ) \subseteq T_j$. Taking complements shows that $\Gamma \setminus T_j \subseteq x_jT_j$ and
The hypothesis $x_j (H\setminus T_j ) \subseteq T_j$ implies that also $x_j ^{-1}(H \setminus T_j ) \subseteq T_j$. %For $j\in \{ 1,\dots ,5\}$ we let
%\[
%\hat{T}_j = \bigcup \{ T_i \csuchthat i\neq j, \ \, 1\leq i \leq 5 \} .
%\]
%Then $\hat{T}_j\subseteq H\setminus T_j$ so that $x_j\hat{T}_j\subseteq T_j$ and $x_j^{-1}\hat{T}_j\subseteq T_j$.
For distinct $i,j \in \{ 1,\dots ,4\}$ it then follows that
\begin{align*}
x_ix_j (H\setminus T_j) &\subseteq x_i T_j \subseteq x_i(H\setminus T_i) \subseteq T_i \\
\mbox{ and } \
(x_ix_j)^{-1}(H\setminus T_i) &\subseteq x_j^{-1}T_i \subseteq x_j^{-1}(H\setminus T_j) \subseteq T_j
\end{align*}
so for $u=x_1x_2$ and $v=x_3x_4$ we have
\[
\begin{array}{cc}
u(H\setminus T_2) \subseteq T_1 & u^{-1}(H\setminus T_1) \subseteq T_2 \\
v(H\setminus T_4) \subseteq T_3 & v^{-1}(H\setminus T_3) \subseteq T_4 .
\end{array}
\]
A ping pong argument now shows that $u$ and $v$ freely generate a non-abelian free subgroup of $H$.
\qedhere[Lemma \ref{lem:BPfgrp}]
%Note that $T_j\neq \emptyset$ implies that $u$ has infinite order: $T_1\subseteq H\setminus T_2$ and the inclusion is proper, so $u^{n+1}(H\setminus T_2) \subseteq u^nT_1$ is a proper subset of $u^n(H\setminus T_2)$ and by induction $u^n(H\setminus T_2)$ is a proper subset of $(H\setminus T_2)$. This shows that $u$ has infinite order. Similarly for $v$. Now take any nonempty reduced word $w$ in $u$ and $v$ and do some ping pong.
\end{proof}

If now $\Gamma$ has property (BP) then taking any $\gamma \in \Gamma \setminus \{ e \}$ and $n=4$ we obtain $\alpha _1,\dots ,\alpha _4\in \Gamma$, $H\leq \Gamma$ and $T_1,\dots ,T_4 \subseteq H$ as in the definition of property (BP). Lemma \ref{lem:BPfgrp} now applies with $x_j = \alpha _j \gamma \alpha _j ^{-1}$ for $j\in \{ 1,\dots ,4 \}$.
\qedhere[Theorem \ref{thm:BP}]
\end{proof}

Lemma \ref{lem:BPfgrp} can be used to show that any non-trivial ergodic invariant random subgroup of a group with property (BP) contains a free group. %concentrates on the collection of subgroups containing a non-abelian free group.

\begin{theorem}\label{thm:BPirs}
Let $\Gamma$ have property \emph{(BP)} and let $\bm{a}= \Gamma \cc ^a (Y,\nu )$ be an ergodic measure preserving action of $\Gamma$. Suppose that $\bm{a}$ is non-free. Then the stabilizer of $\nu$-almost every $y\in Y$ contains a non-abelian free group. In particular, all groups with property \emph{(BP)} are shift-minimal.
\end{theorem}

\begin{proof}
Since $\bm{a}$ is non-free there exists an element $\gamma \in \Gamma\setminus \{ e \}$ such that $\nu (A) =r >0$ where $A = \mbox{Fix}^a(\gamma )$. By the Poincar\'{e} recurrence theorem, for all large enough $n$ (depending on $r$), if $A_1,\dots ,A_n \subseteq Y$ is any sequences of measurable subsets of $Y$ each of measure $r$, then there exist distinct $i_1,\dots ,i_4\leq n$ with $\nu (A_{i_1}\cap A_{i_2}\cap A_{i_3} \cap A_{i_4}) > 0$. Pick such an $n$ with $n\geq 4$. By property (BP) there exists $\alpha _1,\dots ,\alpha _n \in \Gamma$, $H\leq \Gamma$, and pairwise disjoint $T_1,\dots ,T_n \subseteq H$ such that $\alpha _i \gamma \alpha _i ^{-1} (H \setminus T_i ) \subseteq T_i$ for all $i\in \{ 1,\dots ,n \}$. By our choice of $n$ there must exist distinct $i_1,\dots ,i_4 \leq n$ such that
\begin{equation}\label{eqn:4sets}
\nu (\alpha _{i_1} ^aA \cap \alpha _{i_2} ^a A \cap \alpha _{i_3} ^a A \cap \alpha _{i_4}^a A ) >0.
\end{equation}
%Pick any $i_5\leq n$ not appearing in $\{ i_1,\dots ,i_4 \}$.
For $j=1,\dots ,4$ let $x_j = \alpha _{i_j}\gamma \alpha _{i_j}^{-1}$. Lemma \ref{lem:BPfgrp} (applied to $x_1,\dots x_4$ and $T_1,\dots T_4$) shows that $\langle x_1,\dots ,x_4 \rangle$ contains a free group. Additionally, (\ref{eqn:4sets}) shows that $\nu (\mbox{Fix}^a(\langle x_1,\dots ,x_4\rangle ))>0$ since
\[
\mbox{Fix}^a(\langle x_1,\dots ,x_4\rangle ) \supseteq \bigcap _{j=1}^4 \mbox{Fix}^a( x_i ) = \bigcap _{j=1}^4 \alpha _{i_j}^aA .
\]
The event that $\Gamma _y$ contains a free group is therefore non-null. This event is also $\bm{a}$-invariant, so ergodicity now implies that almost every stabilizer contains a free group.

If now $\bm{b}$ is any non-trivial measure preserving action of $\Gamma$ weakly contained in $\bm{s}_\Gamma$ then $\bm{b}$ is ergodic and by Lemma \ref{lem:am} almost every stabilizer is amenable hence does not contain a free group. Then $\bm{b}$ is essentially free by what we have already shown. Therefore $\Gamma$ is shift-minimal.
\end{proof}

In \cite{Be91} B\`{e}dos defines a group $\Gamma$ to be an \emph{ultraweak Powers group} if it has a normal subgroup $N$ that is a weak Powers group such that $C_\Gamma (N)= \{ e \}$. Let us say that $\Gamma$ is an \emph{ultraweak${}^*$ Powers group} if it has a normal subgroup $N$ that is an weak${}^*$ Powers group such that $C_\Gamma (N)= \{ e \}$.

\begin{theorem}
%Let $\bm{a} = \Gamma \cc ^a (X,\mu )$ be an ergodic measure preserving action of a group $\Gamma$.
Let $\Gamma$ be a countable group.
\begin{enumerate}
\item Suppose that $\Gamma$ contains an almost ascendant subgroup $L$ with property \emph{(BP)} such that $C_\Gamma (L) = \{ e \}$. Then for every ergodic m.p.\ action $\bm{a}=\Gamma \cc ^a (X,\mu )$ of $\Gamma$, either $\bm{a}$ is free or $\Gamma _x \cap L$ contains a non-abelian free group almost surely.

\item Suppose that $\Gamma$ contains an ascendant subgroup $L$ such that both $L$ and $C_\Gamma (L)$ have property \emph{(BP)}. Then for every ergodic m.p.\ action $\bm{a}=\Gamma \cc ^a (X,\mu )$ of $\Gamma$, either $\bm{a}$ is free, $\Gamma _x \cap L$ contains a non-abelian free group almost surely, or $\Gamma _x \cap C_\Gamma (L)$ contains a non-abelian free group almost surely.

\item Every non-trivial ergodic invariant random subgroup of an ultraweak${}^*$-Powers group contains a non-abelian free group almost surely.
\end{enumerate}
\end{theorem}

\begin{proof}
(1) Since $L$ has property (BP) it is ICC, so if $\bm{a}\resto L$ is free then $\bm{a}$ itself is free by part (1) of Proposition \ref{prop:ICCalmostasc}. Suppose then that $\bm{a}\resto L$ is non-free. Let $\pi : (X,\mu )\ra (Z,\eta )$ be the ergodic decomposition map for $\bm{a}\resto L$ and let $\mu = \int _z \mu _z \, d\eta (z)$ be the disintegration of $\mu$ with respect to $\eta$. Since $\bm{a}\resto L$ is non-free then the set $A\subseteq Z$, consisting of of all $z\in Z$ such that $L\cc ^a (X,\mu _z)$ is non-free, is $\eta$-non-null. If $z\in A$ then $\mu _{z} (\{ x\csuchthat L_x \mbox{ contains a non-abelian free group}\} ) = 1$ by Theorem \ref{thm:BPirs}. The event that $L_x$ contains a non-abelian free group is therefore $\mu$-non-null. This event is $\Gamma$-invariant (a subgroup contains a free group if and only if any of its conjugates contains one), so ergodicity implies that $L_x$ contains a free group almost surely. Since $L_x = \Gamma _x \cap L$ we are done.

The proof of (2) is similar, using part (2) of Proposition \ref{prop:ICCalmostasc}. (3) is immediate from (1) and the definitions.
\end{proof}

We note also that (BP) is preserved by extensions.

\begin{proposition}
Let $N$ be a normal subgroup of $\Gamma$. If $N$ and $\Gamma /N$ both have property \emph{(BP)} then $\Gamma$ also has property \emph{(BP)}.
\end{proposition}

\begin{proof}
Let $\gamma \in \Gamma \setminus \{ e \}$ and $n\geq 1$ be given.

If $\gamma \in N$ then property (BP) for $N$ implies that there exists $\alpha _1,\dots ,\alpha _n\in N$, $H\leq N$ and pairwise disjoint $T_1,\dots ,T_n\subseteq H$ as in the definition of (BP) for $N$. These also satisfy this instance of property (BP) for $\Gamma$.

If $\gamma \not \in N$ then the image of $\gamma$ in $\Gamma /N$ is not the identity element so property (BP) for $\Gamma /N$ implies that there exist cosets $\alpha _1 N ,\cdots \alpha _n N \in \Gamma /N$, a subgroup $K\leq \Gamma$ containing $N$, and pairwise disjoint $T_1,\dots ,T_n \subseteq K/N$ as in the definition of (BP) for $\Gamma /N$. Then $\alpha _1,\dots ,\alpha _n$, $K$, and the sets $T_i' =\bigcup T_i$, $i=1,\dots ,M$, verify this instance of property (BP) for $\Gamma$.
\end{proof}

\begin{remark}
If a group $\Gamma$ has property (BP) then it has the unique trace property. A quick proof of this follows \cite{BCH94}. The proof of this is almost exactly as in \cite[Lemma 2.2]{BCH94} with just a minor adjustment to the first part of their proof which we now describe. One first shows for any $\gamma \in \Gamma \setminus \{ e \}$ and any $n\geq 2$, if $\alpha _1,\dots ,\alpha _n$, $H$, and $T_1,\dots ,T_n$ are as in the definition of (BP) then for all $z= (z_1,\dots ,z_n ) \in \C ^n$ we have
\begin{equation}\label{eqn:BCH}
\big|\big| \sum _{j=1}^n z_j \lambda _\Gamma (\alpha _j \gamma \alpha _j ^{-1}) \big|\big| \leq 2||z||_2 .
\end{equation}
Let $x_j=\alpha _j \gamma \alpha _j ^{-1}$ so that $x_j \in H$ and $x_j(H\setminus T_j)\subseteq T_j$ for all $j=1,\dots ,n$. Let $1_A$ denote the indicator function of a subset $A\subseteq H$. For $f,g\in \ell ^2 (H)$ we then have
\begin{align*}
|\langle \lambda _H(x_j)f,g\rangle | &\leq |\langle \lambda _H(x_j)(1_{T_j}f),g\rangle |
        + |\langle \lambda _H(x_j)(1_{H\setminus T_j}f),g\rangle | \\
    &=    |\langle \lambda _H(x_j)(1_{T_j}f),g\rangle |
        + |\langle 1_{x_j(H\setminus T_j)}\lambda _H(x_j)(f),1_{T_j}g\rangle | \leq ||1_{T_j}f||\, ||g|| + ||f|| \, ||1_{T_j}g||.
\end{align*}
The remainder of the proof of (\ref{eqn:BCH}) now proceeds as in \cite[Lemma 2.2]{BCH94} using that the $T_j$ are pairwise disjoint. It now follows as in the paragraph following \cite[Definition 1]{BCH94} that $C^*_r(\Gamma )$ has a unique tracial state.
\end{remark}
\subsection{Linear groups}\label{sec:linear}
In the case that $\Gamma$ is a countable linear group, a theorem of Y. Glasner \cite{Gl12} shows that the existence of a non-trivial normal amenable subgroup is the only obstruction to shift-minimality: Glasner shows that every amenable invariant random subgroup of a linear group $\Gamma$ must concentrate on the subgroups of the amenable radical of $\Gamma$. Along with Proposition \ref{prop:noamensub} this implies that a countable linear group $\Gamma$ is shift-minimal if and only if $\Gamma$ contains no non-trivial normal amenable subgroups. Another way to deduce these results is to use Theorem \ref{thm:UTUIRS} below along with the following Theorem of Poznansky.

\begin{theorem}[Theorem 1.1 of \cite{Poz09}]\label{thm:Poz09}
Let $\Gamma$ be a countable linear group. Then the following are equivalent
\begin{enumerate}
\item[(1)] $\Gamma$ is $C^*$-simple.
\item[(2)] $\Gamma$ has the unique trace property.
\item[(3)] $\Gamma$ contains no non-trivial normal amenable subgroups, i.e., $\mbox{\emph{AR}}_\Gamma = \{ e \}$.
\end{enumerate}
\end{theorem}

\begin{corollary}\label{cor:linear}
Let $\Gamma$ be a countable linear group. The properties (1), (2), and (3) of Theorem \ref{thm:Poz09} are equivalent to each of the following properties:
\begin{enumerate}
\item[(4)] $\Gamma$ is shift-minimal.
\item[(5)] $\Gamma$ has no non-trivial amenable invariant random subgroups.
\end{enumerate}
\end{corollary}

\begin{proof}
The implication (2)$\Ra$(5) follows from Theorem \ref{thm:UTUIRS}, the implication (5)$\Ra$(4) is Corollary \ref{cor:notsm}, and (4)$\Ra$(3) follows from Proposition \ref{prop:noamensub}. The remaining implications follow from Poznansky's Theorem \ref{thm:Poz09}.
\end{proof}
%If $\Gamma$ is not shift-minimal then there is a non-trivial $\bm{a}$ weakly contained in $\bm{s}_\Gamma$ that is non-free. Then $\bm{a}$ is NA-ergodic, so Theorem \ref{thm:NAerg} implies that $\theta _{\bm{a}}$ must concentrate on the non-trivial amenable subgroups of $\Gamma$. It is shown by Ab\'{e}rt-Bader-Glasner-Virag that every IRS of a linear group that concentrates on amenable subgroups must concentrate on the subgroups of the amenable radical of $\Gamma$. Thus, the amenable radical of $\Gamma$ is non-trivial.
%
%The other direction follows from Proposition \ref{prop:noamensub}.
%\end{proof}
\subsection{Unique tracial state on $C^*_r(\Gamma)$}\label{sec:UT}
We write $C^*_r (\Gamma )$ for the reduced $C^*$-algebra of $\Gamma$. This is the $C^*$-algebra generated by $\{ \lambda _\Gamma (\gamma ) \csuchthat \gamma \in \Gamma \}$ in $\Es{B}(\ell ^2 (\Gamma ))$, where $\lambda _\Gamma$ denotes the left regular representation of $\Gamma$. Let $1 _e \in \ell ^2 (\Gamma )$ denote the indicator function of $\{ e \}$. We obtain a tracial state $\tau _\Gamma$, called the \emph{canonical trace} on $C^*_r(\Gamma )$, given by $\tau _\Gamma (a) = \langle a(1_e),1_e\rangle$.

Let $\rho$ be a probability measure on $\mbox{Sub}_\Gamma$ and define the function $\varphi _\rho \in \ell ^\infty (\Gamma )$ by
\[
\varphi _\rho (\gamma ) = \rho ( \{ H \csuchthat \gamma \in H \} ) .
\]
It is shown in \cite{IKT09} (see also Theorem \ref{thm:IKT}) and \cite{Ve11} that $\varphi _\rho$ is a positive definite function on $\Gamma$. It will be useful here to identify $\varphi _\rho$ as the diagonal matrix coefficient of a specific unitary representation of $\Gamma$ described below.

Consider the field of Hilbert spaces $\{ \ell ^2(\Gamma /H ) \csuchthat H\in \mbox{Sub}_\Gamma \}$.  For $\gamma \in \Gamma$ denote by $x ^\gamma \in \prod _H \ell ^2(\Gamma /H)$ the vector field $x ^\gamma _H = 1_{\gamma H}$ where $1_{ \gamma H }\in \ell ^2(\Gamma /H )$ is the indicator function of the singleton set $\{ \gamma H \} \subseteq \Gamma /H$.  Then $\{ x ^\gamma \}_{\gamma \in \Gamma}$ determines a fundamental family of measurable vector fields and we let $\Es{H}_\rho = \int ^{\oplus}_H \ell ^2 (\Gamma /H )\, d\rho$ denote the corresponding Hilbert space consisting of all square integrable measurable vector fields. The inner product on $\Es{H}_\rho$ is given by $\langle x ,y \rangle = \int _H \langle x_H,y_H\rangle _{\ell ^2 (\Gamma /H)} \, d\rho$. Define the unitary representation $\lambda _{\rho}$ of $\Gamma$ on $\Es{H}_\rho$ by
\[
\lambda _{\rho } = \int ^{\oplus} _H \lambda _{\Gamma /H} \, d\rho ,
\]
i.e., $\lambda _{\rho }(\gamma )(x)_H = \lambda _{\Gamma /H}(\gamma )(x_H)$, where $\lambda _{\Gamma /H}$ denotes the quasi-regular representation of $\Gamma$ on $\ell ^2 (\Gamma /H )$. We then have
\begin{align*}
\langle \lambda _{\rho} (\gamma ) (x^e) , x^e \rangle
    &=\int _H \langle \lambda _{\rho} (\gamma )(x^e)_H, x^e_H \rangle _{\ell ^2(\Gamma /H)} \, d\rho \\
    &=\int _H \langle \lambda _{\Gamma /H}(\gamma )(1_{H}), 1_H \rangle _{\ell ^2(\Gamma /H)} \, d\rho
    = \rho ( \{ H \csuchthat \gamma \in H \} ) = \varphi _\rho (\gamma ) .
\end{align*}
We have shown the following.

\begin{proposition}\label{prop:measPD}
$(\Es{H} _\rho , \lambda _\rho , x^e )$ is the \emph{GNS} triple associated with the positive definite function $\varphi _\rho$ on $\Gamma$.
\end{proposition}

It is clear that if $\rho$ is conjugation invariant (i.e., if $\rho$ is an invariant random subgroup) then $\varphi _\rho$ will be constant on each conjugacy class of $\Gamma$.

\begin{lemma}\label{lem:amenWC}
If $H$ is an amenable subgroup of $\Gamma$ then $\lambda _{\Gamma /H}$ is weakly contained in $\lambda _\Gamma$. Thus, for all $f\in \ell ^1(\Gamma )$ we have $||\lambda _{\Gamma /H} (f)|| \leq ||\lambda _\Gamma (f) ||$.
\end{lemma}

\begin{proof}
$H$ being amenable implies that the trivial one dimensional representation $1_H$ of $H$ is weakly contained in the left regular representation $\lambda _H$ of $H$ (\cite[Theorem G.3.2]{BHV08}). Thus by \cite[Theorem F.3.5]{BHV08} we have $\lambda _{\Gamma /H} \cong \mbox{Ind}_H^\Gamma (1_H) \prec \mbox{Ind}_H^\Gamma (\lambda _H) \cong \lambda _\Gamma$. The second statement follows immediately from \cite[F.4.4]{BHV08}.
\end{proof}

\begin{theorem}\label{thm:UTUIRS}
If $\rho$ is any measure on $\mbox{\emph{Sub}}_\Gamma$ concentrating on the amenable subgroups then $\lambda _\rho$ is weakly contained in the left regular representation $\lambda _\Gamma$ of $\Gamma$.

Therefore, if $\theta$ is an amenable invariant random subgroup of $\Gamma$ then $\varphi _\theta$ extends to a tracial state on $C^*_r(\Gamma )$ which is distinct from the canonical trace $\tau _\Gamma$ whenever $\theta$ is non-trivial.
\end{theorem}

\begin{proof}
By \cite[F.4.4]{BHV08} to show that $\lambda _\rho \prec \lambda _\Gamma$ it suffices to show that $||\lambda _\rho (f) || \leq ||\lambda _\Gamma (f)||$ for all $f\in \ell ^1 (\Gamma )$. Using that $\rho$ concentrates on the amenable subgroups and Lemma \ref{lem:amenWC} we have for $f\in \ell ^1(\Gamma )$ and $x,y\in \Es{H} _\rho$
\begin{align*}
|\langle \lambda _\rho (f)x ,y \rangle |
    &= \big| \int _H \langle \lambda _{\Gamma /H}(f)(x_H),y_H \rangle _{\ell ^2(\Gamma /H)}\, d\rho \big| \\
    &\leq \int _H ||\lambda _{\Gamma /H}(f)|| \, ||x_H|| \, ||y_H || \, d\rho \\
    &\leq ||\lambda _\Gamma (f)|| \int _H || x_H || \, ||y_H || \, d\rho \\
    &\leq ||\lambda _\Gamma (f)||\,  ||x||\, ||y||
\end{align*}
from which we conclude that $||\lambda _\rho (f)|| \leq ||\lambda _\Gamma (f)||$.

Suppose now $\theta$ is an amenable invariant random subgroup of $\Gamma$. Since $\lambda _\theta$ is weakly contained in $\lambda _\Gamma$, $\lambda _\theta$ extends to a representation of $C^*_r(\Gamma )$ and $\varphi _\theta $ extends to a state on $C^*_r(\Gamma )$ via $a\mapsto \langle \lambda _\theta (a)(x^e), x^e \rangle$. Since $\varphi _\theta$ is conjugation invariant this is a tracial state. If $\theta$ is non-trivial then there is some $\gamma \in \Gamma\setminus \{ e \}$ with $\varphi _\theta (\gamma ) = \theta (\{ H \csuchthat \gamma  \in H \} )>0$ showing that this is distinct from the canonical trace.
\end{proof}

\begin{corollary}\label{thm:UTsm}
Let $\Gamma$ be a countable group with the unique trace property. %and suppose that $C^*_r(\Gamma )$ admits a unique tracial state.
Then $\Gamma$ has no non-trivial amenable invariant random subgroups. It follows that every non-trivial NA-ergodic action of $\Gamma$ is free and $\Gamma$ is shift-minimal.
\end{corollary}

\begin{proof}
That $\Gamma$ has no non-trivial amenable invariant random subgroups follow from Theorem \ref{thm:UTUIRS}. If $\bm{a}$ is a non-trivial NA-ergodic action of $\Gamma$ then the invariant random subgroup $\theta _{\bm{a}}$ is amenable by Theorem \ref{thm:NAerg}, and thus $\theta _{\bm{a}} = \updelta _e$, i.e., $\bm{a}$ is free. Since every m.p.\ action weakly contained in $\bm{s}_\Gamma$ is NA-ergodic, $\Gamma$ is also shift-minimal.
\end{proof}

\begin{remark}
The positive definite function $\varphi _\theta$ associated to an invariant random subgroup $\theta$ is also realized in the Koopman representation $\kappa ^{\bm{s}_\theta}_0$ corresponding to the $\theta$-random Bernoulli shift $\bm{s}_{\theta ,\eta}$ of $\Gamma$ with a non-atomic base space $(Z,\eta )$ (see \cite{T-D12a} for the definition of the $\theta$-random Bernoulli shift). Indeed, take $Z=\R$ and take $\eta$ to be the standard Gaussian measure (with unit variance). Let $p_\gamma : \R ^{\leq \backslash \Gamma } \ra \R$ be the function $p_\gamma (f) = f(H_f\gamma )$. Then $p_\gamma \in L^2 _0 (\eta ^{\theta \backslash \Gamma} )$ and each $p_\gamma$ is a unit vector. In addition we have $\kappa ^{\bm{s}_{\theta ,\eta}}_0 (\gamma )(p_e)=p_\gamma$ and
\begin{equation}\label{eqn:gauss}
\langle p_\gamma , p_e \rangle = \int _H \int _{f\in \R ^{H\backslash \Gamma}} f(H\gamma )f(H)\, d\eta ^{H\backslash \Gamma} \, d\theta (H) =\int _H 1_{\{ H\csuchthat H\gamma =H\} } \, d\theta  = \varphi _\theta (\gamma )
\end{equation}
and so $(L^2_0(\eta ^{\theta \backslash \Gamma}), \kappa ^{\bm{s}_{\theta ,\eta}}_0 , p_e)$ is a triple realizing $\varphi _\theta$. %For fixed $H$ the inner integral is either equal to $0$ if $H\gamma \neq H\delta$, or equal to $1$ if $H\gamma = H\delta$. Thus the right hand side of (\ref{eqn:gauss}) is equal to $\theta ( \{ H \csuchthat H\gamma = H\delta \} ) = \theta (\{ H \csuchthat \gamma\delta ^{-1}\in H \} )$. Since $\kappa ^{\bm{s}_{\theta ,\eta}}_0(\gamma )(p_\delta ) = p_{\gamma \delta }$
%It follows that if $\theta$ is non-atomic and self-normalizing then \cite[Theorem 1.5]{T-D12a} implies that $\bm{s}_{\theta , \eta } \prec \bm{\theta}$.
%
%Let $\pi _{\theta}$ denote the restriction of $\kappa ^{\bm{s}_{\theta ,\eta}} _0$ to the closed invariant subspace of $L^2 (\eta ^{\theta \backslash \Gamma})$ generated by $\{ p_\gamma \} _{\gamma \in \Gamma}$. Then, since $\bm{s}_{\theta ,\eta} \prec \bm{\theta}\prec \bm{s}_\Gamma$, we have
%\[
%\pi _{\theta} \leq \kappa ^{\bm{s}_{\theta ,\eta}}_0 \prec \kappa ^{\bm{s}_{\Gamma}}_0 \sim %\lambda _\Gamma .
%\]
%In particular (see \cite[Appendix F.4]{BHV08}), $\pi _{\theta}$ extends to a representation of the reduced $C^*$-algebra $C^*_r(\Gamma )$, and determines a state $\tau _\theta$ on $C^*_r (\Gamma )$ given by $\tau _\theta (a) = \langle a(p_e), p_e \rangle$. Conjugation invariance of $\theta$ implies that $\tau _\theta$ is tracial, and as $\tau _\theta (\gamma ) = \theta ( \{ H\csuchthat \gamma \in H \} )$ it follows that $\tau _\theta$ is a tracial state distinct from the canonical tracial state on $C^*_r(\Gamma )$.
\end{remark}
\section{Cost}\label{sec:cost}
%Groups not having fixed price 1}
\subsection{Notation and background} See \cite{Ga00} and \cite{KM05} for background on the theory of cost of equivalence relations and groups. We recall the basic definitions to establish notation and terminology. %Given a measure preserving countable Borel equivalence relation $E$ on $(X,\mu )$, let $C_\mu (E)$ denote the cost of of $E$ with respect to $\mu$.
%
%In \cite[Corollary 10.14]{Ke10} it is shown that if $\bm{a}$ and $\bm{b}$ are free measure preserving actions of a finitely generated group $\Gamma$ then $\bm{a}\prec \bm{b}$ implies $C(\bm{b})\leq C(\bm{a})$. This is deduced from the stronger fact \cite[Theorem 10.13]{Ke10} that the cost function function $C: \mbox{FR}(\Gamma ,X,\mu ) \ra \R$, $\bm{a}\mapsto C(\bm{a})$, is upper semi-continuous for finitely generated $\Gamma$. A modification of the arguments used in that proof leads to a generalization which holds for arbitrary groups. This is Theorem \ref{thm:KecOpen} below.
%
%Some preliminary definitions and notation must be established first.

\begin{definition}\label{def:Lgraph} Let $(X,\mu )$ be a standard non-atomic probability space.
\begin{enumerate}
\item[(i)] By an \emph{L-graphing} on $(X,\mu )$ we mean a countable collection $\Phi = \{ \varphi _i:A_i \ra B_i\} _{i\in I}$ of partial Borel automorphism of $X$ that preserve the measure $\mu$. The \emph{cost} of the L-graphing $\Phi$ is given by
    \[
    C_\mu (\Phi ) = \sum _{i\in I}\mu (A_i) .
    \]
    In (ii)-(vi) below $\Phi$ denotes an L-graphing on $(X,\mu )$.

\item[(ii)] We denote by $\mc{G}_\Phi$ the graph on $X$ associated to $\Phi$, i.e., for $x,y\in X$, $(x,y)\in \mc{G}_\Phi$ if and only if $x\neq y$ and $\varphi ^{\pm 1}(x) =y$ for some $\varphi \in \Phi$. We let $d_\Phi : X\times X \ra \N \cup \{ \infty \}$ denote the graph distance corresponding to $\mc{G}_\Phi$, i.e., for $x,y\in X$,
\[
d_\Phi (x,y) =\inf \{ m\in \N \csuchthat \exists \varphi _0,\dots ,\varphi _{m-1} \in \Phi ^* \ (\varphi _{m-1} ^{\pm 1}\circ \cdots \circ \varphi _1 ^{\pm 1} \circ \varphi _0 ^{\pm 1} (x) = y ) \}
\]
where $\Phi ^* = \Phi \cup \{ \mbox{id}_X \}$ and $\mbox{id}_X:X \ra X$ is the identity map.

\item[(iii)] We let $E_\Phi$ denote the equivalence relation on $X$ generated by $\Phi$, i.e., $xE_\Phi y \IFF d_\Phi (x,y)<\infty$. Then $E_\Phi$ is a countable Borel equivalence relation that preserves the measure $\mu$.

\item[(iv)] Let $E$ be a measure preserving countable Borel equivalence relation on $(X, \mu )$. We say that $\Phi$ is an \emph{L-graphing of $E$} if there is a conull set $X_0\subseteq X$ such that $E_\Phi \resto X_0 = E\resto X_0$. This is equivalent to the condition that $[x]_{E_\Phi} = [x]_E$ for $\mu$-almost every $x\in X$. The \emph{cost} of $E$ is defined as
    \[
    C_\mu (E) = \inf \{ C_\mu (\Psi )\csuchthat \Psi \mbox{ is an L-graphing of }E\} .
    \]

\item[(v)] Let $\bm{a} = \Gamma \cc ^a (X,\mu )$ be a measure preserving action of $\Gamma$. Let $Q$ be a subset of $\Gamma$ and let $A : Q \ra \mbox{MALG}_\mu$ be a function assigning to each $\delta \in Q$ a measurable subset $A_\delta$ of $X$. Then $a$ and $A$ define an L-graphing $\Phi ^{a ,A}  = \{ \varphi ^{a, A}_\delta \csuchthat \delta \in Q \}$, where $\varphi ^{a,A} _\delta = \delta ^a \resto A_\delta$, i.e., $\mbox{dom}(\varphi ^{a,A}_\delta )= A_\delta$ and $\varphi ^{a,A}_\delta (x) = \delta ^a x$ for each $x\in A_\delta$. It is clear that $E_{\Phi ^{a,A}} \subseteq E_a$ and
\[
C_\mu (\Phi ^{a,A} ) = \sum _{\delta \in Q} \mu (A_\delta )
\]
so that $C_\mu (\Phi ^{a, A})$ only depends on the assignment $A$ and not on the action $a$.

\item[(vi)] As a converse to (v), whenever $E_\Phi \subseteq E_a$ we may find a function $A = A^{a,\Phi}  : \Gamma \ra \mbox{MALG}_\mu$ such that $\mc{G}_{\Phi ^{a, A}} = \mc{G}_{\Phi}$ and $C_\mu (\Phi ^{a, A }) \leq C_\mu (\Phi )$. Indeed, for each $\varphi \in \Phi$ there exists a measurable partition $X = \bigsqcup _{\delta \in \Gamma} A ^{a,\varphi} _\delta$ such that $\varphi \resto A^{a,\varphi} _\delta =\delta ^a \resto A^{a,\varphi }_\delta$. Then taking $A_\delta = \bigcup _{\varphi \in \Phi} A^{a,\varphi }_\delta$ works.
\end{enumerate}
\end{definition}

For a measure preserving action $\bm{a} = \Gamma \cc ^a (X,\mu )$ of $\Gamma$ denote by $E_a$ the orbit equivalence relation generated by $a$. The cost of $\bm{a}$ is defined by $C(\bm{a}) = C_\mu (E_a)$. Denote by $C(\Gamma )$ the cost of the group $\Gamma$, i.e., $C(\Gamma )$ is the infinimum of costs of free m.p.\ actions of $\Gamma$. 

By "subequivalence relation" we will always mean "Borel subequivalence relation."

\subsection{Cost and weak containment in infinitely generated groups}\label{sec:infgen}

Lemma \ref{lem:open} together with Theorem \ref{thm:KecOpen} provide a generalization of \cite[Theorem 10.13]{Ke10}. The purpose of Lemma \ref{lem:open} is to isolate versions of a few key observations from Kechris's proof.

\begin{lemma}\label{lem:open}
Let $F\subseteq \Gamma$ be finite and let $r\in \R \cup \{ \infty \}$. Then the following are equivalent for a measure preserving action $\bm{a} =\Gamma \cc ^a (X,\mu )$ of $\Gamma$:
\begin{enumerate}
\item There exists a sub-equivalence relation $E$ of $E_a$ such that $E_{a\resto \langle F\rangle}\subseteq E\subseteq E_a$ and $C_\mu (E) < r$.

\item There exists a finite $Q\subseteq \Gamma$ containing $F$ and a sub-equivalence relation $E$ of $E_a$ such that $E_{a\resto \langle F\rangle}\subseteq E\subseteq E_{a\resto\langle Q\rangle}$ and $C_\mu (E) < r$.

\item There exists a finite $Q\subseteq \Gamma$ containing $F$, an assignment $A : Q \ra \mbox{\emph{MALG}}_\mu$, and a natural number $M\in \N$ such that
\[
C_\mu ( \Phi ^{a,A} ) + \sum _{\gamma \in F} \mu (\{ x \csuchthat d_{\Phi ^{a,A}} (x,\gamma ^a x ) > M \} ) < r .
\]
\end{enumerate}
\end{lemma}

\begin{proof}[Proof of Lemma \ref{lem:open}]
We begin with the implication (3)$\Ra$(2). If such an $A : Q \ra \mbox{MALG}_\mu$ and $M\in \N$ exist then define $B : Q \ra \Gamma$ by taking $B\resto Q\setminus F = A\resto Q\setminus F$ and for $\gamma \in F$ taking
\[
B_\gamma = A_\gamma \cup \{ x \csuchthat d_{\Phi ^{a,A}} (x, \gamma ^a x ) > M \} .
\]
Let $E= E_{\Phi ^{a,B}}$. Then $C_\mu (E) \leq C_\mu (\Phi ^{a,B}) < r$ and $E_{\Phi ^{a , B}}\subseteq E_{a\resto \langle Q \rangle}$. In addition we have $E_{a\resto \langle F \rangle} \subseteq E_{\Phi ^{a , B}}$ since for each $\gamma \in F$ and $x\in X$, either $d_{\Phi ^{a,A\resto Q}} (x, \gamma ^a x ) \leq M$ so that $(x,\gamma ^a x) \in E_{\Phi ^{a,A}}\subseteq E_{\Phi ^{a,B}}$, or $d_{\Phi ^{a,A\resto Q}}(x, \gamma ^a x) > M$, in which case $x\in \mbox{dom}(\varphi ^{a,B}_\gamma )$ and so $(x, \gamma ^a x ) \in E_{\Phi ^{a,B}}$.

(2)$\Ra$(1) is obvious, and it remains to show (1)$\Ra$(3). Let $E$ be as in (1) and let $\Phi$ be an L-graphing of $E$ with $C_\mu (\Phi ) =s < r$. Since $E\subseteq E_a$ we may by \ref{def:Lgraph}.(vi) assume without loss of generality that $\Phi = \Phi ^{a,B}$ for some $B:\Gamma \ra \mbox{MALG}_\mu$, $\gamma \mapsto B_\gamma$. Let $\epsilon >0$ be such that  $s + \epsilon < r$.

We have $E_{a\resto \langle F \rangle} \subseteq E = E_{\Phi ^{a,B}}$ so, as $F$ is finite, if we take a large enough finite set $Q\subseteq \Gamma$ containing $F$, we can ensure that
\[
\sum _{\gamma \in F} \mu ( \{ x \csuchthat d_{\Phi ^{a,B\resto Q}}(x, \gamma ^a x)  = \infty  \} ) < \epsilon.
\]
So if we take $M\in \N$ large enough then
\[
\sum _{\gamma \in F} \mu (\{ x \csuchthat d_{\Phi ^{a,B\resto Q}} (x, \gamma ^a x ) > M \} ) <\epsilon .
\]
It follows that $A=B\resto Q$ and $M$ satisfy the desired properties. \qedhere[Lemma \ref{lem:open}]
\end{proof}

\begin{definition}
For each finite $F\subseteq \Gamma$ and $r\in \R\cup \{ \infty \}$ let $A_{F,r}= A_{F,r}(\Gamma ,X,\mu )$ denote the set of $\bm{a}\in A(\Gamma ,X,\mu )$ that satisfy any -- and therefore all -- of the equivalent properties (1)-(3) of Lemma \ref{lem:open}.
\end{definition}

It is clear that the set $A_{F,r}(\Gamma ,X,\mu )$ is an isomorphism-invariant (and in fact, orbit-equivalence-invariant) subset of $A(\Gamma ,X,\mu )$.  In what follows, we let $\mbox{FR}(\Gamma ,X,\mu )$ denote the subset of $A(\Gamma ,X,\mu )$ consisting of all free actions.

\begin{theorem}\label{thm:KecOpen}
Let $\Gamma$ be an infinite countable group. For each finite $F\subseteq \Gamma$ and $r\in \R \cup\{ \infty \}$ the set $A_{F,r}(\Gamma ,X,\mu )\cap \mbox{\emph{FR}}(\Gamma ,X,\mu )$ is contained in the interior of $A_{F,r}(\Gamma ,X,\mu )$. In particular, $A_{F,r}(\Gamma ,X,\mu )\cap \mbox{\emph{FR}}(\Gamma ,X,\mu )$ is open in $\mbox{\emph{FR}}(\Gamma ,X,\mu )$.
\end{theorem}

\begin{proof}
Let $\bm{a} \in A_{F,r}$ be free and let $Q \subseteq \Gamma$, $A : Q \ra \mbox{MALG}_\mu$ and $M\in \N$ be given by Lemma \ref{lem:open}.(3). For each $\gamma \in F$ let $s^a_\gamma = \mu (\{ x \csuchthat d_{\Phi ^{a,A}} (x, \gamma ^a x ) > M \} )$.  Let $s = C_\mu (\Phi ^{a,A}) + \sum _{\gamma \in F} s^a_\gamma$. By hypothesis we have $s<r$. Let $\epsilon >0$ be small enough so that $s + |F|\epsilon < r$. Since the number $C_\mu (\Phi ^{a,A} ) = \sum _{\delta \in Q} \mu (A_\delta )$ is independent of $\bm{a}$, if we can show for each $\gamma \in F$ that the set
\begin{equation}\label{eqn:nonameset}
\{ \bm{b} \in A(\Gamma ,X,\mu )\csuchthat \mu (\{ x \csuchthat d_{\Phi ^{b,A}} (x, \gamma ^b x ) > M \} ) < s^a_\gamma + \epsilon \}
\end{equation}
contains an open neighborhood of $\bm{a}$, then the intersection of these sets as $\gamma$ ranges over $F$ will by Lemma \ref{lem:open} be a subset of $A_{F,r}$ containing an open neighborhood of $\bm{a}$ and we will be done.

Fix then $\gamma \in F$, let $Q^* = Q\cup \{ e \}$ and let $\Sigma$ be the collection
\[
\Sigma = \{ ((\delta _{M-1},\dots ,\delta _0), (\epsilon _{M-1},\dots ,\epsilon _0 ) ) \csuchthat \delta _j \in Q^* \mbox{ and } \epsilon _j\in \{ -1, 1 \} \mbox{ for }j=0,\dots ,M-1 \} .
\]
For each $\bm{b} \in A(\Gamma ,X,\mu )$ and $\sigma \in \Sigma$, writing $\sigma$ as
\begin{equation}\label{eqn:sigma0}
\sigma = ((\delta _{M-1},\dots ,\delta _0), (\epsilon _{M-1},\dots ,\epsilon _0 ) )
\end{equation}
(where $\delta _j \in Q^*$ and $\epsilon _j \in \{ -1,1 \}$ for $j=0,\dots ,M-1$), we define
\begin{align*}
\varphi ^b _{\sigma} &:= (\varphi ^{b,A}_{\delta _{M-1}})^{\epsilon _{M-1}}\circ \cdots \circ (\varphi ^{b,A}_{\delta _0})^{\epsilon _0} .
\end{align*}
Let $\Sigma (\gamma )$ denote the set of all $\sigma \in \Sigma$ with the property that $\delta _{M-1}^{ \epsilon _{M-1}}\cdots \delta _0 ^{\epsilon _0} = \gamma$. Observe that for $\sigma\in \Sigma (\gamma )$ and $\bm{b}\in A(\Gamma ,X,\mu )$, if $x\in \mbox{dom}(\varphi ^b _\sigma )$ then $\varphi ^b _\sigma (x) = \gamma ^b x$ and so $d( x, \gamma ^b x)\leq M$. It follows that
\begin{equation}\label{eqn:Sigma1}
\{ x\csuchthat d_{\Phi ^{b,A}}(x,\gamma ^b x) > M \} \subseteq \bigcap _{\sigma \in \Sigma (\gamma )} X\setminus \mbox{dom}(\varphi ^b_\sigma ) .
\end{equation}
If we assume further that $\bm{b}$ is (essentially) free then, ignoring a null set, the set containment (\ref{eqn:Sigma1}) becomes an equality. Indeed, restricting to a co-null set $X_0$ on which $b$ is free we have, for $x\in X_0$, if $d_{\Phi ^{b,A}}(x,\gamma ^b x) \leq M$ then there exists some $\sigma \in \Sigma$ such that $x\in \mbox{dom}(\varphi ^b _\sigma )$ and $\varphi ^b_\sigma (x) = \gamma ^b x$. Writing $\sigma$ as in (\ref{eqn:sigma0}), this means that $(\delta _{M-1}^{\epsilon _{M-1}}\cdots \delta _0^{\epsilon _{0}})^b x = \gamma ^b x$. Since $\bm{b}$ is free on $X_0$ this implies $\delta _{M-1}^{\epsilon _{M-1}}\cdots \delta _0 ^{\epsilon _0} = \gamma$ and therefore $\sigma \in \Sigma (\gamma )$. %This is what we wanted to show since $x\in \mbox{dom}(\varphi _\sigma )$.

Now, for each $\sigma \in \Sigma$ and $\bm{b} \in A(\Gamma ,X,\mu )$ we see from the definition of $\varphi ^b_\sigma$ that the set $\mbox{dom}(\varphi ^b_\sigma )$ is an element of the Boolean algebra $\mc{A}^b$ generated by
\[
\{ \alpha ^b A_\delta \csuchthat \delta \in Q \mbox{ and }\alpha \in (Q^*\cup Q^{-1})^M \}
\]
where $(Q^*\cup Q^{-1})^M = \{ \delta _{M-1}\cdots \delta _1 \delta _0 \csuchthat \delta _j \in Q^* \cup Q^{-1} \mbox{ for }j=0,\dots ,M-1 \}$. The algebra $\mc{A}^b$ is finite since $Q$ is finite. The Boolean operations are continuous on $\mbox{MALG}_\mu$, so if $\eta >0$ is small enough (depending on $\epsilon$, $Q$, and $A$) then every $\bm{b}$ in the open neighborhood $U_\eta$ of $\bm{a}$ given by
\[
U_\eta = \{ \bm{b} \in A(\Gamma ,X,\mu ) \csuchthat \forall \alpha \in (Q^*\cup Q^{-1})^M \, \forall \delta \in Q  \ (\mu (\alpha ^b A_\delta \Delta \alpha ^a A_\delta ) < \eta ) \}
\]
satisfies
\begin{align*}
\mu \big( \bigcap _{\sigma \in \Sigma (\gamma )} X\setminus \mbox{dom}(\varphi ^b _\sigma ) \big) &<\mu \big( \bigcap _{\sigma \in \Sigma (\gamma )} X\setminus \mbox{dom}(\varphi ^a _\sigma ) \big) + \epsilon = s^a_\gamma + \epsilon
\end{align*}
where the equality follows from the paragraph following (\ref{eqn:Sigma1}) since $\bm{a}$ is free. By (\ref{eqn:Sigma1}) we then have for such $\eta$ and $\bm{b}\in U_\eta$ that
\[
\mu ( \{ x\csuchthat d_{\Phi ^{b,A}}(x,\gamma ^b x) > M \}  ) < s^a_\gamma +\epsilon
\]
which shows that the open neighborhood $U_\eta$ of $\bm{a}$ is contained in the set (\ref{eqn:nonameset}).
\end{proof}

Note that if $\bm{a}\in A(\Gamma , X,\mu )$ and $C_\mu (E_a ) < r$, then $E= E_a$ witnesses that $\bm{a}$ satisfies property (1) of Lemma \ref{lem:open} and therefore $\bm{a}\in A_{F,r}(\Gamma ,X,\mu )$ for all finite $F\subseteq \Gamma$. It is immediate that if $\Gamma$ is generated by a finite set $F_0$ then $A_{F_0,r}(\Gamma ,X,\mu ) = \{ \bm{a}\in A(\Gamma ,X,\mu ) \csuchthat C(\bm{a})<r \}$, so we recover (a slightly stronger formulation of) \cite[Theorem 10.13]{Ke10} in the following Corollary.

\begin{corollary}[Kechris, \cite{Ke10}]
Let $\Gamma$ be an infinite, finitely generated group. Then the cost function $C: A(\Gamma ,X,\mu ) \ra \R$ is upper semicontinuous at each $\bm{a}\in \mbox{\emph{FR}}(\Gamma ,X,\mu )$, i.e.,
\[
\limsup _{\bm{b}\ra \bm{a}}C (\bm{b}) \leq C(\bm{a}) .
\]
\end{corollary}

For general groups, Theorem \ref{thm:KecOpen} has several consequences for cost and weak containment. It will be helpful to introduce the following notation and definitions.

\begin{definition}\label{def:pseudo}
Let $E_0,E_1,E_2,\dots$, and $E$ be m.p.\ countable Borel equivalence relations on $(X,\mu )$. The sequence $(E_n)_{n\in \N}$ is called an \emph{exhaustion of $E$}, denoted $(E_n)_{n\in \N}\nnearrow E$, if $E_0\subseteq E_1\subseteq \cdots$, and $E=\bigcup _n E_n$. The \emph{pseudocost} of $E$, denoted $PC_\mu (E)$, is defined by
\[
PC_\mu (E) = \inf \{ \liminf _n C_\mu (E_n) \csuchthat (E_n)_{n\in \N}\nnearrow E \} .
\]
If $\bm{a}= \Gamma \cc ^a (X,\mu )$ is a m.p.\ action of a countable group $\Gamma$ then define the pseudocost of $\bm{a}$ by $PC(\bm{a}):=PC_\mu (E_a)$. Finally, define the pseudocost of $\Gamma$ by $PC(\Gamma ):= \inf \{ PC(\bm{a})\csuchthat \bm{a}\mbox{ is a free m.p. action of }\Gamma \}$.
\end{definition}

It is shown in Corollary \ref{cor:attain} below that the infimum in the definition of $PC_\mu (E)$ is always attained. If $E$ is aperiodic then $PC _\mu (E)\geq 1$ by \cite[20.1 and 21.3]{KM05}. We have $PC_\mu (E)\leq C_\mu (E)$ as witnessed by the constant sequence $(E_n)_{n\in \N}$ given by $E_n=E$ for all $n$. In many cases we actually have the equality $PC_\mu (E)=C_\mu (E)$ as we now show. Recall that a countable Borel equivalence relation $E$ on a standard Borel space $X$ is called \emph{treeable} if there exists an acyclic Borel graph $\mc{T}\subseteq X\times X$ whose connected components are the equivalence classes of $E$. Such a $\mc{T}$ is called a \emph{treeing} of $E$, and we say that $E$ is \emph{treed} by $\mc{T}$ to mean that $\mc{T}$ is a treeing of $E$. A theorem of Gaboriau (Theorem 1 of \cite{Ga00}) states that if $\mu$ is an $E$-invariant measure on $X$ and if $\mc{T}$ is a treeing of $E$ then $C_\mu (E)=C_\mu (\mc{T})= \frac{1}{2}\int _x \mbox{deg}_{\mc{T}}(x)\, d\mu$. This will be used implicitly below.%We must first examine how cost behaves with respect to increasing unions of subequivalence relations. We show below that cost respects weak containment for free actions in arbitrary countable groups, provided the actions in question have finite cost (which is automatic for actions of finitely generated groups).

\begin{proposition}\label{prop:unions}
Let $E$ be a m.p.\ countable Borel equivalence relation on $(X,\mu )$ and let $(E_n)_{n\in \N}$ be an exhaustion of $E$. %$E_0\subseteq E_1\subseteq \cdots$, is an increasing sequence of subequivalence relations of $E$ with $E=\bigcup _n E_n$.
\begin{enumerate}
\item Suppose that $C_\mu (E) <\infty$. Then $C_\mu (E)\leq \liminf _n C_\mu (E_n)$.
\item Suppose that $E$ is treeable. Then $C_\mu (E)\leq \liminf _n C_\mu (E_n)$.
\item (Gaboriau \cite{Ga00}) Suppose that $\lim _n C_\mu (E_n)=1$. Then $C_\mu (E)=1$.
\end{enumerate}
\end{proposition}

\noindent In terms of pseudocost vs.\ cost this implies

\begin{corollary}\label{cor:PC=C}
Let $E$ be a m.p.\ countable Borel equivalence relation on $(X,\mu )$.
\begin{enumerate}
\item If $C_\mu (E) <\infty$ then $PC_\mu (E)=C_\mu (E)$.
\item If $E$ is treeable then $PC_\mu (E)=C_\mu (E)$.
\item $PC_\mu (E)=1$ if and only if $C_\mu (E) =1$.
\end{enumerate}
\end{corollary}

\begin{proof}[Proof of Proposition \ref{prop:unions}]
(1): Let $r= \liminf _n C_\mu (E_n)$ and fix $\epsilon >0$. We may assume that $r<\infty$. Let $\Phi = \{ \varphi _i \} _{i=0}^\infty$ be an L-graphing of $E$ with $C_\mu (\Phi ) =\sum _{i\geq 0} \mu (\mbox{dom}(\varphi _i)) <\infty$. Let $N$ be so large that $\sum _{i>N} \mu (\mbox{dom}(\varphi _i)) <\epsilon$. %For each $i\leq N$ and $n\in \N$ let $A^i_n = \{ x\in \mbox{dom}(\varphi _i)\csuchthat (x,\varphi _i(x))\not\in E_n \}$. Then $A^i_0\supseteq A^i_1 \supseteq A^i_2\supseteq \cdots$ and $\mu (\bigcap _{n\in \N} A^i_n )=0$.
If $M_0\in \N$ is large enough then for any $n>M_0$ we have $\sum _{i\leq N} \mu (\{ x\in \dom (\varphi _i )\csuchthat (x,\varphi _i(x))\not\in E_n \} )<\epsilon$. Since $r=\liminf _n C_\mu (E_n)$ we can find some $n>M_0$ with $C_\mu (E_n )< r+ \epsilon$. Let $\Psi$ be an L-graphing of $E_n$ with $C_\mu (\Psi )<r+\epsilon$. Then
\[
\Psi \sqcup \{ \varphi _i \}_{i>N} \sqcup \{ \varphi _i \resto \{ x\in \dom (\varphi _i )\csuchthat (x,\varphi _i(x))\not\in E_n \} \} _{i\leq N}
\]
is an $L$-graphing of $E$ with cost strictly less than $r+3\epsilon$.

(2): Let $\mc{T}$ be a treeing of $E$ and let $\mc{T}_n = \mc{T}\cap E_n$. Then $\mc{T} _n\subseteq \mc{T}_{n+1}$ and $\mc{T} = \bigcup _n \mc{T}_n$ so $\lim _n C_\mu (\mc{T}_n) = C_\mu (\mc{T})$.  Let $R_n$ be the equivalence relation generated by $\mc{T}_n$. Then $R_n\subseteq E_n$ and $R_n\cap \mc{T}= \mc{T}_n$. We need the following lemma which is due to Clinton Conley.

\begin{lemma}[C. Conley] \label{lem:Clinton}
Let $F$ be a countable Borel equivalence relation treed by $\mc{T}_F$ and let $R\subseteq F$ be a subequivalence relation treed by $\mc{T}_R \subseteq \mc{T}_F$ (so that $\mc{T}_R= R\cap \mc{T}_F$). Then any equivalence relation $R'$ with $R\subseteq R'\subseteq F$ has a treeing $\mc{T}_{R'}$ with $\mc{T}_R\subseteq \mc{T}_{R'}$.
\end{lemma}

\begin{proof}
Proposition 3.3.(iii) of \cite{JKL02} shows how to obtain a treeing $\mc{T}_{R'}$ of $R'$ from the given treeing $\mc{T}_F$ of $F$. It is clear from their construction that if an edge of $\mc{T}_F$ connects two $R'$-equivalent points, then that edge remains in $\mc{T}_{R'}$. Hence, every edge in $\mc{T}_R$ remains in $\mc{T}_{R'}$.\qedhere[Lemma \ref{lem:Clinton}]
\end{proof}

Apply Lemma \ref{lem:Clinton} to $F=E$, $R=R_n$, and $R'=E_n$, along with $\mc{T}_F=\mc{T}$ and $\mc{T}_R= \mc{T}_n$, to obtain a treeing $\mc{T}_n'$ of $E_n$ with $\mc{T}_n\subseteq \mc{T}_n'$. Then $\liminf _n C_\mu (E_n) = \liminf _n C_\mu (\mc{T}_n' )\geq \liminf _n C_\mu (\mc{T}_n) = C_\mu (\mc{T})$.

(3): Since the $E_n$ are increasing and $\lim _n C_\mu (E_n)=1$ we have $|[x]_{E_n}|\ra \infty$ almost surely (see \cite[22.1]{KM05}), and so $E$ is aperiodic. It follows that $PC_\mu (E)=1$, so by Corollary \ref{cor:attain} there is an exhaustion $(E_n')_{n\in \N}$ of $E$ with $C_\mu (E_n')\ra 1$ such that $E_n'$ is aperiodic for all $n$. It follows from \cite[Proposition 23.5]{KM05} that $C_\mu (E)=1$.
%for any $\epsilon >0$ there is an $n$ and a Borel complete section $S$ for $E_n$ with $\mu (S)<\epsilon$. The statement then follows as in the proof of Proposition 23.5 of \cite{KM05}.  If $C_\mu (E)=1$ then $E$ is aperiodic so we have $1 \leq PC_\mu (E)\leq C\mu (E)=1$. On the other hand, if $PC_\mu (E)=1$
%of aperiodic since $PC_\mu (E)$ % that for any $\epsilon >0$ we may find a subsequence $n_0<n_1<\cdots$ and graphings $\mc{G}_0\subseteq \mc{G}_1\subseteq\cdots$ such that $\mc{G}_n$ is a . Let $X_0 = \{ x\in X\csuchthat |[x]_{E_0}|=\infty \}$ and for $n>0$ define $X_{n} = \{ x\in X\setminus X_{n-1} \csuchthat |[x]_{E_{n}}|=\infty \}$. Define also $X_{\infty} = X\setminus \bigcup _{n\in \N}X_n$. Then $\{ X_n \} _{0\leq n\leq \infty}$ is a partition of $X$ into $E$-invariant sets and $E\resto  X_\infty$ is hyperfinite. E_{n-1}}|<\infty \mbox{ and } |[x]_{E_n}|=\infty \}$ and let $X_{<\infty}=X\setminus X_{\infty}$. Then $X_\infty$ is $E$-invariant and $E\resto X_{<\infty}$ is aperiodic hyperfinite, so $C_\mu (E)= C_\mu (E\resto X_{<\infty}) + C_\mu (E\resto . $A$ is $E$-invariant By \cite[Proposition 23.5]{KM05}.%\cite[Lemma VI.25]{Ga00}
\end{proof}

\begin{remark}
One may also deduce (2) of Proposition \ref{prop:unions} by using the equality $C_\mu (E) -1 = \beta _1(E) -\beta _0(E)$ for treeable $E$ \cite[Corollary 3.23]{Ga02} along with \cite[Corollary 5.13]{Ga02}.
\end{remark}

\begin{corollary}\label{cor:infty}
If $E$ is a m.p.\ treeable equivalence relation on $(X,\mu )$ of infinite cost then any increasing sequence $E_0\subseteq E_1\subseteq \cdots$, with $E=\bigcup _n E_n$ satisfies $C_\mu (E_n) \ra \infty$.
\end{corollary}

\begin{proof}
Immediate from (2) of Proposition \ref{prop:unions}.
\end{proof}
%
%\begin{question}\label{Q:unions}
%Is there an example of a m.p.\ countable Borel equivalence relation $E$ such that $PC_\mu (E) < C_\mu (E)$? Equivalently, does there exist an increasing sequence $E_0\subseteq E_1\subseteq \cdots$, of m.p.\ countable Borel equivalence relations on $(X,\mu )$ with $\sup _n C_\mu (E_n) <\infty$ and $C_\mu (\bigcup _n E_n ) =\infty$?
%\end{question}
%
%If such a sequence $\{ E_n \} _{n\in \N}$ exists then, by Corollary \ref{cor:infty}, $\bigcup _n E_n$ cannot be treeable.

\begin{remark} Corollary \ref{cor:infty} may be seen as a generalization of a theorem of Takahasi.

\begin{corollary}[Takahasi \cite{Ta50}]
Suppose $H_0\subseteq H_1\subseteq \cdots$ is an ascending chain of subgroups of a free group $F$, and assume that the $H_n$ have rank uniformly bounded by some natural number $r <\infty$. Then all $H_n$ coincide for $n$ sufficiently large.
\end{corollary}

\begin{proof}
Suppose that infinitely many $H_n$ are distinct. Then $H=\bigcup _n H_n$ has infinite rank, so %(otherwise a finite generating set for $H$ would be contained in some $H_i$, contradicting that $H_i\neq H$).
Corollary \ref{cor:infty} implies that for any free m.p.\ action $H\cc ^a (X,\mu )$ we have %$C_\mu (E^X_H) = \infty$ and $E^X_H$ is treeable, so
$C_\mu (E_{a\resto H_n}) \ra \infty$, contradicting that $\sup _n C_\mu (E_{a\resto H_n})\leq \sup _n \mbox{rank}(H_n) \leq r$.
\end{proof}
\end{remark}

We will use another characterization of pseudocost in order to show that it respects weak containment. In what follows, a sequence $(Q_n)_{n\in \N}$ of subsets of a countable group $\Gamma$ is called an \emph{exhaustion of $\Gamma$} if $Q_0\subseteq Q_1\subseteq \cdots$ and $\bigcup _n Q_n = \Gamma$. A sequence $(Q_n)_{n\in \N}$ is called a \emph{finite exhaustion of $\Gamma$} if $(Q_n)_{n\in \N}$ is an exhaustion of $\Gamma$ and $Q_n$ is finite for all $n\in \N$.

\begin{lemma}\label{lem:exhaust}
%Let $\bm{b}=\Gamma \cc ^b (X,\mu )$ be a m.p.\ action of $\Gamma$ and
Let $E$ be a m.p.\ countable Borel equivalence relation on $(X,\mu )$ and let $r\in \R \cup \{ \infty \}$. Then the following are equivalent:
\begin{enumerate}
\item[(1)] There exists an exhaustion $(E_n)_{n\in \N}$ of $E$ with $\limsup _n C_\mu (E_n)\leq r$.

\item[(2)] For any countable group $\Gamma$ and any m.p.\ action $\bm{b}=\Gamma \cc ^b (X,\mu )$ with $E=E_b$, and any sequence $(F_n)_{n\in \N}$ of finite subsets of $\Gamma$, there exists a finite exhaustion $(Q_n )_{n\in \N}$ of $\Gamma$ along with an exhaustion $(E_n)_{n\in \N}$ of $E$ such that $F_n\subseteq Q_n$ and $E_{b\resto \langle Q_n\rangle}\subseteq E_n\subseteq E_{b\resto\langle Q_{n+1}\rangle }$ for all $n\in \N$, and $\limsup _n C_\mu (E_n)\leq r$.

\item[(3)] For any countable group $\Gamma$, any m.p.\ action $\Gamma \cc ^b(X,\mu )$ with $E=E_b$, and any sequence $(F_n)_{n\in \N}$ of finite subsets of $\Gamma$, there exists an exhaustion $(E_n)_{n\in \N}$ of $E$ satisfying $E_{b\resto \langle F_n\rangle}\subseteq E_n$ for all $n$ and $\limsup _n C_\mu (E_n)\leq r$.

\item[(4)] For any countable group $\Gamma$ and any m.p.\ action $\bm{b}=\Gamma \cc ^b (X,\mu )$ with $E=E_b$, we have $\bm{b}\in A_{F,r+\epsilon}$ for all finite $F\subseteq \Gamma$ and all $\epsilon >0$.

\item[(5)] There exists a countable group $\Gamma$ and a m.p.\ action $\bm{b}=\Gamma \cc ^b(X,\mu )$ with $E=E_b$ such that $\bm{b}\in A_{F,r+\epsilon}$ for all finite $F\subseteq \Gamma$ and all $\epsilon >0$.

\item[(6)] There exists a countable group $\Gamma$ and a m.p.\ action $\bm{b}=\Gamma \cc ^b(X,\mu )$ with $E=E_b$, along with an exhaustion $(Q_n)_{n\in \N}$ of $\Gamma$ and a (not necessarily increasing) sequence $(E_n)_{n\in \N}$ of subequivalence relations of $E$ such that $E_{b\resto \langle Q_n\rangle}\subseteq E_n$ and $\limsup _n C_\mu (E_n)\leq r$.
\end{enumerate}
\end{lemma}

\begin{remark}
It is clear that each of the conditions (1), (2), (3), and (6) of Lemma \ref{lem:exhaust} are equivalent to their counterparts in which "$\limsup$" is replaced with "$\liminf$" or with "$\lim$."
\end{remark}

\begin{proof}[Proof of \ref{lem:exhaust}]
(1)$\Ra$(4): Assume that $(E_n)_{n\in \N}$ is a sequence as in (1). Let $\Gamma$ and $\bm{b} =\Gamma \cc ^b(X,\mu )$ with $E=E_b$ be given. Fix a finite $F\subseteq \Gamma$ and $\epsilon >0$. Let $n\in \N$ be large enough so that $C_\mu (E_n)< r+\epsilon /2$ and $\sum _{\gamma \in F}\mu (\{ x \csuchthat \gamma ^b x\not\in [x]_{E_n} \} )<\epsilon /2$. Let $\Phi = \big\{ \gamma ^b \resto \{ x\csuchthat \gamma ^b x\not\in [x]_{E_{n}} \} \big\} _{\gamma \in F}$. Then $R:=E_n\vee E_\Phi$ is a subequivalence relation of $E$ containing $E_{b\resto \langle F \rangle}$ with $C_\mu (R)\leq C_\mu (E_n )+ C_\mu (\Phi ) <r+ \epsilon /2 + \epsilon /2 = r+ \epsilon$. Then $R$ witnesses that $\bm{b}\in A_{F,r+\epsilon }(\Gamma ,X,\mu )$. This shows that (4) holds.

(4)$\Ra$(2): Assume (4) holds. Let $\Gamma$ and $\bm{b} =\Gamma \cc ^b(X,\mu )$ with $E=E_b$ be given along with a sequence $(F_n)_{n\in \N}$ of finite subsets of $\Gamma$. We may assume without loss of generality that $(F_n)_{n\in \N}$ is a finite exhaustion of $\Gamma$. Fix some sequence of real numbers $\epsilon _n >0$ with $\epsilon _n\ra 0$. We proceed by induction to construct sequences $(Q_n)_{n\in \N}$ and $(E_n)_{n\in \N}$ as in (2). Define $Q_0= F_0$. Suppose for induction that we have constructed finite subsets $Q_0\subseteq Q_1\subseteq \cdots Q_k$ of $\Gamma$ and equivalence relations $E_0,\dots ,E_{k-1}$ with $F_i\subseteq Q_i$ for all $i\leq k$ and $E_{b\resto\langle Q_i\rangle}\subseteq E_i\subseteq E_{b\resto \langle Q_{i+1}\rangle }$ for all $i<k$. By (4) we have $\bm{b}\in A_{Q_k \cup F_{k+1},r+\epsilon _k}$, so by Lemma \ref{lem:open} there exists a finite $Q_{k+1}\subseteq \Gamma$ containing $Q_k\cup F_{k+1}$ and a subequivalence relation $E_k$ of $E_b$ with $E_{b\resto \langle Q_k\rangle}\subseteq E_k \subseteq E_{b\resto \langle Q_{k+1}\rangle}$ and $C_\mu (E_k)<r+\epsilon _k$. Then $Q_{k+1}$ and $E_k$ extend the induction to the next stage. We obtain from this inductive procedure sequences $(Q_n)$ and $(E_n)$ which satisfy (2) by construction.

(2)$\Ra$(3) is clear. (3)$\Ra$(6) holds since there always exists some countable group $\Gamma$ and some m.p.\ action $\bm{b}=\Gamma \cc ^b (X,\mu )$ with $E=E_b$ (see \cite{FM77}). (6)$\Ra$(5) is routine. Finally, the proof of $(4)\Ra (2)$ shows that $(5)\Ra (1)$.
\end{proof}

\begin{remark}\label{rem:aper}
If the the equivalence relation $E$ in Lemma \ref{lem:exhaust} is aperiodic then condition (1) implies the stronger statement (1${}^*$) in which the equivalence relations $E_n$ are additionally required to be aperiodic. Indeed, assume that $E$ is aperiodic and that (1) holds. Then (3) holds as well. By \cite[3.5]{Ke10} there is an aperiodic $T\in [E]$. Take any countable subgroup $\Gamma \leq [E]$ that generates $E$ and with $T\in \Gamma$. Then $\Gamma$ naturally acts on $(X,\mu )$ as a subgroup of $[E]$. Take some finite exhaustion $\{ F_n\} _{n\in \N}$ of $\Gamma$ with $T\in F_0$. Now apply (3) of Lemma \ref{lem:exhaust} to this sequence $\{ F_n \} _{n\in \N}$ to obtain the desired aperiodic sequence satisfying (1${}^*$).

Similarly, if $E$ is aperiodic then (3), and (6) of Lemma \ref{lem:exhaust} are each equivalent to their counterparts (3${}^*$), and (6${}^*$), in which the equivalence relations $E_n$ are each required to be aperiodic.
\end{remark}

\begin{corollary}\label{cor:attain}
Let $E$ be a m.p.\ countable Borel equivalence relation on $(X,\mu )$. There exists an exhaustion $(E_n)_{n\in \N}\nnearrow E$ with $\lim _n C_\mu (E_n)=PC_\mu (E)$. In other words, the infimum in the definition of pseudocost is always attained. In addition, if $E$ is aperiodic then such an exhaustion $(E_n)_{n\in \N}$ exists with $E_n$ aperiodic for all $n$.
\end{corollary}

\begin{proof}
Let $s= PC_\mu (E)$. By definition of $PC_\mu (E)$, for any $\delta >0$ there exists a sequence $(E_n^\delta )_{n\in \N}\nnearrow E$ with $\limsup _n C_\mu (E_n^\delta ) < s+\delta /2$. By \cite{FM77} there is a countable group $\Gamma$ and some action $\bm{b}=\Gamma \cc ^b (X,\mu )$ of $\Gamma$ such that $E=E_b$. Now, $E$ satisfies (1) of Lemma \ref{lem:exhaust} with respect to the parameter $r=s+\delta /2$, so by (1)$\Ra$(4) of Lemma \ref{lem:exhaust} we have $\bm{b}\in A_{F,s+\delta /2 +\epsilon}$ for all finite $F\subseteq \Gamma$ and $\epsilon >0$. Taking $\epsilon = \delta /2$ shows that $\bm{b}\in A_{F,r+\delta}$ for all finite $F\subseteq \Gamma$. Since $\delta >0$ was arbitrary this shows that $\bm{b}$ satisfies (5) of Lemma \ref{lem:exhaust} with respect to the parameter $s$, so by (5)$\Ra$(1) Lemma \ref{lem:exhaust} there exists a sequence $(E_n)_{n\in \N}\nnearrow E$ with $\limsup _n C_\mu (E_n)\leq s$. Since $s = PC_\mu (E) \leq \liminf _n C_\mu (E_n)$ this shows that in fact $\lim _n C_\mu (E_n) = PC_\mu (E)$. By remark \ref{rem:aper} if $E$ is aperiodic then we can choose such a sequence $(E_n)_{n\in \N}$ with $E_n$ aperiodic for all $n$.
\end{proof}

\begin{corollary}
Let $E$ be an aperiodic m.p.\ countable Borel equivalence relation on $(X,\mu )$. Assume that $E$ is ergodic. Then for any exhaustion $(R_n )_{n\in \N}$ of $E$ satisfying $C_\mu (R_n)<\infty$ for all $n\in \N$, there exists an exhaustion $(E_n)_{n\in \N}$ of $E$ with $R_n\subseteq E_n$ for all $n\in \N$ and $\lim _n C_\mu (E_n) = PC_\mu (E)$.
\end{corollary}

\begin{proof}
Let $(R_n ) _{n\in \N}$ be an exhaustion of $E$ with $C_\mu (R_n)<\infty$ for all $n$. Since $E$ is ergodic we many apply \cite[Lemma 27.7]{KM05} to obtain, for each $n\in \N$, a finitely generated group $\Gamma _n$ and a m.p.\ action $\bm{b}_n = \Gamma _n \cc ^{b_n}(X,\mu )$ with $R_n = R_{b_n}$. There is a unique action $\bm{b}=\Gamma \cc ^a (X,\mu )$ of the free product $\Gamma = \Bast _{n\in \N}\Gamma _n$ satisfying $\bm{b}\resto \Gamma _n = \bm{b}_n$ for all $n\in \N$. For each $n\in \N$ let $F_n$ be a finite generating set for $\Gamma _n$. By Corollary \ref{cor:attain} there exists an exhaustion $(E_n' ) _{n\in \N}$ of $E$ with $\lim _n C_\mu (E_n') = r$ where $r=PC_\mu (E)$. This shows that $E$ satisfies (1) of Lemma \ref{lem:exhaust}, so, by applying (3) of Lemma \ref{lem:exhaust} to the action $\bm{b}$ and the sequence $(F_n)_{n\in \N}$, we obtain an exhaustion $(E_n)_{n\in \N}$ of $E$ with $R_n = E_{b\resto \Gamma _n} \subseteq E_n$ and $\limsup _n C_\mu (E_n) \leq r$. Since $r=PC_\mu (E)$ it follows that $\lim _n C_\mu (E_n)=PC_\mu (E)$.
\end{proof}

\begin{corollary}\label{cor:PCdef}
Let $\bm{a}=\Gamma \cc ^a (X,\mu )$ be a m.p.\ action of $\Gamma$. Then $PC(\bm{a})\leq r$ if and only if $\bm{a}\in A_{F,r+\epsilon}$ for every finite $F\subseteq \Gamma$ and $\epsilon >0$.
%\begin{align*}
%PC(\bm{a}) &= \inf \{ r \in \R \csuchthat \forall \mbox{finite }F\subseteq \Gamma \, \forall \epsilon > 0 \, \ \bm{a}\in A_{F,r+\epsilon}(\Gamma ,X,\mu ) \} .
%\end{align*}
\end{corollary}

\begin{proof}
This follows from the equivalence (1)$\IFF$(4) from Lemma \ref{lem:exhaust}.
\end{proof}

%As an immediate consequence can be stated in terms of the sets $A_{F,r}(\Gamma ,X,\mu )$.
%
\begin{corollary}\label{cor:AFrPC}
Let $\bm{a} = \Gamma \cc ^a (X,\mu )$ and $\bm{b} =\Gamma \cc ^b(Y,\nu )$ be measure preserving actions of a countable group $\Gamma$. Assume that $\bm{a}$ is free. If $\bm{a}\prec \bm{b}$ then $PC(\bm{b})\leq PC(\bm{a})$.
\end{corollary}

\begin{proof}
Let $r=PC(\bm{a})$. Fix $F\subseteq \Gamma$ finite and $\epsilon >0$. % it suffices to show that $\bm{b}\in A_{F,r+\epsilon}(\Gamma ,Y,\nu )$.
Since $PC(\bm{a})=r$ we have $\bm{a}\in A_{F,r+\epsilon}(\Gamma ,X,\mu )$ by Corollary \ref{cor:PCdef}. Since $\bm{a}$ is free, Theorem \ref{thm:KecOpen} implies that $\bm{a}$ is contained in the interior of $A_{F,r+\epsilon}(\Gamma ,X,\mu )$, so by \cite[Proposition 10.1]{Ke10} there exists some $\bm{c}\in A_{F,r+\epsilon}(\Gamma ,X,\mu )$ which is isomorphic to $\bm{b}$. Hence $\bm{b}\in A_{F,r+\epsilon}(\Gamma ,Y,\nu )$ and therefore $PC(\bm{b})\leq r$ by Corollary \ref{cor:PCdef}.
\end{proof}

\begin{corollary}\label{cor:incr}
Let $\bm{a} =\Gamma \cc ^a (X,\mu )$ and $\bm{b} = \Gamma \cc ^b (Y,\nu )$ be measure preserving actions of a countably infinite group $\Gamma$. Assume that $\bm{a}$ is free and is weakly contained in $\bm{b}$. Then there exists an exhaustion $(E_n )_{n\in \N}$ of $E$ with $\lim _n C_\mu (E_n)\leq C(\bm{a})$ and $E_n$ aperiodic for all $n\in \N$.
\end{corollary}

\begin{proof}%[Proof of Theorem \ref{thm:incr}]
Corollary \ref{cor:AFrPC} tells us that $PC(\bm{b})\leq PC(\bm{a})$, so by \ref{cor:attain} we can find an exhaustion $(E_n )_{n\in \N}$ of $E$, with $\lim _n C_\mu (E_n)\leq PC(\bm{a})$ and $E_n$ aperiodic for all $n\in \N$. Since $PC(\bm{a})\leq C(\bm{a})$ we are done. \qedhere
\end{proof}

\begin{corollary}\label{cor:weakcon}
Let $\bm{a}$ and $\bm{b}$ be m.p.\ actions of a countably infinite group $\Gamma$. Assume that $\bm{a}$ is free and $\bm{a}\prec \bm{b}$.
\begin{enumerate}
\item If $C (\bm{b})<\infty$ then $C (\bm{b})\leq C (\bm{a})$.
\item If $E_b$ is treeable then $C(\bm{b})\leq C(\bm{a})$.
\item If $C(\bm{a})=1$ then $C(\bm{b})=1$.
\end{enumerate}
\end{corollary}

\begin{proof}
(1) and (2): Suppose $C(\bm{b})<\infty$ or $E_b$ is treeable. Then by Corollary \ref{cor:PC=C} and Corollary \ref{cor:AFrPC} we have $C(\bm{b})=PC(\bm{b})\leq PC(\bm{a})\leq C(\bm{a})$.

Similarly, if $C(\bm{a})=1$ then by Corollary \ref{cor:AFrPC} we have $PC(\bm{b}) \leq PC(\bm{a})\leq C(\bm{a})=1$, so $PC(\bm{b})=1$ and thus $C(\bm{b}) =1$ by Corollary \ref{cor:PC=C}.
\end{proof}

\begin{definition}
A group $\Gamma$ is said to have \emph{fixed price $1$} if $C(\bm{a}) = 1$ for every free measure preserving action $\bm{a}$ of $\Gamma$.
\end{definition}

In \cite{AW11}, Ab\'{e}rt and Weiss combine their theorem on free actions (stated above in Theorem \ref{thm:AW}) with \cite[Theorem 10.13]{Ke10} to characterize finitely generated groups $\Gamma$ with fixed price 1 in terms of the Bernoulli shift $\bm{s}_\Gamma$. We can now remove the hypothesis that $\Gamma$ is finitely generated.

\begin{corollary}\label{cor:FP1}
Let $\Gamma$ be a countable group. Then the following are equivalent:
\begin{enumerate}
\item[(1)] $\Gamma$ has fixed price $1$
\item[(2)] $C(\bm{s}_\Gamma ) = 1$
\item[(3)] $C(\bm{a}) = 1$ for some m.p.\ action $\bm{a}$ weakly equivalent to $\bm{s}_\Gamma$.
\item[(4)] $PC(\bm{a})=1$ for some m.p.\ action $\bm{a}$ weakly equivalent to $\bm{s}_\Gamma$.
\item[(5)] $\Gamma$ is infinite and $C(\bm{a})\leq 1$ for some non-trivial m.p.\ action $\bm{a}$ weakly contained in $\bm{s}_\Gamma$.
\end{enumerate}
\end{corollary}

\begin{proof}
(1)$\Ra$(2) holds since $\bm{s}_\Gamma$ is free. (2)$\Ra$(3) is clear.  $(3)\IFF (4)$ follows from Corollary \ref{cor:PC=C}. Suppose that (3) holds and we will prove (1). Let $\bm{a}$ be weakly equivalent to $\bm{s}_\Gamma$ with $C(\bm{a})=1$. This implies $\bm{a}$ is free. If $\bm{b}$ is another free measure preserving action of $\Gamma$ then $\bm{a}\prec \bm{b}$ by Theorem \ref{thm:AW}, so Corollary \ref{cor:weakcon} shows that $C(\bm{b}) =1$. Thus $\Gamma$ has fixed price $1$. This shows that properties (1), (2), and (3) are equivalent. The implication (3)$\Ra$(5) is clear.

The proof of the remaining implication (5)$\Ra$(3) uses Lemma \ref{lem:cost1}, proved in \S\ref{sec:fp1sm} below. Assume that (5) holds. Let $\bm{a}=\Gamma \cc ^a (X,\mu )$ be a non-trivial action weakly contained in $\bm{s}_\Gamma$ with $C(\bm{a})\leq 1$. Let $\theta = \theta _{\bm{a}}$. If $\Gamma$ is amenable then (1) holds, so we may assume that $\Gamma$ is non-amenable. Then $\bm{s}_\Gamma$ is strongly ergodic, hence both $\bm{a}$ and $\bm{\theta}$ are weakly mixing. It follows that $\theta$ is either a point mass at some finite normal subgroup $N$ of $\Gamma$, or $\theta$ concentrates on the infinite subgroups of $\Gamma$.

{\bf Case 1:} $\theta$ is a point mass at some finite normal subgroup $N\leq \Gamma$. Then $C(\bm{a})=1$ since $E_a$ is aperiodic. By \cite[Proposition 4.7]{CKT-D12} there is some $\bm{b}=\Gamma \cc ^b (Y,\nu )$ weakly equivalent to $\bm{s}_\Gamma$ such that $\bm{a}$ is a factor of $\bm{b}$, say via the factor map $\pi :Y\ra X$. Let $Y_0$ be a Borel transversal for the orbits of $N\cc ^b (Y,\nu )$ and let $\sigma :Y\ra Y_0$ be the corresponding selector. Let $\nu _0$ denote the normalized restriction of $\nu$ to $Y_0$ and let $\bm{b}_0$ be the action of $\Gamma$ on $(Y_0, \nu _0)$ given by $\gamma ^{b_0}y = \sigma (\gamma ^by)$. Then $\pi$ factors $\bm{b}_0$ onto $\bm{a}$. Since $\theta _{\bm{a}}=\theta _{\bm{b}_0}=\updelta _N$, the actions $\bm{a}$ and $\bm{b}_0$ descend to free actions $\tilde{\bm{a}}$ and $\tilde{\bm{b}}_0$ respectively of $\Gamma /N$, and $\pi$ factors $\tilde{\bm{b}}_0$ onto $\tilde{\bm{a}}$. Then $C(\tilde{\bm{a}})= C(\bm{a})=1$, so $C(\bm{\tilde{b}_0})=1$ by Corollary \ref{cor:weakcon}. Since $E_{b_0}=E_b\resto Y_0$ we have $C_{\nu _0}(E_b\resto Y_0) = 1$, so $C(\bm{b})=C_\nu (E_b)=1$ by \cite[Theorem 25.1]{KM05} (\cite[Theorem 21.1]{KM05} also works). This shows that (3) holds.
%
%Then $E_b\resto Y_0 = E_{b_0}$. The actions $\bm{a}$ and $\bm{b}_0$ descend to the actions $\tilde{\bm{a}}$ and $\tilde{\bm{b}}_0$ respectively, of $\Gamma /N$. In addition, $\bm{a}\sqsubseteq \bm{b}_0$ so $\tilde{\bm{a}}\sqsubseteq \tilde{\bm{b}_0}$. Since $\Gamma$ is infinite and $N$ is finite, $C(\bm{a})=1$, so this implies $C(\bm{b}_0)=1$. Since $E_{b_0}=E_b\resto Y_0$ it follows that $C(\bm{b})=1$. Thus $\Gamma$ satisfies (3). %The argument of \cite[VI.19.(ii)]{Ga00} % $\bm{a}$ descends to a free action $\bm{b}$ of $\Gamma /N$ and $C(\Gamma )$

{\bf Case 2:} $\theta$ is infinite. We have $\bm{a}\prec \bm{s}_\Gamma$, so $\bm{a}$ is NA-ergodic and therefore $\theta$ is amenable by Theorem \ref{thm:NAerg}. Then $C(\bm{\theta}_{\bm{a}}\times \bm{s}_\Gamma )=1$ by Lemma \ref{lem:cost1}, and $\bm{\theta} _{\bm{a}}\times \bm{s}_\Gamma$ is weakly equivalent to $\bm{s}_\Gamma$, so (3) holds.
\end{proof}

\begin{note}\label{note:OE} Similar to \cite[Corollary 10.14]{Ke10}, one may strengthen Corollaries \ref{cor:AFrPC}, \ref{cor:incr}, and \ref{cor:weakcon} by replacing the hypothesis $\bm{a}\prec \bm{b}$ their statements with the weaker hypothesis that
\begin{equation}\label{eqn:OE}
\bm{a} \in \ol{\{ \bm{c}\in A(\Gamma ,X,\mu ) \csuchthat E_c\mbox{ is orbit equivalent to }E_b\}}
\end{equation}
where $(X,\mu )$ is the underlying space of $\bm{a}$. The proofs remain the same. Note that (\ref{eqn:OE}) is actually slightly weaker than the hypothesis $\bm{a}\preceq \bm{b}$ from \cite[Corollary 10.14]{Ke10}, since the action $\bm{c}$ from (\ref{eqn:OE}) ranges over all of $A(\Gamma ,X,\mu )$ and not just $\mbox{FR}(\Gamma ,X,\mu )$. Specializing to the case where $\Gamma$ is finitely generated, we recover a somewhat strengthened version of the first statement of \cite[Corollary 10.14]{Ke10}.
%Note also that any action weakly containing a free action is itself free, so the assumptions of Corollary \ref{cor:weakcon} actually imply that $\bm{b}$ is free, although this does not enter into the proof. Indeed, the weaker assumption that $\bm{a} \in \ol{\{ \bm{c}\in A(\Gamma ,X,\mu ) \csuchthat E_c\mbox{ is orbit equivalent to }E_b\}}$ does not imply that $\bm{b}$ is free.
\end{note}

\subsection{The cost of a generic action}\label{sec:generic} The results of the previous section have consequences for generic properties (with respect to the weak topology) in $\mbox{FR}(\Gamma ,X,\mu )$ related to cost. We begin by proving analogues of Corollaries \ref{cor:attain} and \ref{cor:PC=C} for groups. Recall that a countable group $\Gamma$ is called \emph{treeable} if it admits a free measure preserving action $\bm{a}$ such that $E_a$ is treeable.

\begin{proposition}\label{prop:analogue}
Let $\Gamma$ be a countably infinite group.
\begin{enumerate}
\item[(1)] Suppose that $C(\Gamma )<\infty$. Then for any free m.p.\ action $\bm{b}=\Gamma \cc ^b (X,\mu )$ of $\Gamma$, and any exhaustion $(E_n)_{n\in \N}$ of $E_b$, we have $\liminf _{n\ra\infty} C_\mu (E_n) \geq C(\Gamma )$. Hence $PC(\Gamma )=C(\Gamma )$.
\item[(2)] Suppose that $\Gamma$ is treable. Then $PC(\Gamma )=C(\Gamma )$.
\item[(3)] $PC(\Gamma ) =1$ if and only if $C(\Gamma ) =1$.
\item[(4)] $PC(\Gamma )$ is attained by some free m.p.\ action of $\Gamma$. In fact, if $\bm{a}\in \mbox{FR}(\Gamma ,X,\mu )$ has dense conjugacy class in $(\mbox{FR}(\Gamma ,X,\mu ), w)$ then $PC(\bm{a} ) = PC(\Gamma )$.
\end{enumerate}
\end{proposition}

\begin{proof}
(1): Let $\bm{b}$ be a free m.p.\ action of $\Gamma$. It suffices to show that $PC(\bm{b})\geq C(\Gamma )$. Let $\bm{a}$ be a free m.p.\ action of $\Gamma$ with $C(\bm{a})=C(\Gamma )<\infty$ and let $\bm{c}=\bm{a}\times \bm{b}$. Then by the remark at the bottom of p.\ 78 in \cite{Ke10} we have $C(\bm{c})\leq C(\bm{a}) = C(\Gamma )$, hence $C(\bm{c})=C(\Gamma )<\infty$. Since $C(\bm{c})<\infty$ we have $PC(\bm{c})=C(\bm{c})$ by (1) of Corollary \ref{cor:PC=C}. In addition, $\bm{b}\prec \bm{c}$ and $\bm{b}$ is free, so Corollary \ref{cor:AFrPC} implies $PC(\bm{b})\geq PC(\bm{c})=C(\bm{c})=C(\Gamma )$.

(2): Let $\bm{b}$ be a free m.p.\ action of $\Gamma$. Once again it suffices to show $PC(\bm{b})\geq C(\Gamma )$. Let $\bm{a}$ be a free m.p.\ action of $\Gamma$ with $E_a$ treeable and let $\bm{c}=\bm{a}\times \bm{b}$. By \cite[Proposition 30.5]{KM05} $E_c$ is treeable and $C(\bm{c})=C(\bm{a})=C(\Gamma )$. Then (2) of Corollary \ref{cor:PC=C} implies that $PC(\bm{c})=C(\bm{c})$, so, as $\bm{b}\prec \bm{c}$, Corollary \ref{cor:AFrPC} implies that $PC(\bm{b})\geq PC(\bm{c})=C(\bm{c})=C(\Gamma )$.
%Can also just use an argument as in the remark from p. 78 of Ke10 again instead of invoking Corollary \ref{cor:AFrPC}.

(3): This is immediate from (3) of Corollary \ref{cor:PC=C}.

(4): If $\bm{a}\in \mbox{FR}(\Gamma ,X,\mu )$ has dense conjugacy class this means that $\bm{b}\prec \bm{a}$ for every m.p.\ action $\bm{b}$ of $\Gamma$ \cite[Proposition 10.1]{Ke10} (also note that such an $\bm{a}$ exists by \cite[Theorem 10.7]{Ke10}). Corollary \ref{cor:AFrPC} then shows that $PC(\bm{a})\leq \inf \{ PC(\bm{b}) \csuchthat \bm{b}\in \mbox{FR}(\Gamma ,X,\mu ) \} = PC(\Gamma )$, hence $PC(\bm{a})=PC(\Gamma )$.
\end{proof}

By \cite[Proposition 10.10]{Ke10} the cost function $\bm{a}\mapsto C(\bm{a} )$ is constant on a dense $G_\delta$ subset of $\mbox{FR}(\Gamma ,X,\mu )$. Let $C_{\mbox{\tiny{gen}}}(\Gamma )\in [0,\infty ]$ denote this constant value. Similarly, the pseudocost function $\bm{a}\mapsto PC(\bm{a} )$ is constant on a dense $G_\delta$ subset of $\mbox{FR}(\Gamma ,X,\mu )$. Denote this constant value by $PC_{\mbox{\tiny{gen}}}(\Gamma )$. Problem 10.11 of \cite{Ke10} asks whether $C_{\mbox{\tiny{gen}}}(\Gamma )= C(\Gamma )$ holds for every countably infinite group $\Gamma$, and \cite[Corollary 10.14]{Ke10} shows that the equality holds whenever $\Gamma$ is finitely generated.

\begin{corollary}\label{cor:Cgen}
Let $\Gamma$ be a countably infinite group. Then
\begin{enumerate}
\item The set $\mbox{\emph{MINPCOST}}(\Gamma ,X,\mu ) = \{ \bm{a}\in \mbox{\emph{FR}}(\Gamma ,X,\mu )\csuchthat PC(\bm{a})=PC(\Gamma ) \}$ is dense $G_\delta$ in $A(\Gamma ,X,\mu )$. In particular, $PC_{\mbox{\tiny{gen}}}(\Gamma )=PC(\Gamma )$.
\item Either $C_{\mbox{\tiny{gen}}}(\Gamma ) = C(\Gamma )$ or $C_{\mbox{\tiny{gen}}}(\Gamma )= \infty$.
\item If $PC(\Gamma ) =1$ then $C_{\mbox{\tiny{gen}}}(\Gamma ) = C(\Gamma ) =1$.
\end{enumerate}
\end{corollary}

\begin{proof}
(1): Let $r= PC(\Gamma )$. Corollary \ref{cor:PCdef} shows that
\[
\mbox{MINPCOST}(\Gamma ,X,\mu ) = \bigcap \{ A_{F,r+1/n}(\Gamma ,X,\mu ) \cap \mbox{FR}(\Gamma ,X,\mu ) \csuchthat F\subseteq \Gamma \mbox{ is finite and }n\in \N \} .
\]
To show this set is dense $G_\delta$ in $A(\Gamma ,X,\mu )$ it therefore suffices to show that $A_{F,r+\epsilon}(\Gamma ,X,\mu )\cap \mbox{FR}(\Gamma ,X,\mu )$ is dense $G_\delta$ for each $F\subseteq \Gamma$ finite and $\epsilon >0$. By \cite[Theorem 10.8]{Ke10}, the set $\mbox{FR}(\Gamma ,X,\mu )$ is dense $G_\delta$ in $A(\Gamma ,X,\mu )$. Theorem \ref{thm:KecOpen} shows that $A_{F,r+\epsilon}$ is relatively open in $\mbox{FR}(\Gamma ,X,\mu )$, so it only remains to show that it is dense. By Proposition \ref{prop:analogue} we have $PC(\bm{a})=PC(\Gamma )$ whenever $\bm{a}\in \mbox{FR}(\Gamma ,X,\mu )$ has a dense conjugacy class. Since the set of actions with dense conjugacy class is dense $G_\delta$ in $\mbox{FR}(\Gamma ,X,\mu )$ the result follows. %we conclude $PC_{\mbox{\tiny{gen}}}(\Gamma ) = PC(\Gamma )$.

(2): Suppose that $C_{\mbox{\tiny{gen}}}(\Gamma )=r<\infty$. This means the generic $\bm{a}\in \mbox{FR}(\Gamma ,X,\mu )$ has $C(\bm{a})=r$. Since $r<\infty$ it follows from Corollary \ref{cor:PC=C} that $C(\bm{a})=r \Ra C(\bm{a})=PC(\bm{a})$. Thus the generic free action $\bm{a}$ satisfies $PC(\bm{a})= r= C(\bm{a})$ and by part (1) we therefore have $C(\Gamma )\geq PC(\Gamma ) = PC_{\mbox{\tiny{gen}}}=C_{\mbox{\tiny{gen}}}(\Gamma )\geq C(\Gamma )$, which shows that $C_{\mbox{\tiny{gen}}}(\Gamma ) =C(\Gamma )$.

(3) follows from (1) along with Corollary \ref{cor:PC=C}.
\end{proof}

% Corollary \ref{cor:AFrPC}, if $\bm{a}\prec \bm{b}$ and $\bm{a}$ is free then $PC(\bm{c})\leq PC(\bm{a})$.Suppose that $C_{\mbox{\tiny{gen}}}(\Gamma )<\infty$. Let $\bm{a}\in \mbox{FR}(\Gamma ,X,\mu )$ be a free action with $C(\bm{a})=C(\Gamma )$ \cite[Proposition 29.1]{KM05}. The generic element of $\mbox{FR}(\Gamma ,X,\mu )$ is universal in the sense that it weakly contains all other actions of $\Gamma$ \cite[Corollary 10.3 and Theorems 10.7 and 10.8]{Ke10}. It follows that there is a dense $G_\delta$ set $B\subseteq \mbox{FR}(\Gamma ,X,\mu )$ such that each $\bm{b}\in B$ is universal and satisfies $C(\bm{b})=C_{\mbox{\tiny{gen}}}(\Gamma )$. In particular, each $\bm{b}\in B$ satisfies $\bm{a}\prec \bm{b}$ and $C(\bm{b})<\infty$. By \ref{cor:weakcon} this implies $C(\bm{b})\leq C(\bm{a})$, so in fact $C(\bm{b})=C(\Gamma )=C(\bm{a})$ for all $\bm{b}\in B$. This shows that $C_{\mbox{\tiny{gen}}}(\Gamma ) = C(\Gamma )$.
%
Let $\mbox{MINCOST}(\Gamma ,X,\mu ) = \{ \bm{a} \in \mbox{FR}(\Gamma ,X,\mu ) \csuchthat C(\bm{a})=C(\Gamma )\}$. %The following may be seen as a generalization of the corresponding statements from \cite[Corollary 10.14]{Ke10}.

\begin{corollary}\label{cor:Gdelta}
Let $\Gamma$ be a countably infinite group. Then the set
\begin{align*}
D=\big\{\bm{b}\in \mbox{\emph{FR}}(\Gamma ,X,\mu ) \csuchthat &\exists\mbox{aperiodic subequivalence relations }\\
&E_0\subseteq E_1\subseteq E_2\subseteq \cdots \mbox{ of } E_b, \mbox{ with } E_b = \bigcup _n E_n \mbox{ and }\lim _n C_\mu (E_n)= C(\Gamma ) \big\}
\end{align*}
is dense $G_\delta$ in $A(\Gamma ,X,\mu )$. Additionally, if $C(\Gamma )<\infty$ then we have the equality of sets
\begin{equation}\label{eqn:mincost}
\mbox{\emph{MINCOST}}(\Gamma ,X,\mu ) = D\cap \{ \bm{b}\in \mbox{\emph{FR}}(\Gamma ,X,\mu )\csuchthat C (\bm{b})<\infty \} .
\end{equation}
In particular, if all free actions of $\Gamma$ have finite cost then $\mbox{\emph{MINCOST}}(\Gamma ,X,\mu )=D$ is dense $G_\delta$.
\end{corollary}

\begin{proof}
We begin by showing $D$ is dense $G_\delta$. By \cite[Theorem 10.8]{Ke10}, $\mbox{FR}(\Gamma ,X,\mu )$ is dense $G_\delta$ in $A(\Gamma ,X,\mu )$. If $C(\Gamma ) = \infty$ then $D=\mbox{FR}(\Gamma ,X,\mu )$ and we are done, so we may assume that $C(\Gamma )<\infty$. Then $C(\Gamma )=PC(\Gamma )$ by Proposition \ref{prop:analogue}, so it follows from Corollary \ref{cor:attain} that $D= \{ \bm{a}\in \mbox{FR}(\Gamma ,X,\mu )\csuchthat PC(\bm{a}) = PC(\Gamma ) \} = \mbox{MINPCOST}(\Gamma ,X,\mu )$, and therefore $D$ is dense $G_\delta$ by Corollary \ref{cor:Cgen}.

For the second statement of the theorem, suppose that $C(\Gamma )<\infty$. Then $C(\Gamma ) =PC(\Gamma )$ by Proposition \ref{prop:analogue}. The inclusion from left to right in (\ref{eqn:mincost}) is clear. If $\bm{b}$ has finite cost and $\bm{b}\in D$ then, $PC(\bm{b})\leq C(\Gamma )=PC(\Gamma )$, hence $PC(\bm{b})=PC(\Gamma ) = C(\Gamma )$, i.e., $\bm{b}\in \mbox{MINCOST}(\Gamma ,X,\mu )$.
\end{proof}

\subsection{Cost and invariant random subgroups}\label{sec:costIRS}

Equip each of the spaces $\Gamma ^\Gamma$ and $2^\Gamma$ with the pointwise convergence topology.
\begin{lemma}\label{lem:transv}
There exists a continuous assignment $\mbox{\emph{Sub}}_\Gamma \ra \Gamma ^\Gamma$, $H\mapsto \sigma _H$, with the following properties:
\begin{itemize}
\item[(i)] For each $H\in \mbox{\emph{Sub}}_\Gamma$, $\sigma _H : \Gamma \ra \Gamma$ is a selector for the right cosets of $H$ in $\Gamma$, i.e., $\sigma _H (\delta ) \in H\delta$ for all $\delta \in \Gamma$, and $\sigma _H$ is constant on each right coset of $H$.
\item[(ii)] $\sigma _H (h) =e$ whenever $h\in H$.
\item[(iii)] The corresponding assignment of transversals $\mbox{\emph{Sub}}_\Gamma \ra 2^\Gamma$, $H\mapsto T_H:=\sigma _H(\Gamma )$, is continuous.
\end{itemize}
\end{lemma}

\begin{proof}
Fix a bijective enumeration $\Gamma = \{ \gamma _m \} _{m\in \N}$ of $\Gamma$ with $\gamma _0=e$, and define $\sigma _H(\gamma _m)=\gamma _i$ where $i$ is least such that $\gamma _m \gamma _i^{-1} \in H$. This is continuous and (i) and (ii) are clearly satisfied, and (iii) follows from continuity of $H\mapsto \sigma _H$, since the map $\Gamma ^\Gamma \ra 2^\Gamma$ sending $f:\Gamma \ra \Gamma$ to its set of fixed points is continuous.
%Clearly continuous: $\sigma _H(\gamma _m)=\gamma _i$ iff $\gamma _m\gamma _i^{-1}\in H$ and for all $j<i$, $\gamma _m\gamma _i ^{-1}\not\in H$. Similarly, everything else is easy to check. Not that for (iii), the map $\Gamma ^\Gamma \ra 2^\Gamma$ sending $f$ to its set of fixed points is continuous.
\end{proof}

Define the set
\[
A(\mbox{Sub}_\Gamma ,X,\mu ) := \{ (H, \bm{a} ) \csuchthat H\in \mbox{Sub}_\Gamma , \mbox{ and }\bm{a}\in A(H ,X,\mu ) \} .
\]
This set has a natural Polish topology in which $(H_n, \bm{a}_n)\ra (H,\bm{a})$ if and only if $H_n\ra H$ and $\bm{a}_n\ra \bm{a}$ pointwise. We make this precise by taking $*$ to be some point isolated from $\mbox{Aut}(X,\mu )$ and then defining $\gamma ^b = *$ whenever $H\leq \Gamma$, $\bm{b}\in A(H,X,\mu )$, and $\gamma \not\in H$. Then $(H_n,\bm{a}_n)\ra (H,\bm{a})$ means that $\gamma ^{a_n} \ra \gamma ^a$ for every $\gamma \in \Gamma$. %This is the same as the topology inherited by identifying each $(H,a) \in A(\mbox{Sub}_\Gamma ,X,\mu )$ with $(H,a^*)\in \mbox{Sub}_\Gamma \times (\mbox{Aut}(X,\mu )\cup \{ * \} )^\Gamma$ where $a^*(\gamma ) = *$ if $\gamma\not\in H=\mbox{dom}(\bm{a})$, and where $*$ is an isolated point.
%
%Identifying $\Gamma$ with $\N$ via an injective enumeration $\Gamma = \{ \gamma _m \} _{m\in \N}$ gives an identification of each $H\leq \Gamma$ with either $\N$ or an initial segment of $\N$ via the map $n\mapsto h_n^H\in H$ where $h_n^H$ is the $n$-th element of $H$ appearing in $(\gamma _m )_{m\in \N}$. The set
%\[
%A(\mbox{Sub}_\Gamma ,X,\mu ) := \{ (H, \bm{a} ) \csuchthat H\in \mbox{Sub}_\Gamma , \mbox{ and }\bm{a}\in A(H ,X,\mu ) \}
%\]
%may then be identified with a subset of $\mbox{Sub}_\Gamma \times (\mbox{Aut}(X,\mu )\cup \{ \ast \} )^\N$ via the map $(H,\bm{a})\mapsto (H, \varphi _H(\bm{a}))$ where $\varphi _H(\bm{a})(n)= (h^H_n)^a$ if $h^H_n$ is defined (i.e., $n< |H|$) and $\varphi _H(\bm{a})(n)=\infty$. Treating $\ast$ as an isolated point, one checks that the image of this map is closed when , and so we equip $A(\mbox{Sub}_\Gamma ,X,\mu )$ with the Polish topology inherited by this identification.

\begin{lemma}\label{lem:meas}
For any $r\in \R$ the sets
\begin{align*}
S_r &= \{ H\in \mbox{\emph{Sub}}_\Gamma \csuchthat C(H)< r \} \\
A_r &= \{ (H,\bm{a})\in A(\mbox{\emph{Sub}}_\Gamma ,X,\mu ) \csuchthat \bm{a}\mbox{\emph{ is free and }}C(\bm{a}) < r \}
\end{align*}
are analytic. In particular, the map $H\mapsto C(H)$ is universally measurable.
\end{lemma}

\begin{proof}
%We adopt the notation of \cite[Proposition 18.1]{KM05}.
It suffices to show that $A_r$ is analytic since $S_r$ is the image of $A_r$ under projection onto $\mbox{Sub}_\Gamma$ which is continuous. We may assume that $X=2^\N$ and that $\mu$ is the uniform product measure.

Let $\Gamma \cc ^s X^\Gamma$ denote the left shift action given by $(\gamma^s\cdot f)(\delta ) = f(\gamma^{-1}\delta )$ for $f\in X^\Gamma$. Let $H\mapsto \sigma _H$ and $H\mapsto T_H\subseteq \Gamma$ be a continuous assignment of selectors and transversals given by Lemma \ref{lem:transv}. For $(H,\bm{a})\in A(\mbox{Sub}_\Gamma ,X,\mu )$ define the map $\Phi _{H,\bm{a}}:X \ra X^\Gamma$ by $\Phi _{H,\bm{a}}(x)(ht) = (h^{-1})^a x$ for $h\in H$, $t\in T_H$, $x\in X$. Then $\Phi _{H,\bm{a}}$ is injective and equivariant from $H\cc ^a X$ to the shift action $H\cc ^s X^\Gamma$ and so the measure $\mu _{H,\bm{a}}:= (\Phi _{H,\bm{a}})_*\mu$ is $H\cc ^s X^\Gamma$ invariant, and the systems $H\cc ^a (X,\mu )$ and $H\cc ^s (X^\Gamma ,\mu _{H,\bm{a}})$ are isomorphic. %In addition, for each $\gamma \in \Gamma$ the projection of $\mu (H,\bm{a})$ to the $\gamma$-coordinate is equal to $\mu$.
Let $P$ denote the space of Borel probability measures on $X^\Gamma$ equipped with the weak${}^*$-topology.
\begin{claim}
The map $A(\mbox{\emph{Sub}}_\Gamma ,X,\mu ) \ra P$, $(H,\bm{a})\mapsto \mu _{H,\bm{a}}$ is continuous. \end{claim}

\begin{proof}[Proof of claim]
Suppose that $(H_n, \bm{a}_n)\ra (H_\infty,\bm{a}_\infty)$ in $A(\mbox{Sub}_\Gamma ,X,\mu )$. Letting $\mu _n=\mu _{H_n,\bm{a}_n}$, it suffices to check that $\mu _n (A)\ra \mu _\infty (A)$ whenever $A\subseteq X^\Gamma$ is of the form $A=\{ f\in X^\Gamma \csuchthat \forall \gamma \in F \ (f(\gamma ) \in A_\gamma )\}$ where $F\subseteq \Gamma$ is finite and $A_\gamma \subseteq X$ is Borel. For $\gamma \in F$ write $\gamma = h_\gamma t_\gamma$ where $t_\gamma \in T_{H_\infty}$ and $h_\gamma \in H_\infty$. By continuity of $H\mapsto \sigma _H$ and $H\mapsto T_H$, for all large enough $n$, $h_\gamma \in H_n$ and $t_\gamma \in T_{H_n}$ for all $\gamma \in F$. Then $\mu _n(A) = \mu (\bigcap _{\gamma \in F}h_\gamma ^{a_n}(A_\gamma )) \ra \mu (\bigcap _{\gamma \in F}h_\gamma ^{a}(A_\gamma ))=\mu _\infty (A)$ since $\bm{a}_n\ra \bm{a}$.
\qedhere[Claim]
\end{proof}

\noindent Now let $E_H$ denote the orbit equivalence relation on $X^\Gamma$ generated by $H\cc ^s X^\Gamma$. The set
\[
B = \{ (H,\nu ) \in \mbox{Sub}_\Gamma \times P \csuchthat \nu \mbox{ is }E_H\mbox{-invariant} \mbox{ and } H\cc ^s (X^\Gamma ,\nu ) \mbox{ is essentially free} \}
\]
is Borel so by the proof of \cite[Proposition 18.1]{KM05} the set $D = \{ (H,\nu ) \in B \csuchthat C_\nu (E_H) <r \}$ is analytic. We have $(H,\bm{a})\in A_r$ if and only if $(H, \mu _{H,\bm{a}})\in D$, which shows that $A_r$ is analytic. %For any $H\leq \Gamma$ and any measure preserving action $\bm{a}$ of $H$ there exists a measure $\mu$ on $X= [0,1]^\Gamma$ such that $H\cc ^s (X,\mu )$ is isomorphic to $\bm{a}$ (see \cite[10.{\bf (B)}]{Ke10}). Therefore $C(H)<r$ if and only if there is some $\mu \in P(X)$ such that $(H,\mu )\in D$.  This condition is analytic.
\end{proof}

It follows that for any ergodic invariant random subgroup $\theta$ of $\Gamma$ there is an $r\in \R \cup \{\infty \}$ such that $C(H) = r$ for almost all $H\leq \Gamma$.  The following is an analogue of \cite[\S 5]{BG05} for cost. I would like to thank Lewis Bowen for a helpful discussion related to this.

\begin{theorem}\label{thm:gencost1}$\ $
Let $\theta$ be an invariant random subgroup of $\Gamma$ and suppose that $\theta$ concentrates on the infinite subgroups of $\Gamma$ which have infinite index in $\Gamma$. If $\theta (\{ H \csuchthat C(H) < \infty \} )\neq 0$ then $C(\Gamma ) =1$.

Thus, if $C(\Gamma ) >1$ then for any ergodic non-atomic m.p.\ action $\Gamma \cc ^a (X,\mu )$, either $\Gamma _x$ is finite almost surely, or $C(\Gamma _x) = \infty$ almost surely.
\end{theorem}

\begin{proof}
To see that the second statement follows from the first observe that an ergodic non-atomic m.p.\ action cannot have stabilizers which are finite index. We now prove the first statement. By decomposing $\theta$ into its ergodic components we may assume without loss of generality that $\theta$ is ergodic and there is an $r\in \R$ such that $C(H) < r$ almost surely.
%
%Identifying $\Gamma$ with $\N$ via an injective enumeration $\Gamma = \{ \gamma _n \} _{n\in \N}$ gives an identification of each infinite $H\leq \Gamma$ with $\N$ as well by taking $h_n$ to be the $n$-th element of $H$ appearing in $(\gamma _m )_{m\in \N}$. We may then identify $\{ (H, \bm{a} ) \csuchthat \bm{a}\in A(H ,X,\mu ) \}$ with a closed subset of $\mbox{Sub}_\Gamma \times \mbox{Aut}(X,\mu )^\N$.

By Lemma \ref{lem:meas} the set $A_r= \{ (H, \bm{a} ) \in A(\mbox{Sub}_\Gamma ,X,\mu )\csuchthat \bm{a}\mbox{ is free and }C(\bm{a})< r \}$ is an analytic subset of $A(\mbox{Sub}_\Gamma ,X,\mu )$. Since $C(H)<r$ almost surely, we may measurably select for each $H\in \mbox{Sub}_\Gamma$ a free action $\bm{a}_H \in \mbox{FR}(H ,X,\mu )\subseteq \mbox{A}(H , X,\mu )$ of $H$ such that almost surely $C(\bm{a}_H) < r$ (we are applying \cite[18.1]{Ke95} to the flip of the graph of the projection function $A_r\ra \mbox{Sub}_\Gamma$, $(H,\bm{a})\mapsto H$). A co-inducing process can now be used to obtain an action $\bm{b}$ of $\Gamma$ from the selection $H\mapsto \bm{a}_H \in A(H ,X,\mu )$ as follows.

Let $H\mapsto \sigma _H$ be as in Lemma \ref{lem:transv}. Let $\mbox{COS}_\Gamma \subseteq 2^\Gamma$ denote the closed subspace of all right cosets of subgroups of $\Gamma$, on which $\Gamma$ acts continuously by left translation $\gamma ^\ell \cdot H\delta = \gamma H \delta$. The function $\rho : \Gamma \times \mbox{COS}_\Gamma  \ra \Gamma$ defined by
\[
\rho (\gamma ,H\delta ) = (\sigma _{\gamma H\gamma ^{-1}}(\gamma \delta ))^{-1}\gamma \sigma _H(\delta )
\]
is a continuous cocycle of this action with values in $\Gamma$. %Note that if $H_n\delta _n\ra H\delta$ then $H_n\delta _n \delta ^{-1} \ra H$ so for all large enough $n$, $e\in H_n\delta _n\delta ^{-1}$, i.e., $H_n\delta _n\delta ^{-1}= H$. Thus $H_n\ra H$.
It is clear that $\rho (\gamma ,H\delta )\in \delta ^{-1}H\delta$, so the map $(\gamma ,H\delta )\mapsto \rho (\gamma ,H\delta )^{a_{\delta ^{-1}H\delta}}$ is a well-defined measurable cocycle with values in $\mbox{Aut}(X,\mu )$. We therefore obtain an action $b$ of $\Gamma$ on the space $W = \{ (H,f)\csuchthat H\leq \Gamma \mbox{ and }f:H\backslash \Gamma \ra X \}$ given by $\gamma ^b(H,f ) = (\gamma H\gamma ^{-1}, \gamma ^{b_H}f)$ where $\gamma ^{b_H}f : \gamma H \gamma ^{-1}\backslash \Gamma \ra X$ is given by
\[
(\gamma ^{b_H}f)(\gamma H\delta ) = \rho (\gamma ,H\delta )^{a_{\delta ^{-1}H\delta}}(f(H\delta )) .
\]
This action preserves the measure $\kappa = \int _H (\updelta _H \times \mu ^{H\backslash \Gamma}) \, d\theta (H)$ since
\begin{align*}
\gamma ^b _* \kappa &= \int _H (\updelta _{\gamma H \gamma ^{-1}}\times \gamma ^{b_H}_*\mu ^{H\backslash \Gamma} ) \, d\theta
= \int _H \Big( \updelta _{\gamma H \gamma ^{-1}} \times \prod _{\gamma H\delta \in \gamma H \gamma ^{-1}\backslash \Gamma} (\rho (\gamma ,H\delta )^{a_{\delta ^{-1}H\delta}})_*\mu \Big) \, d\theta \\
&= \int _H \Big(\updelta _{\gamma H\gamma ^{-1}}\times \prod _{\gamma H\delta \in \gamma H \gamma ^{-1}\backslash \Gamma} \mu \Big) \, d\theta
=\int _H\big(\updelta _{\gamma H\gamma ^{-1}}\times \mu ^{\gamma H\gamma ^{-1}\backslash \Gamma}\big) \, d\theta = \int _H \updelta _H \times \mu ^{H\backslash \Gamma}d\theta = \kappa .
\end{align*}

\begin{lemma}\label{lem:free}$\ $
\begin{enumerate}
\item For each $(H, f)\in W$, and $h\in H$ we have $(h^{b_H}f)(H) = h ^{a_H}(f(H))$ and thus the map $X^{H\backslash \Gamma} \ra X$, $f\mapsto f(H)$ factors
    \[
    \bm{b}_H = H \cc ^{b_H}(X^{H\backslash \Gamma},\mu ^{H\backslash \Gamma} )
    \]
    onto $\bm{a}_H$.

\item (Analogue of \cite[Lemma 2.1]{Io11}) For almost all $H\leq \Gamma$ and every $\gamma \in \Gamma \setminus \{ e \}$ the sets
\[
W^H_\gamma = \{ f \in X^{H\backslash \Gamma}\csuchthat \gamma H\gamma ^{-1} = H \mbox{ and }(\gamma ^{b_H}f)(H) = f(H) \}
\]
are $\mu ^{H\backslash \Gamma}$-null. In particular, $\bm{b}$ is essentially free.
\end{enumerate}
\end{lemma}

\begin{proof}
(1) is clear from the definition of $b_H$. For (2), If $f\in W^H_\gamma$ then $\rho (\gamma ,H)^{a_H} (f(H\gamma ^{-1})) = f(H)$ by definition of $b_H$. So for each $H$ with $\bm{a}_H$ essentially free, if $\gamma \in H\setminus \{ e \}$ then $f\in W^H_\gamma$ if and only if $\gamma ^{a_H}(f(H))=f(H)$, so that $W^H_\gamma$ is null, while if $\gamma \in \Gamma \setminus H$ then $W^H_\gamma \subseteq \{ f \in X^{H\backslash \Gamma}\csuchthat \rho (\gamma ,H)^{a_H} (f(H\gamma ^{-1})) = f(H) \}$, which is null since $H\gamma ^{-1}\neq H$ and $\mu$ is non-atomic. Since almost all $\bm{a}_H$ are essentially free we are done. \qedhere[Lemma \ref{lem:free}]
\end{proof}

We now apply a randomized version of an argument due to Gaboriau (see \cite[Theorem 35.5]{KM05}). There is another measure preserving action $\bm{s}= \Gamma \cc ^s (W,\kappa )$ of $\Gamma$ on $(W,\kappa )$ given by $\gamma ^s (H,f) = (\gamma H \gamma ^{-1} , \gamma ^{s_H} f)$ where $(\gamma ^{s_H}f)(\gamma H\delta ) = f(H\delta )$ (this is the random Bernoulli shift determined by $\theta$ \cite[\S 5.3]{T-D12a}). The projection map $W \ra \mbox{Sub}_\Gamma$, $(H,f) \mapsto H$ factors both $\bm{b}$ and $\bm{s}$ onto $\bm{\theta}$. We let $\bm{a}$ denote the corresponding relatively independent joining of $\bm{b}$ and $\bm{s}$ over $\bm{\theta}$, i.e., $\bm{a}$ is the measure preserving action of $\Gamma$ on
\[
(Z,\eta) = \big( \big\{ (H,f,g) \csuchthat f,g\in X ^{H\backslash \Gamma} \big\} , \ \int _H (\updelta _H \times \mu ^{\Gamma /H}\times \mu ^{\Gamma /H}) \, d\theta \big)
\]
given by $\gamma ^a (H,f,g) = (\gamma H \gamma ^{-1}, \gamma ^{b_H}f, \gamma ^{s_H}g)$ where $(\gamma ^{s_H}g)(\gamma H\delta )=g(H\delta )$. This action is free since it factors onto $\bm{b}$.

Let $p : Z \ra W$ denote the projection map $p((H,f,g))= (H,g)$. For each $(H,g)\in W$ the set $p^{-1}((H,g))$ is $a\resto H$-invariant, and we let $E_{(H,g)}$ denote the orbit equivalence relation on $p^{-1} ((H,g))$ generated by $\bm{a}\resto H$, i.e., $(H,f_1,g)E_{(H,g)}(H,f_2,g)$ if and only if there is some $h\in H$ such that $h^{b_H}f_1 = f_2$. Define the equivalence relation $E$ on $Z$ by $E = \bigsqcup _{(H,g)\in W} E_{(H,g)}$, i.e.,
\[
(H_1,f_1,g_1)E(H_2,f_2,g_2) \IFF (H_1 , g_1) = (H_2, g_2) \mbox{ and } \exists h\in H_1 \, (h^{b_H}f_1= f_2 ).
\]
Recall that if $F \subseteq R$ are countable Borel equivalence relations on a standard Borel space $Y$, then $F$ is said to be \emph{normal} in $R$ if there exists some countable group $\Delta$ of Borel automorphisms of $Y$ which generates $R$ and satisfies $xFy \Ra \delta (x) F\delta (y)$ for all $\delta \in \Delta$.
\begin{lemma}\label{lem:normsub-eq}
$E$ is a normal subequivalence relation of $E_a$ that is almost everywhere aperiodic and with $C_\eta (E) < r$.
\end{lemma}

\begin{proof}[Proof of Lemma \ref{lem:normsub-eq}]
It is clear that $E$ is an equivalence relation and that $E$ is contained in $E_{a}$. Also, $E$ is almost everywhere aperiodic since $\theta$ concentrates on the infinite subgroups of $\Gamma$ by hypothesis. Let $\gamma \in \Gamma$ and let $(H,f,g), (H,f',g)\in \mbox{Sub}_\Gamma \times X$ be $E$-related so that $h^{b_H}f = f'$ for some $h\in H$. To show $E$ is normal in $E_a$ we must show that $\gamma ^{a} (H,f,g)$ and $\gamma ^a(H,h^{b_H}f,g)$ are $E$ related as well, i.e., we must find some $k\in \gamma H\gamma ^{-1}$ such that $(k\gamma )^{b_H} f_1 = \gamma ^{b_H}(h^{b_H}f_1)$. The element $k = \gamma h \gamma ^{-1}$ works.

If we disintegrate $\eta$ via the $E$-invariant map $p:Z\ra W$, then for each $(H,g)\leq \Gamma$, the equivalence relation $E_{(H,g)}$ on $(p^{-1}((H,g)),\eta _{(H,g)})$  is isomorphic to the orbit equivalence relation generated by $b_H\resto H$ on $(X^{H\backslash \Gamma} ,\mu ^{H\backslash \Gamma})$. By Lemma \ref{lem:free}.(1), $\bm{b}_H$ factors onto $\bm{a}_H$, so for $\theta$-almost every $H$ we have $r\leq C_{\eta _{(H,g)}}(E_{(H,g)}) = C(\bm{b}_H ) \leq C(\bm{a}_H) < r$ by \cite[bottom of p. 78]{Ke10}. Then by \cite[Proposition 18.4]{KM05} we have
\[
C_{\eta} (E) = \int _{H,g} C _{\eta _{(H,g)}} ( E_{(H,g)}) \, d\theta (H) < r .
\qedhere[\mbox{Lemma } \ref{lem:normsub-eq}]
\]
\end{proof}

Since $H$ is almost surely infinite index, the equivalence relation $E_s$ on $W$ generated by $\bm{s}$ is aperiodic. By \cite{Ke10} the full group $[E_s]$ contains an aperiodic transformation $T: W\ra W$. Let $B: \Gamma \ra \mbox{MALG}_\kappa$, $\gamma \mapsto B_\gamma$, be a partition of $W$ such that $T\resto B_\gamma = \gamma ^s \resto B_\gamma$. Then $A: \Gamma \ra \mbox{MALG}_\kappa$ given by $A_\gamma = p^{-1} (B_\gamma )$ is a partition of $Z$, and determines the L-graphing $\Phi ^{a,A} = \{ \varphi ^{a,A}_\gamma \} _{\gamma \in \Gamma}$ where $\varphi ^{a,A}_\gamma \resto A_\gamma = \gamma ^a \resto A_\gamma$.

Fix $\epsilon >0$ and find by Lemma \ref{lem:normsub-eq} a graphing $\{ \varphi _i \}_{i\in \N}$ of $E\subseteq Z$ of finite cost $\sum _i C_\eta (\varphi _i) <\infty$. Let $M$ be so large that $\sum _{i>M} C_\eta ( \varphi _i) < \epsilon /2$. Let $Y_0\subseteq W$ be a Borel complete section for $E_T$ with $\kappa (Y_0)<\epsilon /(2M)$, and let $Y= p^{-1}(Y_0)$. Then $\eta (Y)=\kappa (Y_0)<\epsilon /M$, and $Y$ is $E$-invariant so that $\{ \varphi _i\resto Y \} _{i\in \N}$ is an L-graphing of $E\resto Y$. It follows that
\[
C_{\eta\resto Y}(E\resto Y) \leq \sum _{i\in \N} C_{\eta} (\{ \varphi _i \resto Y\} )
\leq M\cdot \eta (Y) + \sum _{i\geq M}C_\eta ( \{ \varphi _i \} ) < \epsilon.
\]
\begin{claim}\label{claim:vee}
$E\subseteq E\resto Y \vee E_{\Phi ^{a,A}}$.
\end{claim}

\begin{proof}
Suppose $(H,f,g)E(H,f',g)$. Since $Y_0$ is a complete section for $E_T$ there exists $\gamma _1,\dots ,\gamma _k$ and $\epsilon _1,\dots ,\epsilon _k \in \{ -1,1 \}$ such that $(\varphi ^{s,B}_{\gamma _k})^{\epsilon _k}\circ \cdots (\circ \varphi ^{s,B}_{\gamma _1})^{\epsilon _1} ((H,g)) \in Y_0$. Let $\gamma = \gamma _k ^{\epsilon _k}\cdots \gamma _1 ^{\epsilon _1}$ and let $(H_0,g_0)=\gamma ^s ((H,g))\in Y_0$. It follows that
\begin{align*}
\gamma ^a (H,f,g) &= (\gamma _k ^{\epsilon _k})^a\cdots (\gamma _1 ^{\epsilon _1})^a(H,f,g) = (\varphi ^{a,A}_{\gamma _k})^{\epsilon _k}\circ\cdots \circ (\varphi ^{a,A}_{\gamma _1})^{\epsilon _1} (H,f,g)  \\
\gamma ^a (H,f',g) &= (\gamma _k ^{\epsilon _k})^a\cdots (\gamma _1 ^{\epsilon _1})^a(H,f',g) = (\varphi ^{a,A}_{\gamma _k})^{\epsilon _k}\circ\cdots \circ (\varphi ^{a,A}_{\gamma _1})^{\epsilon _1} (H,f',g).
\end{align*}
This shows that $(H,f,g)E_{\Phi ^{a,A}}\gamma ^a (H,f,g)$ and $\gamma ^a (H,f',g) E_{\Phi ^{a,A}}(H,f',g)$. As $\gamma ^a(H,f,g) = (H_0,\gamma ^{b_H}f,g_0)\in Y$ and $\gamma ^a(H,f',g)= (H_0,\gamma ^{b_H}f',g_0)\in Y$ we will be done if we can show these two points are $E$-related. Let $h\in H$ be such that $h^{b_H}f = f'$ and let $k=\gamma h\gamma ^{-1}$.  Then $k\in \gamma H \gamma ^{-1} = H_0$ and
\[
k^a(H_0, \gamma ^{b_H}f,g_0) = (k\gamma )^a(H, f, g)= (\gamma h)^a(H,f,g) = \gamma ^a (H,f',g) = (H_0 ,\gamma ^{b_H}f',g_0)
\]
which shows that $(H_0,\gamma ^{b_H}f,g_0)E_{(H_0,g_0)}(H_0,\gamma ^{b_H}f',g_0)$. \qedhere[Claim \ref{claim:vee}]
\end{proof}

We have $C_\eta (E\resto Y \vee E_{\Phi ^{a,A}} ) \leq 1 + \epsilon$. Since we have shown that $E \subseteq E\resto Y \vee E_{\Phi ^{a,A}}$ and that $E$ is an aperiodic normal subequivalence relation of $E_a$, it follows from \cite[24.10]{KM05} that $C_\eta (E_a)\leq C_\eta (E\resto Y\vee E_{\Phi ^{a,A}}) \leq 1+\epsilon$. As $\epsilon >0$ was arbitrary it follows that $C_\eta (E_a)=1$ and therefore $C(\Gamma ) =1$.
\end{proof}

\subsection{Fixed price 1 and shift-minimality}\label{sec:fp1sm}

The following lemma will be needed for Theorem \ref{thm:smfp1}.

\begin{lemma}\label{lem:cost1}
Let $\theta$ be an invariant random subgroup of a countable group $\Gamma$ that concentrates on the infinite amenable subgroups of $\Gamma$. Let $\bm{a} = \Gamma \cc ^a (X,\mu )$ be a free measure preserving action of $\Gamma$ and let
\[
\bm{\theta \times a} = \Gamma \cc ^{c\times a} (\mbox{\emph{Sub}}_\Gamma \times X, \theta \times \mu )
\]
be the product $\Gamma$-system. Then $C_{\theta \times \mu}(E_{c\times a}) = 1$.
\end{lemma}

\begin{remark}
The proof shows that the hypothesis that $\theta$ is amenable can be weakened to the hypothesis that $\theta$ concentrates on groups of fixed price $1$.
\end{remark}

\begin{proof}
The proof is similar to that of Lemma \ref{lem:normsub-eq}. %The action $c\times a$ is given by
%\[
%\gamma ^{c\times a}\cdot (H ,x ) = (\gamma H \gamma ^{-1} , \gamma ^a x )
%\]
%for $\gamma \in \Gamma$ and $(H ,x) \in \mbox{Sub}_\Gamma \times X$.
Since $E_{c\times a}$ is aperiodic it suffices to show that $C_{\theta \times \mu}(E_{c\times a})\leq 1$. For each $H\in \mbox{Sub}_\Gamma$ let $E_{a\resto H}$ denote the orbit equivalence relation on $X$ generated by $\bm{a}\resto H = H\cc ^a (X,\mu )$. Define the subrelation $E \subseteq E_{c\times a}$ on $\mbox{Sub}_\Gamma \times X$ by $E = \{ ((H,x),(H,y))\csuchthat xE_{a\resto H} y\}$, i.e., %$E = \bigcup _{H\in \mbox{\tiny{Sub}}_\Gamma} \{ (H,H) \} \times E_{a\resto H}$, i.e.,
\[
(H, x) E (L,y) \ \IFF \ H=L \mbox{ and }(\exists h\in H )\, (h^a \cdot x = y ) .
\]
Then $E$ is a normal sub-equivalence relation of $E_{c\times a}$. Since $\theta$ concentrates on the infinite subgroups of $\Gamma$, $E$ is aperiodic on a $(\theta\times \mu )$-conull set. By \cite[24.10]{KM05} and then \cite[Proposition 18.4]{KM05} we therefore have
\[
C_{\theta \times \mu} (E_{c\times a}) \leq C_{\theta \times \mu} (E)= \int _{H} C_\mu (E_{a\resto H} ) \, d\theta (H) = 1
\]
%The projection map $\pi : \mbox{Sub}_\Gamma \times X \ra \mbox{Sub}_\Gamma$, $(H, x)\mapsto H$, is an $E$-invariant Borel map with $\pi _*(\theta \times \mu )= \theta$. By \cite[Proposition 18.4]{KM05} and the definition of product measure we have
%\[
%C_{\theta \times \mu} (E_{c\times a}) \leq C_{\theta \times \mu} (E) = \int _{H} C_\mu %(E_{a\resto H} ) \, d\theta (H) = 1,
%\]
where the last equality follows from \cite[Corollary 31.2]{KM05} since $\theta$-almost every $H$ is infinite amenable.
\end{proof}

\begin{theorem}\label{thm:smfp1}
Let $\Gamma$ be a countably infinite group that contains no non-trivial finite normal subgroup. If $\Gamma$ is not shift-minimal then $\Gamma$ has fixed price $1$.
\end{theorem}

\begin{proof}
Suppose that $\Gamma$ is not shift-minimal. %If $\Gamma$ is amenable then it has fixed price $1$ \cite[Corollary 31.2]{KM05}, so we may assume that $\Gamma$ is non-amenable so that $\bm{s}_\Gamma$ is strongly ergodic.
By Corollary \ref{cor:notsm} either $\Gamma$ has a non-trivial normal amenable subgroup $N$ that is necessarily infinite by our hypothesis on $\Gamma$, or there is an infinitely generated amenable invariant random subgroup $\theta$ of $\Gamma$ that is weakly contained in $\bm{s}_\Gamma$. In the first case define $\theta = \updelta _N$, so that in either case $\theta$ concentrates on the infinite amenable subgroups of $\Gamma$, and $\bm{\theta}\prec \bm{s}_\Gamma$.

Let $(X,\mu )$ denote the underlying measure space of $\bm{s}_\Gamma$ and consider the product $\Gamma$-system
\[
\bm{\theta} \times \bm{s}_\Gamma = \Gamma \cc ^{c\times s} (\mbox{Sub}_\Gamma \times X, \theta \times \mu ) .
\]
By Lemma \ref{lem:cost1} we have $C(\bm{\theta}\times \bm{s}_\Gamma ) =1$ . The action $\bm{\theta}$ is weakly contained in $\bm{s}_\Gamma$, so $\bm{\theta}\times \bm{s}_\Gamma$ is weakly equivalent to $\bm{s}_\Gamma$.  This implies that $\Gamma$ has fixed price $1$ by (3)$\Ra$(1) of Corollary \ref{cor:FP1}.
\end{proof}

\begin{corollary}\label{cor:equivFP1}
Suppose that $\Gamma$ does not have fixed price $1$. Then the following are equivalence
\begin{enumerate}
\item $\Gamma$ is shift-minimal.
\item $\Gamma$ contains no non-trivial finite normal subgroups.
\item $\mbox{\emph{AR}}_\Gamma$ is trivial.
\end{enumerate}
\end{corollary}

\begin{proof}
%(2)$\Ra$(3): Assume that (2) holds. If $\mbox{AR}_\Gamma$ were infinite than $\Gamma$ would have fixed price $1$ by \cite[Proposition 35.2]{KM}, which is not true by hypothesis. Therefore $\mbox{AR}_\Gamma$ is a finite normal subgroup of $\Gamma$, so is trivial by (2).
%
%(3)$\Ra$(1): This is immediate from \ref{thm:smfp1} by our assumption that $\Gamma$ does not have fixed price $1$.
(3)$\Ra$(2) is obvious. (2)$\Ra$(1) is immediate from Theorem \ref{thm:smfp1} by our assumption that $\Gamma$ does not have fixed price $1$. (1)$\Ra$(3) holds in general with no assumptions on $\Gamma$.
\end{proof}

\begin{corollary}\label{cor:modARsm}
Let $\Gamma$ be any group that does not have fixed price $1$. Then $\mbox{\emph{AR}}_\Gamma$ is finite and $\Gamma /\mbox{\emph{AR}}_\Gamma$ is shift-minimal.
\end{corollary}

\begin{proof}
Any group containing an infinite normal amenable subgroup has fixed price $1$ \cite[Proposition 35.2]{KM05}. Therefore $N=\mbox{AR}_\Gamma$ is finite. Let $\bm{a}=\Gamma \cc ^a (X,\mu )$ be a free measure preserving action of $\Gamma$ of cost $C_\mu (E_{\bm{a}})>1$. The measure preserving action $\bm{b}$ of $\Gamma /N$ on the ergodic components of $\bm{a}\resto N$ is free, and since $N$ is finite we have $C(\bm{b}) \geq C(\bm{a})>1$. Thus, $\Gamma /N$ does not have fixed price 1, and $\mbox{AR}_{\Gamma /N} = \{ e \}$ by Proposition \ref{prop:ARbasic}. Corollary \ref{cor:equivFP1} now shows that $\Gamma /N$ is shift-minimal.
\end{proof}

\section{Questions}\label{sec:Questions}

\subsection{General implications} A countable group $\Gamma$ is called \emph{$C^*$-simple} if the reduced $C^*$-algebra of $\Gamma$ is simple, i.e., $C^*_r(\Gamma )$ has no non-trivial closed two-sided ideals. As observed in the introduction, there is a strong parallel between shift-minimality and $C^*$-simplicity. The following characterization of $C^*$-simplicity of a countable group $\Gamma$ may be found in \cite{Ha07}. Let $\lambda _\Gamma$ denote the left-regular representation of $\Gamma$ on $\ell ^2 (\Gamma )$.

\begin{proposition}\label{prop:CS}
Let $\Gamma$ be a countable group. Then $\Gamma$ is $C^*$-simple if and only if $\pi \prec \lambda _\Gamma$ implies $\pi \sim \lambda _\Gamma$ for all nonzero unitary representations $\pi$ of $\Gamma$.
\end{proposition}

In this characterization of $C^*$-simplicity we may actually restrict our attention to irreducible representations of $\Gamma$. That is, $\Gamma$ is $C^*$-simple if and only if every irreducible unitary representation $\pi$ of $\Gamma$ that is weakly contained in $\lambda _\Gamma$ is actually weakly equivalent to $\lambda _\Gamma$. See \cite{BH00}. See also \cite[Appendix F]{BHV08} and \cite{Di77} for more on weak containment of unitary representations.

Characterization (6) of shift-minimality from Proposition \ref{prop:minimal} also has an analogue for $C^*$-simplicity. Let $\Es{H}$ be an infinite dimensional separable Hilbert space and let $\mbox{Irr}_\lambda (\Gamma , \Es{H})$ denote the Polish space of irreducible representation of $\Gamma$ on $\Es{H}$ that are weakly contained in $\lambda _\Gamma$ (see \cite{Di77}). Let $\Es{U}(\Es{H})$ be the Polish group of all unitary operators on $\Es{H}$. Then $\Gamma$ is $C^*$-simple if and only if $\Gamma$ is ICC and the conjugation action of $\Es{U}(\Es{H})$ on $\mbox{Irr}_\lambda (\Gamma ,\Es{H})$ is minimal (i.e., every orbit is dense). See \cite[Appendix H.{\bf (C)}]{Ke10}.

Consider now the following properties of a countable group $\Gamma$:
\begin{enumerate}
\item[(UT)] $\Gamma$ has the unique trace property.
\item[(CS)] $\Gamma$ is $C^*$-simple.
\item[(SM)] $\Gamma$ is shift-minimal.
\item[(UIRS${}_0$)] $\Gamma$ has no non-trivial amenable invariant random subgroup that is weakly contained in $\bm{s}_\Gamma$. %The only amenable invariant random subgroup of $\Gamma$ that is weakly contained in $\bm{s}_\Gamma$ is the point mass $\updelta _{\langle e \rangle}$ on the trivial subgroup.
\item[(UIRS)] $\Gamma$ has no non-trivial amenable invariant random subgroups.%The only amenable invariant random subgroup of $\Gamma$ is the point mass $\updelta _{\langle e \rangle}$ on the trivial subgroup.
\item[($\mbox{AR}_e$)] $\Gamma$ has no non-trivial amenable normal subgroups, i.e., the amenable radical $\mbox{AR}_\Gamma$ of $\Gamma$ is trivial.
\end{enumerate}

All of the known implications (besides (SM)$\IFF$(UIRS${}_0$)) are depicted in Figure \ref{fig:rel} in the introduction. It is known that (UT) and (CS) imply ($\mbox{AR}_e$) (\cite{PS79}, see also \cite[Proposition 3]{BH00}), though it is an open question whether there are any other implications among the properties (UT), (CS), and ($\mbox{AR}_e$) in general \cite{BH00}. The following questions concern some of the remaining implications.% and may help shed some light on the situation.

The implication (UT)$\Ra$(SM) was shown in Theorem \ref{thm:UTsm}. One of the most pressing questions is:

\begin{question}\label{Q:CS->SM}
Does (CS) imply (SM)? That is, are $C^*$-simple groups shift-minimal?
\end{question}

For a positive answer to Question \ref{Q:CS->SM} it would suffices by Corollary \ref{cor:notsm} to show that if $\theta$ is a non-atomic self-normalizing amenable IRS of a countable group $\Gamma$ that is weakly contained in $\bm{s}_\Gamma$ then the tracial state on $C^*_r(\Gamma )$ extending $\varphi _\theta$ from the proof of Theorem \ref{thm:UTUIRS} is not faithful.

The implication from (UT) to (UIRS) is quite direct. The converse would mean that a tracial state on $C^*_r(\Gamma )$ different from $\tau _\Gamma$ somehow gives rise to a non-trivial amenable invariant random subgroup of $\Gamma$. This is addressed by the following question:

\begin{question}\label{Q:UIRS->UT}
Does (UIRS) imply (UT)? That is, if $\Gamma$ does not have any non-trivial amenable invariant random subgroups then does $C^*_r(\Gamma )$ have a unique tracial state?
\end{question}

We know from Theorem \ref{thm:smequiv} that (SM) and (UIRS${}_0$) are equivalent. The equivalence of (SM) and (UIRS) is open however (clearly though (UIRS)$\Ra$(UIRS${}_0$))

\begin{question}\label{question:UIRS}
Does (UIRS${}_0$) imply (UIRS)?
\end{question}

To obtain a positive answer to Question \ref{question:UIRS} it would be enough to show the following: ($\star$) Every ergodic amenable invariant random subgroup of a countable group $\Gamma$ that is not almost ascendant is weakly contained in $\bm{s}_\Gamma$.

Indeed, assume that ($\star$) holds and suppose that $\Gamma$ does not have (UIRS), i.e., there is an amenable invariant random subgroup $\theta$ of $\Gamma$ other than $\updelta _{\langle e\rangle}$. By moving to an ergodic component of $\theta$ we may assume without loss of generality that $\theta$ is ergodic. If $\theta$ is not almost ascendant then ($\star$) implies that $\bm{\theta}$ is weakly contained in $\bm{s}_\Gamma$, which shows that $\Gamma$ does not have (UIRS${}_0$). On the other hand, if $\theta$ is almost ascendant then, by Corollary \ref{cor:aaAR}, $\theta$ concentrates on the subgroups of $\mbox{AR}_\Gamma$, and in particular $\mbox{AR}_\Gamma$ is non-trivial, so $\updelta _{\mbox{\scriptsize{AR}}_\Gamma}$ witnesses that $\Gamma$ does not have (UIRS${}_0$).

The implication (SM)$\Ra$($\mbox{AR}_e$) is shown in Proposition \ref{prop:noamensub} above. The converse is a tantalizing question:

\begin{question}\label{question:AReSM}
Does ($\mbox{AR}_e$) imply (SM)? That is, if $\Gamma$ has no non-trivial amenable normal subgroup then is every non-trivial m.p.\ action that is weakly contained in $\bm{s}_\Gamma$ free?
\end{question}

To obtain a positive answer to Question \ref{question:AReSM} by Corollary \ref{cor:notsm} it would be enough to show that if $\theta$ is a non-atomic self-normalizing invariant random subgroup weakly contained in $\bm{s}_\Gamma$ then $\theta$ concentrates on subgroups of the amenable radical of $\Gamma$. (Note that $\theta$ does indeed concentrate on the amenable subgroups of $\Gamma$ by NA-ergodicity.)

\subsection{Cost and pseudocost} In the infinitely generated setting it appears that pseudocost, rather than cost, may be a more useful way to define an invariant. In addition to the properties exhibited in \S\ref{sec:infgen}, pseudocost enjoys many of the nice properties already known to hold for cost. For instance, pseudocost respects ergodic decomposition, and $PC(\Gamma )\leq PC(N)$ whenever $N$ is an infinite normal subgroup of $\Gamma$. (The proofs are routine: for the first statement one uses the corresponding fact about cost along with basic properties of pseudocost, and the proof of the second is nearly identical to the corresponding proof for cost.)

\begin{question}\label{Q:unions}%[Question \ref{Q:unions}]
Is there an example of a m.p.\ countable Borel equivalence relation $E$ such that $PC_\mu (E) < C_\mu (E)$?
\end{question}

By Corollary \ref{cor:PC=C}.(1) the equality $PC_\mu (E)=C_\mu (E)$ holds whenever $C_\mu (E)<\infty$, so the question is whether it is possible to have $PC_\mu (E)<\infty$ and $C_\mu (E)=\infty$. Equivalently: does there exist an increasing sequence $E_0\subseteq E_1\subseteq \cdots$, of m.p.\ countable Borel equivalence relations on $(X,\mu )$ with $\sup _n C_\mu (E_n) <\infty$ and $C_\mu (\bigcup _n E_n ) =\infty$? If such a sequence $(E_n)_{n\in \N}$ exists then, letting $E=\bigcup _n E_n$, Corollary \ref{cor:PC=C}.(2) implies that $E$ could not be treeable. In addition, $E$ would provide an example of strict inequality $\beta _1(E) +1 < C_\mu (E)$. This follows from \cite[5.13, 3.23]{Ga02}. Gaboriau has shown that any aperiodic m.p.\ countable Borel equivalence $R$ satisfies $\beta _1 (R)+1 \leq C_\mu (R)$  \cite{Ga02}, although it is open whether this inequality can ever be strict. Note that a positive answer to \ref{Q:unions} would not necessarily provide a counterexample to the fixed price conjecture, even if the equivalence relation $E$ comes from a free action of some group $\Gamma$; at this time there is no way to rule out the possibility that such a $\Gamma$ has fixed cost $\infty$ while at the same time admitting various free actions with finite pseudocost.
%since %$\beta _1(E) +1 \leq \liminf _n \beta _1(E_n) +1 \leq \liminf _n C_\mu (E_n) < \infty = C_\mu (E)$

\begin{question}\label{Q:maxcost} 
Suppose that a countable group $\Gamma$ has some free action $\bm{a}$ with $C_\mu (\bm{a} )=\infty$. Does it follow that $C_\mu (\bm{s}_\Gamma )=\infty$?
\end{question}

By Corollary \ref{cor:AFrPC}, $\bm{s}_\Gamma$ attains the maximum pseudocost among free actions of $\Gamma$. Corollary \ref{cor:weakcon} implies that 
\[
C(\bm{s}_\Gamma )\geq \sup \{ C(\bm{b})\csuchthat \bm{b}\in \mbox{FR}(\Gamma ,X,\mu ) \mbox{ and either }C(\bm{b})<\infty \mbox{ or }E_b \mbox{ is treeable}\} .
\]
This is not enough to conclude that $\bm{s}_\Gamma$ always attains the maximum cost among free actions of $\Gamma$. A positive answer to Question \ref{Q:maxcost} would imply that $\bm{s}_\Gamma$ always attains this maximum cost.

It would be just as interesting if $\bm{s}_\Gamma$ could detect whether $C(\Gamma )<\infty$.

\begin{question}\label{Q:fincost}
Suppose that a countable group $\Gamma$ has some free action $\bm{a}$ with $C_\mu (\bm{a})<\infty$. Does it follow that $C_\mu (\bm{s}_\Gamma )<\infty$?
\end{question}

At this time it appears that one cannot rule out any combination of answers to Questions \ref{Q:maxcost} and \ref{Q:fincost}. A positive answer to both questions would amount to showing that no group has both free actions of infinite cost and free actions of finite cost -- this would essentially affirm a special case of the fixed price conjecture!

\subsection{Other questions}  It is shown in \cite{T-D12b} that the natural analogue of Question \ref{question:AReSM}, where "amenable" is replaced by "finite" and "weakly contained in" is replaced by "is a factor of," has a positive answer:

\begin{theorem}[Corollary 1.6 of \cite{T-D12b}]\label{thm:Cor1.6}
Let $\Gamma$ be a countable group. If $\Gamma$ has no non-trivial finite normal subgroups then every non-trivial totally ergodic action of $\Gamma$ is free.

In particular, if $\Gamma$ has no non-trivial finite normal subgroups then every non-trivial factor of $\bm{s}_\Gamma$ is free.
\end{theorem}

Here, a measure preserving action of $\Gamma$ is called \emph{totally ergodic} if all infinite subgroups of $\Gamma$ act ergodically. Theorem \ref{thm:Cor1.6} motivates the following question concerning strong NA-ergodicity.

\begin{question}\label{question:strongNA}
Let $\Gamma \cc ^ a (X,\mu )$ be a non-trivial measure preserving action of a countable group $\Gamma$. Suppose that for each non-amenable subgroup $\Delta \leq \Gamma$ the action $\Delta \cc ^a (X,\mu )$ is strongly ergodic. Does it follow that the stabilizer of almost every point is contained in the amenable radical of $\Gamma$?
\end{question}

A positive answer to \ref{question:strongNA} would imply a positive answer to \ref{question:AReSM} by Proposition \ref{prop:NAerg}.

\medskip

The following question concerns the converse of Proposition \ref{prop:findex}:

\begin{question}\label{Q:fi}
Suppose $\Gamma$ is shift-minimal. Is it true that every finite index subgroup of $\Gamma$ is shift-minimal?
\end{question}

Question \ref{Q:fi} is equivalent to the question of whether every finite index \emph{normal} subgroup $N$ of a shift-minimal group $\Gamma$ is shift-minimal. Indeed, suppose the answer is positive for normal subgroups and let $K$ be a finite index subgroup of a shift-minimal group $\Gamma$. Then $K$ is ICC, since the ICC property passes to finite index subgroups. Since the group $N=\bigcap _{\gamma \in \Gamma}\gamma K \gamma ^{-1}$ is finite index and normal in $\Gamma$, it is shift-minimal by our assumption. Proposition \ref{prop:findex} then implies that $K$ is shift-minimal.

Corollary \ref{cor:torfree} provides a positive answer to Question \ref{Q:fi} for finite index subgroups which are torsion-free. Theorem \ref{thm:nolocfin} gives a positive answer for finite index normal subgroups $N$ of $\Gamma$ for which there is no infinite locally finite invariant random subgroup that is weakly contained in $\bm{s}_N$. Note that a positive answer to the analogue of Question \ref{Q:fi} for $C^*$-simplicity was demonstrated in \cite{BH00} (and likewise for the unique trace property).

The results from \S\ref{sec:infgen} and \S\ref{sec:fp1sm} suggest that the following may have a positive answer:

\begin{question}
If an infinite group $\Gamma$ has positive first $\ell ^2$-Betti number then is it true that $C^*_r(\Gamma / \mbox{AR}_\Gamma)$ is simple and has a unique tracial state?
\end{question}

There are already partial results in this direction: Peterson and Thom \cite{PT11} have shown a positive answer under the additional assumptions that $\Gamma$ is torsion free and that every non-trivial element of $\Z \Gamma$ acts without kernel on $\ell ^2\Gamma$.

Finally, we record here a question raised earlier in this paper.

\begin{question*}[Question \ref{Q:CompactAut}] Let $\Gamma$ be a countable group acting by automorphisms on a compact Polish group $G$ and assume the action is tempered. Does it follow that the action is weakly contained in $\bm{s}_\Gamma$? As a special case, is it true that the action $\mbox{SL}_2(\Z ) \cc (\T ^2 ,\lambda ^2 )$ is weakly contained in $\bm{s}_{\mbox{\scriptsize{SL}}_2(\Z )}$?
\end{question*}
%\appendix
%\appendixpage
\begin{appendices}
\section{Invariant random subgroups as subequivalence relations}  \label{sec:subequiv}

This first appendix studies \emph{invariant random partitions} of $\Gamma$ which are a natural generalization of invariant random subgroups. In \S\ref{subsec:IRP} it is shown that every invariant random partition of $\Gamma$ comes from a pair $(\bm{a},F)$ where $\bm{a}$ is a free m.p.\ action of $\Gamma$ and $F$ is a (Borel) subequivalence relation of $E_a$. It is shown in \S\ref{subsec:normsubeq} that for an invariant random subgroup any such pair $(\bm{a},F)$ will have the property that $F$ is \emph{normalized} by $\bm{a}$, i.e., $\gamma ^a$ is in the normalizer of the full group of $F$ for every $\gamma \in \Gamma$.

Many of the ideas here are inspired by (and closely related to) the notion of a \emph{measurable subgroup} developed by Bowen-Nevo \cite{BoNe09} and Bowen \cite{Bo11b}. See also Remark \ref{rem:BoNe}.

\subsection{Invariant random partitions}\label{subsec:IRP} By a \emph{partition} of $\Gamma$ we mean an equivalence relation on $\Gamma$. The set $\Es{P}_\Gamma$ of all partitions of $\Gamma$ is a closed subset of $2^{\Gamma \times \Gamma}$ and $\Gamma$ acts continuously on $\Es{P}_\Gamma$ by left translation $\Gamma \cc ^{\ell} \Es{P}_\Gamma$, i.e.,
\[
(\alpha ,\beta )\in \gamma P \ \IFF \ (\gamma ^{-1}\alpha ,\gamma ^{-1}\beta )\in P
\]
for each $\gamma ,\alpha ,\beta \in \Gamma$ and $P\in \Es{P}_\Gamma$. For $P\in \Es{P}_\Gamma$ and $\alpha \in \Gamma$ let $[\alpha ]_P= \{ \beta \csuchthat (\alpha ,\beta )\in P \}$ denote the $P$-class of $\alpha$. Then it is easy to check that $\gamma [\alpha ]_P = [\gamma \alpha ]_{\gamma P}$ for all $\gamma \in \Gamma$.

\begin{definition}\label{def:IRP}
An \emph{invariant random partition} of $\Gamma$ is a translation-invariant Borel probability measure on $\Es{P}_\Gamma$.
\end{definition}

\begin{remark}\label{rem:IRSIRP}
Let $\mbox{IRP}_\Gamma$ denote the space of all invariant random partitions of $\Gamma$. This is a convex set that is compact and metrizable in the weak${}^*$-topology. Similarly, let $\mbox{IRS}_\Gamma$ denote the compact convex set of all invariant random subgroups of $\Gamma$. There is a natural embedding $\Phi : \mbox{Sub}_\Gamma \hookrightarrow \Es{P}_\Gamma$ that assigns to each $H\in \mbox{Sub}_\Gamma$ the partition of $\Gamma$ determined by the right cosets of $H$, i.e., $[\delta ]_{\Phi (H)}= H\delta$ for $\delta \in \Gamma$. Observe that this embedding is $\Gamma$-equivariant between the conjugation action $\Gamma \cc ^c \mbox{Sub}_\Gamma$ and the translation action $\Gamma \cc ^\ell \Es{P}_\Gamma$. %This allows one to view $\mbox{Sub}_\Gamma$ as an invariant closed subspace of $\Es{P}_\Gamma$.
We thus obtain an embedding $\Phi _* : \mbox{IRS}_\Gamma \hookrightarrow  \mbox{IRP}_\Gamma$, $\theta \mapsto \Phi _*\theta$. %In what follows we will often not distinguish between and element of $\mbox{IRS}_\Gamma$ and its image in $\mbox{IRP}_\Gamma$ under them embedding $\Phi$.
\end{remark}

Suppose now that $F\subseteq X\times X$ is a measure preserving countable Borel equivalence relation on $(X,\mu )$ and $\bm{a}=\Gamma \cc ^a (X,\mu )$ is a m.p.\ action of $\Gamma$. Each point $x\in X$ determines a partition $P^a_F(x)$ of $\Gamma$ given by
\[
P^a_F(x) = \{ (\alpha , \beta )\in \Gamma \csuchthat \beta ^{-1}x F \alpha ^{-1}x \} .
\]
Note that $P^a_F(x) = P^a_{F\cap E_a}(x)$ for all $x\in X$, so if we are only concerned with properties of $P^a_F$ then we might as well assume that $F\subseteq E_a$.

\begin{proposition}\label{prop:IRPdist}
The map $x\mapsto P^a_F(x)$ is equivariant and therefore $(P^a_F)_*\mu$ is an invariant random partition of $\Gamma$.
\end{proposition}

\begin{proof}
For any $\gamma \in \Gamma$ and $x\in X$ we have
\[
(\alpha ,\beta )\in P^a_F(\gamma x) \IFF \alpha ^{-1}\gamma xF\beta ^{-1}\gamma x \IFF (\gamma ^{-1}\alpha ,\gamma ^{-1}\beta )\in P^a_F(x) \IFF (\alpha ,\beta )\in \gamma ^{\ell}\cdot P^a_F(x). \qedhere
\]
\end{proof}

Proposition \ref{prop:IRPdist} has a converse in a strong sense: given an invariant random partition $\rho$ of $\Gamma$ there is a \emph{free} m.p. action $\bm{b}= \Gamma \cc ^b (Y, \nu )$ of $\Gamma$ and a subequivalence relation $\Es{F}$ of $E_b$ with $(P^b_{\Es{F}})_*\nu = \rho$.  In fact, $\Es{F}$ and $b$ can be chosen independently of $\rho$, with only $\nu$ depending on $\rho$, as we now show. Let $\bm{\rho}$ denote the m.p.\ action $\Gamma \cc ^{\ell} (\Es{P}_\Gamma ,\rho )$ and let $\bm{b}=\bm{\rho}\times \bm{s}_\Gamma$ (any free action of $\Gamma$ will work in place of $\bm{s}_\Gamma$) so that $(Y,\nu ) = (\Es{P}_\Gamma \times [0,1]^\Gamma , \rho \times \lambda ^\Gamma )$. Define $\Es{F}\subseteq Y\times Y$ by
\begin{equation}\label{eqn:defF}
(P,x)\Es{F} (Q,y) \ \IFF \ \exists \gamma \in \Gamma \ (\gamma ^{-1} \in [e]_P \mbox{ and } (\gamma P ,\gamma x )=(Q, y) ) .
\end{equation}
%12/5/2012: Avoid confusion. For any subgroup $H\leq \Gamma$ the partition of $\Gamma$ determined by the \emph{left} coset of $H$ is an invariant random partition. This is obvious because left translation by $\Gamma$ just permutes the left cosets of $H$, so the underlying partition remains the same. This might seem strange since the characterization of those IRP's which are IRS's is in terms of invariance of the distribution corresponding to partitions determined by subgroups of $\Gamma$. However, note that if $H$ is a subgroup of $\Gamma$ such that the partition of $\Gamma$ determined by the \emph{right} cosets of $H$ is invariant under \emph{left} translation, then $H$ is normal in $\Gamma$. An easy exercise shows that the only point mass IRP's are point masses at subgroups. Thus, the only point mass IRP's which are IRS's are point masses at normal subgroups.

\begin{theorem}\label{thm:IRP} Let $\rho$ be an invariant random partition of $\Gamma$ and write $\bm{b}=\Gamma \cc ^b (Y,\nu )$ for the action $\bm{\rho}\times \bm{s}_\Gamma$. Let $\Es{F}$ be given by \emph{(\ref{eqn:defF})}. Then $\Es{F}$ is an equivalence relation contained in the equivalence relation $E_{b}$ generated by the $b$, %free action $\bm{b}=\bm{\rho}\times \bm{s}_\Gamma$, %\Gamma \cc ^{\ell\times s} (X,\mu )$, %\Es{P}_\Gamma \times [0,1]^\Gamma ,\rho \times \lambda ^\Gamma )$,
and $P^{b}_{\Es{F}}((P,x)) = P$ for $\nu$-almost every $(P,x) \in Y$. In particular, $(P^{a}_{\Es{F}})_*\nu = \rho$.
\end{theorem}

\begin{proof}[Proof of Theorem \ref{thm:IRP}]
It is clear that $\Es{F}\subseteq E_{b}$. We show that $\Es{F}$ is an equivalence relation: It is clear that ${\Es{F}}$ is reflexive. To see ${\Es{F}}$ is symmetric, suppose $(P,x){\Es{F}}(Q,y)$, as witnessed by $\gamma ^{-1}\in [e]_P$ with $\gamma  P = Q$ and $\gamma x = y$. Then $\gamma \in [e]_{\gamma  P} = [e]_Q$ and $(\gamma ^{-1} Q ,\gamma ^{-1}y ) = (P,x)$, so $(Q,y){\Es{F}}(P,x)$. For transitivity, if $(P,x){\Es{F}}(Q,y){\Es{F}}(R,z)$ as witnessed by $\gamma ^{-1}\in [e]_P$ with $(\gamma P, \gamma x)= (Q,y)$ and $\delta ^{-1}\in [e]_Q$ with $(\delta Q,\delta y)=(R,z)$ then $\gamma ^{-1}\in [e]_{P}$ and $\gamma  P = Q$ implies $[e]_{Q}= [e]_{\gamma  P}= \gamma  [e]_P$. Therefore $\delta ^{-1}\in \gamma  [e]_P$, i.e., $(\delta \gamma )^{-1} \in [e]_P$ and $(\delta \gamma P ,\delta \gamma x ) (\delta Q , \delta y) = (R,z)$.

Fix now $(P,x) \in Y$. We show that $P^b_{\Es{F}}((P,x))= P$. For each $\alpha ,\beta \in \Gamma$ we have by definition
\begin{align}
(\alpha ,\beta ) \in P^a_{\Es{F}}((P,x)) \ &\IFF (\alpha ^{-1}P, \alpha ^{-1} x) F (\beta ^{-1} P,\beta ^{-1}x)  \nonumber \\
\label{eqn:gamma} &\IFF \ \exists \gamma \in \Gamma \, \big( \gamma ^{-1}\in [e]_{\alpha ^{-1}P} \mbox{ and } (\gamma \alpha ^{-1}P , \gamma \alpha ^{-1} x) = (\beta ^{-1}Q, \beta ^{-1}x) \big) .
\end{align}
Therefore, if $(\alpha , \beta ) \in P^b_{\Es{F}}((P,x))$ as witnessed by some $\gamma$ as in (\ref{eqn:gamma}) then $\gamma \alpha ^{-1}x = \beta ^{-1}x$ so freeness of $\bm{a}$ implies $\gamma = \beta ^{-1}\alpha$. Then $\alpha ^{-1}\beta = \gamma ^{-1} \in [e]_{\alpha ^{-1}P}$, i.e., $(\alpha ^{-1}\beta ,e)\in \alpha ^{-1}P$, which is equivalent to $(\beta ,\alpha ) \in P$. This shows that $P^b_{\Es{F}}((P,x))\subseteq P$. For the reverse inclusion, if $(\alpha ,\beta )\in P$ then $\gamma = \beta ^{-1}\alpha$ satisfies (\ref{eqn:gamma}) and thus $(\alpha ,\beta )\in P^b_{\Es{F}}((P,x))$.
\end{proof}

\begin{definition}
Let $\bm{a}=\Gamma \cc ^a (X,\mu )$ be a m.p.\ action of $\Gamma$ and let $F$ be a subequivalence relation of $E_a$. If $\rho$ is an invariant random partition of $\Gamma$ then the pair $(\bm{a} ,F)$ is called a \emph{realization} of $\rho$ if $(P^a_F)_*\mu = \rho$. If $\theta$ is an invariant random subgroup of $\Gamma$ then $(\bm{a},F)$ is called a \emph{realization} of $\theta$ if it is a realization of $\Phi _*\theta$, where $\Phi _*: \mbox{IRS}_\Gamma \ra \mbox{IRP}_\Gamma$ is the embedding defined in Remark \ref{rem:IRSIRP}. A realization $(\bm{a},F)$ is called \emph{free} if $\bm{a}$ is free.
\end{definition}

The following is a straightforward consequence of Theorem \ref{thm:IRP} and the definitions.

\begin{corollary}\label{cor:realize}
Every invariant random partition admits a free realization. %For every invariant random partition $\rho$ of $\Gamma$ there is a free m.p. action $\Gamma \cc ^b(Y,\nu )$ of $\Gamma$ and a subequivalence relation $F$ of $E_b$ such that $(P^b_F) _*\nu = \rho$.
\end{corollary}

The remainder of this subsection works toward a characterization of the set $\Phi _*(\mbox{IRS}_\Gamma )$. Let $K$ be a metrizable compact space and consider the set $\Es{P}_\Gamma \otimes K$ of all pairs $(P,f )$ where $f:P^*\ra K$ is a function with $\mbox{dom}(f) = P^*= \{ [\alpha ]_P\csuchthat \alpha \in \Gamma \}$ and taking values in $K$. The set $\Es{P}_\Gamma \otimes K$ has a natural compact metrizable topology coming from its identification with the closed set
\[
\widetilde{\Es{P}_\Gamma \otimes K} = \{ (P,g)\in \Es{P}_\Gamma \times K^\Gamma \csuchthat g\mbox{ is constant on each }P\mbox{-class} \} \subseteq \Es{P}_\Gamma \times K^\Gamma
\]
via the injection $(P,f)\mapsto (P,\tilde{f})$ where $\tilde{f}(\alpha ) = f([\alpha ]_P)$ for $\alpha \in \Gamma$.  Observe that $\widetilde{\Es{P}_\Gamma \otimes K}$ is invariant in $\Es{P}_\Gamma \times K^\Gamma$ with respect to the product action $\ell \times s$ of $\Gamma$ (where $s$ denotes the shift action $\Gamma \cc ^s K^\Gamma$), so we obtain a continuous action $\Gamma \cc ^{\ell\otimes s} \Es{P}_\Gamma \otimes K$.  Explicitly, this action is given by $\gamma \cdot (P,f)  = (\gamma P , \gamma ^{s_P}f)$ where $\gamma ^{s_P}f : (\gamma P)^* \ra K$ is the function
\begin{align*}%\label{eqn:firstsec}
(\gamma ^{s_P}f)([\alpha ]_{\gamma P}) &= f(\gamma ^{-1}[\alpha ]_{\gamma P}) = f([\gamma ^{-1}\alpha ]_P ) .
\end{align*}
There is a natural equivalence relation $\Es{R}=\Es{R}_K$ on $\Es{P}_\Gamma \otimes K$ given by
\[
(P,f)\Es{R} (Q,g) \ \IFF \ \exists \gamma \in [e]_P \ (\gamma ^{-1}(P,f)=(Q,g) ) .
\]
It is clear that $\Es{R}$ is an equivalence relation that is contained in $E_{\ell \otimes s}$.

\begin{lemma}\label{lem:Psubset}
$P\subseteq P^{\ell\otimes s}_{\Es{R}} ((P,f))$ for every $(P,f ) \in \Es{P}_\Gamma \otimes K$.
\end{lemma}

\begin{proof}
Suppose that $(\alpha ,\beta )\in P$. Then $\beta ^{-1}\alpha \in [e]_{\beta ^{-1}P}$ so for any $f\in K^{P^*}$, from the definition of $\Es{R}$ we have
\[
(\beta ^{-1}P , \beta ^{-1}f ) \, \Es{R} \, (\beta ^{-1}\alpha )^{-1}(\beta ^{-1}P, \beta ^{-1}f ) = (\alpha ^{-1}P, \alpha ^{-1}f ) ,
\]
i.e., $\beta ^{-1}(P,f) \Es{R} \alpha ^{-1}(P,f)$. This means that $(\alpha ,\beta )\in P^{\ell\otimes s}_{\Es{R}}((P,f))$ by definition.
\end{proof}

If $\rho$ is an invariant random partition and $\mu$ is a Borel probability measure on $K$ then the measure $\rho \otimes \mu$ on $\Es{P}_\Gamma \otimes K$ given by
\[
\rho \otimes \mu = \int _P (\updelta _P \times \mu ^{P^*} )\, d\rho
\]
is $\ell\otimes s$-invariant. %We let $\bm{s}_\rho = \Gamma \cc ^{\ell\otimes s}(\Es{P}_\Gamma \otimes K , \, \rho\otimes \mu )$.
%(Note that by the second equality we indeed have that $\mbox{dom}(\gamma ^{s_P}f) = (\gamma P)^*$.)

\begin{theorem}
Let $\rho$ be an invariant random partition of $\Gamma$, let $\mu$ be any atomless measure on $K$, and let $\Es{R}= \Es{R}_K$. Then the following are equivalent:
\begin{enumerate}
\item $\rho \in \Phi _*(\mbox{\emph{IRS}}_\Gamma )$
\item $(\rho\otimes \mu )$-almost every $\Es{R}$-class is trivial.
\end{enumerate}
%In addition, when the equivalent properties (1) and (2) hold then $P^{\ell\times s}_{\Es{R}} (P,f ) = P$ for $(\rho\otimes \mu )$-almost every $(P,f)\in \Es{P}_\Gamma \otimes K$.
\end{theorem}

\begin{proof}
(1)$\Ra$(2): Suppose that (1) holds. %and let $\theta \in \mbox{IRS}_\Gamma$ be such that $\Phi _*\theta = \rho$.
It follows that $(\rho\otimes \mu )$ concentrates on pairs $(\Phi (H) , f) \in \Es{P}_\Gamma \otimes K$ with $H \in \mbox{Sub}_\Gamma$. It therefore suffices to show that the $\Es{R}$-class of such a pair $(\Phi (H),f)$ is trivial.  If $(\Phi (H),f)\Es{R}(Q,g)$ then there is some $\gamma \in [e]_{\Phi (H)}= H$ with $\gamma ^{-1}\Phi (H) =Q$ and $\gamma ^{-1} f = Q,g$. But $\gamma ^{-1}\Phi (H) = \Phi (\gamma ^{-1}H\gamma ) = \Phi (H)$ (since $\gamma \in H$) so that $Q=\Phi (H)$. In addition, for each $\delta \in \Gamma$ we have $\gamma [\delta ]_{\Phi (H)} = \gamma H\delta = H\delta = [\delta ]_{\Phi (H)}$ since $\gamma \in H$. Therefore $g([\delta ]_{\Phi (H)}) = (\gamma ^{-1}f)([\delta ]_{\Phi (H)}) = f(\gamma [\delta ]_{\Phi (H)}) = f([\delta ]_{\Phi (H)})$, showing that $g=f$.

(2)$\Ra$(1): Suppose that (2) holds. Since $\mu$ is non-atomic, for each $P\in \Es{P}_\Gamma$ the set $\{ f\in K^{P^*} \csuchthat f\mbox{ is injective}\}$ is $\mu ^{P^*}$-conull. This along with (2) implies that there is a $\Gamma$-invariant $(\rho\otimes \mu )$-conull set $Y\subseteq \Es{P}_\Gamma \otimes K$ on which $\Es{R}$ is trivial and such that $f:P ^* \ra K$ is injective whenever $(P,f)\in Y$. The projection $Y_0 = \{ P \in \Es{P}_\Gamma \csuchthat \exists f \ (P,f)\in Y \}$ is then $\rho$-conull so it suffices to show that $Y_0\subseteq \Phi (\mbox{Sub}_\Gamma )$. Fix $P\in Y_0$ and an $f:P^*\ra K$ with $(P,f)\in Y$.
\begin{claim}
Let $\alpha ,\beta \in \Gamma$. Then $(\alpha ,\beta ) \in P$ if and only if $\beta \alpha ^{-1} \in [e]_P$.
\end{claim}

\begin{proof}[Proof of Claim]
Suppose $(\alpha ,\beta )\in P$. Lemma \ref{lem:Psubset} implies $(\alpha ,\beta )\in P^a_{\Es{R}}{(P,f)}$ so as the relevant $\Es{R}$-classes are trivial this implies $\alpha ^{-1}( P,f) = \beta ^{-1}(P,f)$ and thus $\alpha \beta ^{-1}P=P$ and $\alpha \beta ^{-1}f=f$. Then $f([e]_P)=(\alpha \beta ^{-1}f)([e]_P) = f([\beta \alpha ^{-1}]_P)$ so injectivity of $f$ shows that $[\beta\alpha ^{-1}]_P = [e]_P$, i.e., $\beta\alpha ^{-1} \in [e]_P$.

Conversely, suppose $\beta\alpha ^{-1}\in [e]_P$. Then $(\beta \alpha )^{-1}(P,f)\Es{R}(P,f)$ by definition of $\Es{R}$, and since the $\Es{R}$-classes are trivial this implies $(\beta \alpha )^{-1}(P,f)=(P,f)$ and thus $\beta ^{-1}(P,f) =\alpha ^{-1}(P,f)$. Therefore $f([\beta ]_P) = (\beta ^{-1}f)([e]_{\beta ^{-1}P}) = (\alpha ^{-1}f)([e]_{\alpha ^{-1}P}) = f([\alpha ]_P)$. Since $f$ is injective we conclude that $[\beta ]_P =[\alpha ]_P$, i.e., $(\alpha ,\beta )\in P$.\qedhere[Claim]
\end{proof}

\noindent It is immediate from the claim that $[e]_P$ is a subgroup of $\Gamma$ and that $P$ is the partition determined by the right cosets of $[e]_P$, i.e., $P= \Phi ([e]_P)$.\qedhere
\end{proof}
%\begin{claim}
%$P^a_{\Es{R}}(P,f) = P$.
%\end{claim}
%
%\begin{proof}[Proof of Claim]
%By Lemma \ref{lem:Psubset} it suffices to show that $P^a_{\Es{R}}(P,f)\subseteq P$.  For $\alpha ,\beta \in \Gamma$, $(\alpha ,\beta )\in P^a_{\Es{R}}(P,f)$ then by definition $(\alpha ^{-1}P,\alpha ^{-1}f)\Es{R}(\beta ^{-1}P,\beta ^{-1}f)$, so as the $\Es{R}$-classes are trivial we actually have $\alpha ^{-1}P=\beta ^{-1}P$ and $\alpha ^{-1}f = \beta ^{-1}f$.  Therefore $f([\beta ]_P) = (\beta ^{-1}f)([e]_{\beta ^{-1}P}) = (\alpha ^{-1}f)([e]_{\alpha ^{-1}P}) = f([\alpha ]_P)$. Since $f$ is injective we conclude that $[\beta ]_P =[\alpha ]_P$, i.e., $(\alpha ,\beta )\in P$. \qedhere[Claim]
%\end{proof}

\subsection{Normalized subequivalence relations}\label{subsec:normsubeq}

As in the previous section let $F\subseteq X\times X$ be a m.p.\ countable Borel equivalence relation on $(X,\mu )$ and let $\bm{a}=\Gamma \cc ^a (X,\mu )$ be a m.p.\ action of $\Gamma$.

\begin{definition}
$F$ is said to be \emph{normalized} by $\bm{a}=\Gamma \cc ^a (X,\mu )$ if there is a conull set $X_0\subseteq X$ such that
\[
xFy \ \Ra \ \gamma x F \gamma y
\]
for all $\gamma \in \Gamma$ and $x,y\in X_0$.  Equivalently, $F$ is normalized by $\bm{a}$ if the image of $\Gamma$ in $\Aut (X,\mu )$ is contained in the normalizer of the full group of $F$. A realization $(\bm{a},F)$ of an invariant random partition $\rho$ of $\Gamma$ is called \emph{normal} if $F$ is normalized by $\bm{a}$.
\end{definition}

Note that if $F$ is normalized by $\bm{a}$ then $F\cap E_a$ is normalized by $\bm{a}$ and $P^a_{F\cap E_a}(x) = P^a_F (x)$ so it makes sense once again to restrict our attention to the case where $F\subseteq E_a$. %We now work toward several characterizations of $F$ being normalized by $\bm{a}$.
%
%In preparation for Proposition \ref{prop:IRSIRP}
Define now
\[
\Gamma ^a _F(x) = \{ \gamma \in \Gamma \csuchthat \gamma ^{-1}x Fx \}
\]
It follows from the definitions that $\Gamma ^a _F(x) = [e]_{P^a_F(x)}$.

\begin{proposition}\label{prop:IRSIRP}
Let $F$ be a subequivalence relation of $E_a$. Then the following are equivalent
\begin{enumerate}
\item $F$ is normalized by $\bm{a}$.
\item For almost all $x$, $\Gamma ^a_F(x)$ is a subgroup of $\Gamma$ and $P^a _F (x)$ is the partition of $\Gamma$ determined by the right cosets of $\Gamma ^a _F(x)$, i.e.,
    \[
    (\alpha ,\beta )\in P^a_F(x) \ \IFF \ \Gamma ^a_F(x)\alpha = \Gamma ^a_F(x)\beta .
    \]
    for all $\alpha ,\beta \in \Gamma$.
\item $\Gamma ^a_F(\gamma x) = \gamma \Gamma ^a_F(x)\gamma ^{-1}$ for almost all $x\in X$ and all $\gamma \in \Gamma$.
\item The set $[e]_P$ is a subgroup of $\Gamma$ for $(P^a_F)_*\mu$-almost every $P\in \Es{P}_\Gamma$ and the map $P\mapsto [e]_P$ is an isomorphism from $\Gamma \cc ^\ell (\Es{P}_\Gamma , (P ^a_F)_*\mu )$ to $\Gamma \cc ^c (\mbox{\emph{Sub}}_\Gamma , (\Gamma ^a_F)_*\mu )$.
\end{enumerate}
\end{proposition}

\begin{proof}
(1)$\Ra$(2): Suppose (1) holds. By ignoring a null set we may assume without loss of generality that $xFy \, \Ra \, \gamma x F \gamma y$ for all $x,y\in X$ and $\gamma \in \Gamma$. We have that $e\in \Gamma ^a_F(x)$ for all $x$. If $\gamma \in \Gamma ^a_F(x)$ then $\gamma ^{-1}xFx$ so by normality we have $xF\gamma x$ and thus $\gamma ^{-1}\in \Gamma ^a_F(x)$. If in addition $\delta \in \Gamma ^a_F(x)$ then $\delta ^{-1}xFxF\gamma x$ so that $\delta ^{-1}xF\gamma x$ which by normality implies $\gamma ^{-1}\delta ^{-1} xFx$, i.e., $\delta \gamma \in \Gamma ^a_F(x)$. This shows that $\Gamma ^a_F(x)$ is a subgroup. It remains to show that $[\delta ]_{P^a_F(x)} = \Gamma ^a_F(x)\delta$. We have $\gamma \in [\delta ]_{P^a_F(x)}$ if and only if $\delta ^{-1}xF\gamma ^{-1}x$ which by normality is equivalent to $(\delta \gamma ^{-1})x Fx$, i.e., $\gamma \in \Gamma ^a_F(x)\delta$.

(2)$\Ra$(3): Suppose (2) holds. Then for almost all $x$ and all $\gamma ,\delta \in \Gamma$ we have
\begin{align*}
\delta \in \Gamma ^a_F(\gamma x) \ &\IFF \ \delta ^{-1}\gamma x F\gamma ^a x \ \IFF \ \gamma ^{-1}\delta ^{-1} \gamma xFx \ \IFF \ \delta \in \gamma \Gamma ^a_F(x)\gamma ^{-1}.
\end{align*}

(3)$\Ra$(1): Suppose that (3) holds. Let $X_0\subseteq X$ be an $E_a$-invariant conull set such that $\Gamma ^a_F(\gamma x) = \gamma \Gamma ^a_F(x)\gamma ^{-1}$ for all $x\in X_0$ and $\gamma \in \Gamma$. Then for any $x,y\in F$, if $xFy$ then $xE_ay$ so that $y=\delta x$ for some $\delta \in \Gamma$. This means that $\delta ^{-1} \in \Gamma ^a_F(x)$ and, so for all $\gamma \in \Gamma$ we have $\gamma \delta ^{-1}\gamma ^{-1} \in \Gamma ^a_F(\gamma x)$ and thus
\[
\gamma y = (\gamma \delta  ^{-1}\gamma ^{-1} )^{-1} (\gamma x) F \gamma x.
\]
This shows that $F$ is normalized by $\bm{a}$.

(2)+(3)$\Ra$(4): Assume (2) and (3) hold. Then the measure $(P ^a_F)_*\mu$ concentrates on $\Phi (\mbox{Sub}_\Gamma )$. It follows that $P\mapsto [e]_P$ is injective on a $(P ^a_F)_*\mu$-conull set. By (3) this map is equivariant on a conull set. Since the composition $x\mapsto P^a_F(x) \mapsto [e]_{P^a_F(x)}$ is the same as $x\mapsto \Gamma ^a_F(x)$ this map is measure preserving.

Finally, the implication (4)$\Ra$(3) is clear.
\end{proof}

The following corollary is immediate.

\begin{corollary}\label{cor:getIRS}
If $F$ is normalized by $\bm{a}$ then $(\Gamma ^a_F)_*\mu$ is an invariant random subgroup of $\Gamma$.
\end{corollary}

Theorem \ref{thm:IRP} also implies a converse to Corollary \ref{cor:getIRS}. Let $\theta$ be an invariant random subgroup of $\Gamma$ and let $\rho = \Phi _*\theta$. Let $\bm{b}$ and $\Es{F}$ be defined as in Theorem \ref{thm:IRP}. Let $\bm{a} = \bm{\theta}\times \bm{s}_\Gamma$ so that $(X,\mu ) = (\mbox{Sub}_\Gamma \times [0,1]^\Gamma ,\theta \times \lambda )$. Then the map $\Psi : (H,x ) \mapsto (\Phi (H),x)$ is an isomorphism of $\bm{a}$ with $\bm{b}$. Letting $\Es{F}_0 = (\Psi \times \Psi )^{-1}(\Es{F})$, we have that
\begin{equation}\label{eqn:defF0}
(H,x)\Es{F}_0 (L,y) \ \IFF \ H=L \mbox{ and }(\exists h\in H )(h^a x =y ) .
\end{equation}

\begin{corollary}\label{cor:IRSsuber}
$\Es{F}_0$ is a subequivalence relation of $E_{a}$ on $X$ which is normalized by $\bm{a}$ and satisfies $\Gamma ^{a} _{\Es{F}_0}(H,x)= H$ for $\theta\times \mu$-almost-every $(H,x) \in X$. Thus $(P^a_{\Es{F}_0})_*\mu = \Phi _*\theta$. It follows that every invariant random subgroup of $\Gamma$ admits a normal, free realization.
\end{corollary}
%for any invariant random subgroup $\theta$ of $\Gamma$ there is a free m.p.\ action $\bm{a}=\Gamma \cc ^a(X,\mu )$ of $\Gamma$ and a subequivalence relation $F$ of $E_a$ normalized by $\bm{a}$ such that $(\Gamma ^b_F)_*\mu = \theta$.
\begin{proof}
%Viewing $\theta$ is an element of $\mbox{IRP}_\Gamma$, the equivalence relation $F$ is the same as the one defined in Theorem \ref{thm:IRP}.
All that needs to be checked is that $\Es{F}_0$ is normalized by $\bm{\theta}\times \bm{a}$. If $(H,x)\Es{F}_0 (L,y)$ then $H=L$ and $h^ax =y$ for some $h\in H$.  Then for any $\gamma \in \Gamma$ we must show that $\gamma \cdot (H,x) \, \Es{F}_0 \, \gamma \cdot (H,h^ax)$. Now, $\gamma \cdot (H, x) = (\gamma H \gamma ^{-1} , \gamma ^a x)$, so as $\gamma h \gamma ^{-1}\in \gamma H \gamma ^{-1}$ the definition (\ref{eqn:defF0}) of $\Es{F}_0$ shows that
\[
(\gamma H \gamma ^{-1} , \gamma ^a x) \, \Es{F}_0 \, \gamma h \gamma ^{-1}\cdot (\gamma H \gamma ^{-1} , \gamma ^a x) = \gamma \cdot (H , h^a x) \qedhere
\]
\end{proof}

\begin{remark}\label{rem:judic}
In Corollary \ref{cor:IRSsuber}, if $\theta$ concentrates on the amenable subgroups of $\Gamma$ then $\Es{F}_0$ will always be an amenable equivalence relation. For other properties of $\theta$, a judicious choice of free action $\bm{d}$ in place of $\bm{s}_\Gamma$ in the definition of $\bm{a}$ may ensure that properties of $\theta$ are reflected by the equivalence relation $F$. For example, if $\theta$ concentrates on subgroups of cost $r$ then the proof of Theorem \ref{thm:gencost1} above shows that $\bm{d}$ can be chosen so that the corresponding equivalence relation $\Es{F}_0$ has cost $r$. Similarly, if $\theta$ concentrates on treeable subgroups then $\Es{F}_0$ can be made a treeable equivalence relation.
\end{remark}

\begin{remark}\label{rem:BoNe}
Following \cite[\S 2.2]{BoNe09} let $2^\Gamma _e =\{ L\in 2^\Gamma \csuchthat e\in L \}$ and define the equivalence relation $\Es{R}_e \subseteq 2^\Gamma _e\times 2^\Gamma _e$ by
\[
(L,K) \in \Es{R}_e \IFF \exists \gamma \in L \ \, \gamma ^{-1}L = K .
\]
Then any $\Es{R}_e$-invariant Borel probability measure on $2^\Gamma _e$ is called a a \emph{measurable subgroup} of $\Gamma$ (see \cite{BoNe09} and \cite{Bo11b}).  If $\rho$ is any invariant random partition of $\Gamma$ then the image of $\rho$ under $P\mapsto [e]_P$ is a measurable subgroup of $\Gamma$. I do not know whether every measurable subgroup of $\Gamma$ comes from an invariant random partition in this way.
\end{remark}

Creutz and Peterson \cite{CP12} define the \emph{subgroup} partial order on $(\mbox{IRS}_\Gamma ,\leq )$ as follows: Let $\theta _1, \theta _2 \in \mbox{IRS}_\Gamma$. Then $\theta _1$ is called a \emph{subgroup} of $\theta _2$ (written $\theta _1\leq \theta _2$) if there exists a joining of $\theta _1$ and $\theta _2$ that concentrates on the set $\{ (H,L) \in \mbox{Sub}_\Gamma \csuchthat H\leq L \}$. It is shown in \cite{CP12} that this is a partial order on $\mbox{IRS}_\Gamma$. The same idea can be used to define a notion of refinement for invariant random partitions.

For partitions $P,Q \in \Es{P}_\Gamma$, $P$ is said to \emph{refine} $Q$, written $P\leq Q$, if $P$ is a subset of $Q$. Equivalently $P\leq Q$ means that $[\alpha ]_P\subseteq [\alpha ]_Q$ for every $\alpha \in \Gamma$. If $\rho _1$ and $\rho _2$ are invariant random partitions of $\Gamma$ then $\rho _1$ \emph{refines} $\rho _2$, written $\rho _1 \leq \rho _2$, if there exists a joining of $\rho _1$ and $\rho _2$ that concentrates on the set $\{ (P,Q)\in \Es{P}_\Gamma \times \Es{P}_\Gamma \csuchthat P\leq Q \}$. It is clear that the restriction of the refinement relation on $\Es{P}_\Gamma$ (respectively, $\mbox{IRP}_\Gamma$) to $\mbox{Sub}_\Gamma$ (respectively, $\mbox{IRS}_\Gamma$) is the subgroup relation.

The point of view developed in this section can be used to give a characterization of the partial orders $(\mbox{IRS}_\Gamma ,\leq )$ and $(\mbox{IRP}_\Gamma ,\leq )$ in terms of subequivalence relations of free actions of $\Gamma$.

\begin{theorem}\label{thm:IRPleq}
Let $\rho _1 , \rho _2 \in \mbox{\emph{IRP}}_\Gamma$. Then the following are equivalent
\begin{enumerate}
\item[(1)] $\rho _1 \leq \rho _2$
\item[(2)] There exists a free m.p. action $\Gamma \cc ^a (X,\mu )$ of $\Gamma$ and equivalence relations $F_1\subseteq F_2\subseteq E_a$ with $(P^a_{F_1})_*\mu = \rho _1$ and $(P^a_{F_2})_*\mu = \rho _2$.
\end{enumerate}
If $\theta _1, \theta _2 \in \mbox{\emph{IRS}}_\Gamma$ then then following are equivalent
\begin{enumerate}
\item[(1')] $\theta _1\leq \theta _2$.
\item[(2')] There exists a free m.p. action $\Gamma \cc ^a (X,\mu )$ of $\Gamma$ and normalized equivalence relations $F_1\subseteq F_2\subseteq E_a$ with $(\Gamma ^a_{F_1})_*\mu = \theta _1$ and $(\Gamma ^a_{F_2})_*\mu = \theta _2$.
\end{enumerate}
\end{theorem}

\begin{proof}
Suppose (2) holds and let $P^a_{F_1}\times P^a_{F_2} : X\ra \Es{P}_\Gamma \times \Es{P}_\Gamma$ be the map $x\mapsto (P^a_{F_1}(x), P^a_{F_2}(x))$. Then $(P^a_{F_1}\times P^a_{F_2})_*\mu$ is a joining of $\rho _1$ and $\rho _2$ with the desired property.

Assume that (1) holds and let $\nu$ be a joining of $\rho _1$ and $\rho _2$ witnessing that $\rho _1 \leq \rho _2$. Let $X = \Es{P}_\Gamma \times \Es{P}_\Gamma \times [0,1]^\Gamma$, let $\mu = \nu \times\lambda ^\Gamma$, and let $a= \ell \times \ell \times s$. Then we define the equivalence relations $F_1$ and $F_2$ on $X$ by
\begin{align*}
(P_1,P_2,x)F_1 (Q_1,Q_2,y) \ &\IFF \ \exists \gamma \in \Gamma (\gamma ^{-1}\in [e]_{P_1} \mbox{ and } \gamma ^a \cdot (P_1,P_2,x) = (Q_1,Q_2,y) ) \\
(P_1,P_2,x)F_2 (Q_1,Q_2,y) \ &\IFF \ \exists \gamma \in \Gamma (\gamma ^{-1}\in [e]_{P_2} \mbox{ and } \gamma ^a \cdot (P_1,P_2,x) = (Q_1,Q_2,y) ) .
\end{align*}
Then as in the proof of Theorem \ref{thm:IRP}, $F_1$ and $F_2$ are equivalence relations that are contained in $E_a$ and $(\bm{a},F_i )$ is a realization of $F_i$ for each $i\in \{ 1,2\}$. The defining property of $\nu$ also ensures that $F_1\subseteq F_2$.

The equivalence of (1') and (2') then follows from the equivalence of (1) and (2) along with Proposition \ref{prop:IRSIRP}.
\end{proof}

Finally, we note the following (observed by Vershik \cite{Ve11} in the case of invariant random subgroups), which is a consequence of \cite[\S 1]{IKT09}.

\begin{theorem}\label{thm:IKT}
Let $\rho$ be an invariant random partition of $\Gamma$. Then the function
\[
\varphi _\rho (\gamma ) = \rho ( \{ P\csuchthat \gamma \in [e]_P \} )
\]
is a positive definite function on $\Gamma$.
\end{theorem}

\begin{proof} By Corollary \ref{cor:realize} there is a free m.p. action $\bm{b}=\Gamma \cc ^b (Y,\nu )$ of $\Gamma$ and a subequivalence relation $F$ of $E_b$ such that $(P^b_F)_*\nu = \rho$. Thus
\[
\varphi _\rho (\gamma ) = \nu (\{ y \csuchthat \gamma ^{-1}yFy \} ) .
\]
This is a positive definite function by \cite{IKT09}.
\end{proof}

\section{The amenable radical of a countable group}\label{app:ar}

Every countable discrete group $\Gamma$ contains a largest normal amenable subgroup called the \emph{amenable radical} of $\Gamma$ (see, e.g., \cite[4.1.12]{Zi84}). We write $\mbox{AR}_\Gamma$ for the amenable radical of $\Gamma$. We present in this appendix some facts concerning $\mbox{AR}_\Gamma$ for countable $\Gamma$. %Everything in \S\ref{app:basic} is likely well known, although I am unaware of a reference.

\subsection{Basic properties of $\mbox{AR}_\Gamma$}\label{app:basic}
\begin{proposition}\label{prop:ARbasic}
Let $\Gamma$ be a countable group.
\begin{enumerate}
\item[(1)] $\mbox{\emph{AR}} _\Gamma $ is an amenable characteristic subgroup of $\Gamma$ which contains every normal amenable subgroup of $\Gamma$.
\item[(2)] Suppose $\varphi : \Gamma \ra \Delta$ is a group homomorphism and that $\mbox{ker}(\varphi )$ is amenable. Then $\varphi (\mbox{\emph{AR}}_\Gamma )=\mbox{\emph{AR}}_{\varphi (\Gamma )}$. In particular, the amenable radical of the quotient group $\Gamma / \mbox{\emph{AR}}_\Gamma$ is trivial.
\item[(3)] If $H$ is normal in $\Gamma$ then $\mbox{\emph{AR}}_H$ is a normal subgroup of $\mbox{\emph{AR}}_\Gamma$ with $\mbox{\emph{AR}}_H = \mbox{\emph{AR}}_\Gamma \cap H$.
\item[(4)] If $H$ is finite index in $\Gamma$ then $\mbox{\emph{AR}}_H$ is a finite index subgroup of $\mbox{\emph{AR}}_\Gamma$ with $\mbox{\emph{AR}}_H = \mbox{\emph{AR}}_\Gamma \cap H$.
\end{enumerate}
\end{proposition}

\begin{proof}
For (1) see \cite{Zi84}. For (2), let $N = \mbox{ker}(\varphi )$. It is clear that $\varphi (\mbox{AR}_\Gamma )$ is a normal amenable subgroup of $\varphi (\Gamma )$, so that $\varphi (\mbox{AR}_\Gamma ) \leq \mbox{AR}_{\varphi (\Gamma )}$ by (1). The group $K = \varphi ^{-1}(\mbox{AR}_{\varphi (\Gamma )})$ is normal in $\Gamma$ and $K$ is amenable since both $N$ and $K/N \cong \mbox{AR}_{\varphi (\Gamma )}$ are amenable. Hence $K\leq \mbox{AR}_\Gamma$ and so $\mbox{AR}_{\varphi (\Gamma )} \leq \varphi (K) \leq \varphi (\mbox{AR}_\Gamma )$.

We now prove (3). Suppose that $H$ is normal in $\Gamma$. It is clear that $\mbox{AR}_\Gamma \cap H$ is normal in $\mbox{AR}_\Gamma$, so it suffices to show that $\mbox{AR}_\Gamma \cap H = \mbox{AR}_H$. Conjugation by any element of $\Gamma$ is an automorphism of $H$, so fixes (setwise) the characteristic subgroup $\mbox{AR}_H$. This shows that $\mbox{AR}_H$ is normal in $\Gamma$, and since it is amenable it must be contained in $\mbox{AR}_\Gamma$. Thus $\mbox{AR}_H \leq \mbox{AR}_\Gamma \cap H$. In addition, $\mbox{AR}_\Gamma\cap H$ is a normal amenable subgroup of $H$, so $\mbox{AR}_\Gamma \cap H \leq \mbox{AR}_H$. This proves (3).

We need the following Lemma for (4):

\begin{lemma}\label{lem:anorm}
Suppose that $K$ is an amenable subgroup of $\Gamma$ whose normalizer $N_\Gamma (K)$ is finite index in $\Gamma$. Then $K\leq \mbox{\emph{AR}}_\Gamma$.
\end{lemma}

\begin{proof}[Proof of Lemma \ref{lem:anorm}]
Suppose first that $K$ is finite. $N_\Gamma (K)$ being finite index means $K$ has only finitely many conjugates in $\Gamma$, so as $K$ itself is finite this implies that every element of $K$ has a finite conjugacy class in $\Gamma$. Thus, $K\subseteq \mbox{FC}_\Gamma \subseteq \mbox{AR}_\Gamma$, where $\mbox{FC}_\Gamma$ is the amenable characteristic subgroup of $\Gamma$ consisting of all elements of $\Gamma$ with finite conjugacy classes (see e.g., \cite[Appendix J]{Ha07}).

Suppose now that $K$ is infinite. The normal core $N= \bigcap _{\gamma \in \Gamma} \gamma N_\Gamma (K)\gamma ^{-1}$ of $N_\Gamma (K)$ in $\Gamma$ is a normal finite index subgroup of $\Gamma$. Thus, letting $H = K\cap N$, we have $[K:H]=[KN:N]\leq [\Gamma :N]<\infty$, and so $H$ is finite index in $K$. It is clear that $H$ is normal in $N$, %(in fact it is normal in $N_\Gamma (K)$), since both $K$ and $N$ are normal in $N_\Gamma (K)$,
and $H$ is an amenable group since it is a subgroup of $K$. Thus $H\leq \mbox{AR}_N$. In addition, $\mbox{AR}_{N}$ is normal in $\Gamma$ since $\mbox{AR}_N$ is characteristic in $N$ and $N$ is normal in $\Gamma$. Therefore
\[
H\leq \mbox{AR}_N\leq \mbox{AR}_\Gamma .
\]
Now, $H$ is finite index in $K$, and $H\leq \mbox{AR}_\Gamma$, so the image $p(K)$ of $K$ in $\Gamma /\mbox{AR}_\Gamma$ under the quotient map $p$ is a finite subgroup of $\Gamma$. So if $p(K)$ were non-trivial then $\Gamma /\mbox{AR}_\Gamma$ would have non-trivial amenable radical, contrary to part (2).
\qedhere[Lemma \ref{lem:anorm}]
\end{proof}

We can now show (4). If $H$ is finite index in $\Gamma$, then $\mbox{AR}_H$ is an amenable subgroup of $\Gamma$ whose normalizer $N_\Gamma (\mbox{AR}_H)$ contains $H$. Therefore $N_\Gamma (\mbox{AR}_H)$ is finite index in $\Gamma$, so $\mbox{AR}_H \leq \mbox{AR}_\Gamma$ by Lemma \ref{lem:anorm}, and thus $\mbox{AR}_H\leq \mbox{AR}_\Gamma \cap H$. The group $\mbox{AR}_\Gamma$ is normal in $\Gamma$, so $\mbox{AR}_\Gamma \cap H$ is normal in $H$ and since it is an amenable group we have the other inclusion $\mbox{AR}_\Gamma \cap H \leq \mbox{AR}_H$.
\end{proof}

\begin{lemma}\label{lem:aaseries}
Let $\Gamma$ be a countable group and let $\{ H_\alpha \} _{\alpha \leq \lambda}$ be an almost ascendant series in $\Gamma$ (Definition \ref{def:almostasc}). Then $\{ \mbox{\emph{AR}}_{H_\alpha} \} _{\alpha \leq \lambda}$ is an almost ascendant series in $\mbox{\emph{AR}}_\Gamma$. The same holds if we replace "almost ascendant" by "ascendant." %Thus, $\mbox{\emph{AR}}_{H_\alpha}\leq \mbox{\emph{AR}}_{H_{\alpha +1}}$ and for limit ordinals $\alpha\leq \lambda$ we have
%\[
%\mbox{\emph{AR}}_{H_\alpha} = \bigcup _{\beta <\alpha}\mbox{\emph{AR}}_{H_\beta}.
%\]
\end{lemma}

\begin{proof}
We show by transfinite induction on ordinals $\alpha$ (with $\alpha \leq \lambda$) that $\{ \mbox{AR}_{H_\beta} \} _{\beta \leq \alpha}$ is an almost ascendant series in $\mbox{AR}_{H_\alpha}$.  If $\alpha =\beta +1$ is a successor ordinal then %by induction $\mbox{AR}_{H_\beta}$ is either normal or finite index in $\mbox{AR}_{H_{\beta +1}}$. We have
by hypothesis $H_\beta$ is either normal or finite index in $H_{\beta +1}$. Proposition \ref{prop:ARbasic} then implies that $\mbox{AR}_{H_{\beta}}$ is either normal or finite index in $\mbox{AR}_{H_{\beta + 1}}$.

Suppose now that $\alpha$ is a limit ordinal and let $K= \bigcup _{\beta <\alpha}\mbox{AR}_{H_\beta}$. We must show that $\mbox{AR}_{H_\alpha} = K$. By the induction hypothesis the groups $\mbox{AR}_{H_\beta}$, $\beta <\alpha$, are increasing with $\beta$, so $K$ is amenable, being an increasing union of amenable groups. Additionally, $K$ is normal in $H_\alpha$ as we now show. For each $h\in H_\alpha$ there is some $\beta _0 <\alpha$ such that $h\in H_{\beta _0}$. Therefore $h\in H_{\beta }$ for all $\beta _0\leq \beta < \alpha$.  Thus $h$ normalizes $\mbox{AR}_{H_{\beta }}$ for all $\beta _0\leq \beta < \alpha$, and since the $\mbox{AR}_{H_\beta}$ are increasing we have
\begin{align*}
hKh^{-1} &= \bigcup _{\beta _0\leq \beta <\alpha} h \mbox{AR}_{H_\beta}h^{-1} = \bigcup _{\beta _0 \leq \beta <\alpha} \mbox{AR}_{H_\beta} = K.
\end{align*}
It follows that $K\leq \mbox{AR}_{H_\alpha}$. We have the equality $K= \mbox{AR}_{H_\alpha}$ since %$\mbox{AR}_{H_\alpha} \cap H_\beta \leq \mbox{AR}_{H_\beta}\leq K$ for all $\beta < \alpha$, so that
$\mbox{AR}_{H_\alpha}= \bigcup _{\beta < \alpha}(\mbox{AR}_{H_\alpha}\cap H_\beta ) \leq \bigcup _{\beta <\alpha} \mbox{AR}_{H_\beta} = K$.
\end{proof}

\begin{corollary}\label{cor:aaAR}
Let $\Gamma$ be a countable group and let $H$ be an almost ascendant subgroup of $\Gamma$. Then
\[
\mbox{\emph{AR}}_H = \mbox{\emph{AR}}_\Gamma \cap H ,
\]
In particular, $\mbox{\emph{AR}}_H$ is contained in $\mbox{\emph{AR}}_\Gamma$, and $\mbox{\emph{AR}}_\Gamma$ contains every almost ascendant amenable subgroup of $\Gamma$.
\end{corollary}

\begin{proof}
The containment $\mbox{AR}_H\leq \mbox{AR}_\Gamma \cap H$ is immediate from Lemma \ref{lem:aaseries}. We have equality since $\mbox{AR}_\Gamma \cap H$ is an amenable normal subgroup of $H$.
\end{proof}

\begin{corollary}
Let $\Gamma$ be a countable group and let $\gamma \in \Gamma$. If the centralizer $C_\Gamma (\gamma )$ of $\gamma$ is almost ascendant in $\Gamma$ then $\gamma \in\mbox{\emph{AR}}_\Gamma$. Thus, if $\mbox{\emph{AR}}_\Gamma$ is trivial then the centralizer of any non-trivial element of $\Gamma$ is not almost ascendant.
\end{corollary}

\begin{proof}
The group $\langle \gamma \rangle$ is a normal amenable subgroup of $C_\Gamma (\gamma )$, so if $C_\Gamma (\gamma )$ is almost ascendent then $\langle \gamma \rangle \leq \mbox{AR}_{C_\Gamma (\gamma )} \leq \mbox{AR}_\Gamma$ by \ref{cor:aaAR}.
\end{proof}

\subsection{Groups with trivial amenable radical}

\begin{lemma}\label{lem:normAR}
Let $N$ be a normal subgroup of $\Gamma$. Then $\mbox{\emph{AR}}_{\Gamma}$ is trivial if and only if both $\mbox{\emph{AR}}_N$ and $\mbox{\emph{AR}}_{C_\Gamma (N)}$ are trivial.
\end{lemma}

\begin{proof}
Since $N$ is normal in $\Gamma$, $C_\Gamma (N)$ is normal in $\Gamma$ as well. Thus, if $\mbox{AR}_\Gamma$ is trivial it follows from Proposition \ref{prop:ARbasic} that both $\mbox{AR}_N$ and $\mbox{AR}_{C_\Gamma (N)}$ are trivial.

Suppose now that $\mbox{AR}_N$ and $\mbox{AR}_{C_\Gamma (N)}$ are trivial. We have
\[
\mbox{AR}_\Gamma \cap N = \mbox{AR}_N = \{ e \}
\]
and thus $\mbox{AR}_\Gamma$ and $N$ must commute, being normal subgroups of $\Gamma$ with trivial intersection. This means that $\mbox{AR}_\Gamma \leq C_\Gamma (N)$ and so
\[
\mbox{AR}_\Gamma = \mbox{AR}_\Gamma \cap C_\Gamma (N) = \mbox{AR}_{C_\Gamma (N)} = \{ e \} . \qedhere
\]
\end{proof}

\begin{lemma}\label{lem:downward}
Suppose $\{ H_\alpha \} _{\alpha \leq \lambda }$ is an ascendant series of length $\lambda$ and suppose $\Gamma = H_{\lambda}$ has trivial amenable radical. Then $\mbox{\emph{AR}}_{C_\Gamma (H_\alpha )} = \{ e\}$ for all $\alpha \leq \lambda$.
\end{lemma}

\begin{proof}
We proceed by transfinite induction on $\lambda$.  By Corollary \ref{cor:aaAR} we know that $\mbox{AR}_{H_\alpha} = \{ e \}$ for all $\alpha \leq \lambda$.

{\bf Limit stages:} Suppose first that $\lambda$ is a limit ordinal. Fix $\alpha \leq \lambda$ and let $H= H_\alpha$. By intersecting each term of the ascendant series $\{ H_\beta \} _{\beta \leq \lambda}$ with $C_\Gamma (H)$ we obtain the series $\{ C_{H_\beta}(H) \} _{\beta \leq \lambda}$ which is ascendant in $C_\Gamma (H)$. Lemma \ref{lem:aaseries} implies that $\{ \mbox{AR}_{C_{H_\beta}(H)} \} _{\beta \leq \lambda}$ is an ascendant series in $\mbox{AR}_{C_\Gamma (H)}$ and so
\begin{equation}\label{lambdaunion}
\mbox{AR}_{C_\Gamma (H)} = \bigcup _{\alpha \leq \beta <\lambda}\mbox{AR}_{C_{H_\beta}(H)}
\end{equation}
where the union is increasing. For each $\beta$ with $\alpha \leq \beta <\lambda$ the series $\{ H_\xi \} _{\xi \leq \beta}$ has length strictly less than $\lambda$, so by the induction hypothesis we have
\[
\mbox{AR}_{C_{H_\beta} (H)} = \{ e\} .
\]
Since this holds for each $\beta$ with $\alpha\leq \beta <\lambda$, equation (\ref{lambdaunion}) shows that $\mbox{AR}_{C_\Gamma (H)} = \{ e \}$ as was to be shown.

{\bf Successor stages:} Suppose now that $\lambda = \mu +1$ is a successor ordinal. Fix for the moment some $\alpha < \lambda$ and let $H=H_\alpha$. Applying the induction hypothesis to the ascendant series $\{ H_\beta \} _{\beta \leq \mu}$ in $H_\mu$ we obtain that $\mbox{AR}_{C_{H_\mu}(H)} = \{ e \}$. Since $H_\mu$ is normal in $\Gamma$, $C_{H_\mu}(H)$ is normal in $C_\Gamma (H)$, so it follows from Proposition \ref{prop:ARbasic}.(3) that
\begin{equation}\label{eqn:triv}
\mbox{AR}_{C_{\Gamma}(H)}\cap H_\mu = \mbox{AR}_{C_\Gamma (H)}\cap C_{H_\mu}(H) = \mbox{AR}_{C_{H_\mu}(H)} = \{ e\} .
\end{equation}
Since $\alpha$ was an arbitrary ordinal satisfying $\alpha <\lambda$, (\ref{eqn:triv}) holds for all $\alpha < \lambda$. We use this to show the following.

\begin{claim}\label{claim:alpha+1}
Let $\xi$ and $\beta$ be ordinals with $\xi \leq \beta <\lambda$. Then
\[
\mbox{\emph{AR}}_{C_\Gamma (H_\xi )}\leq \mbox{\emph{AR}}_{C_\Gamma (H_\beta)}
\]
\end{claim}

\begin{proof}[Proof of Claim \ref{claim:alpha+1}]
We show by transfinite induction on $\beta < \lambda$ that $\{ \mbox{AR}_{C_\Gamma (H_\xi)}\} _{\xi \leq \beta}$ is increasing in $\xi$. If $\beta = 0$ this is trivial. If $\beta =\alpha +1$ is a successor ordinal then the induction hypothesis tells us that $\{ \mbox{AR}_{C_\Gamma (H_\xi )} \} _{\xi\leq \alpha}$ is increasing with $\xi$ and we must show that $\mbox{AR}_{C_\Gamma (H_\alpha )} \leq \mbox{AR}_{C_\Gamma (H_{\alpha +1})}$.

Since $H_\alpha$ is normal in $H_{\alpha +1}$, Proposition \ref{prop:ARbasic}.(2) shows that $H_{\alpha +1}$ normalizes $\mbox{AR}_{C_\Gamma (H_\alpha )}$.  Thus, for $\delta \in H_{\alpha +1}$ and $\gamma \in \mbox{AR}_{C_\Gamma (H_\alpha )}$ we have
\begin{align*}
(\delta \gamma \delta ^{-1})\gamma ^{-1} &\in \mbox{AR}_{C_\Gamma (H_\alpha )}  \\
\delta (\gamma \delta ^{-1}\gamma ^{-1}) &\in H_\mu (\gamma H_{\mu}\gamma ^{-1}) = H_\mu \\
\mbox{so that } \ \ \ \ \ \delta\gamma \delta ^{-1}\gamma ^{-1} &\in \mbox{AR}_{C_\Gamma (H_\alpha )} \cap H_\mu = \{ e \}
\end{align*}
by (\ref{eqn:triv}) (we use in the second line that $H_{\alpha +1}\leq H_\mu$ and $H_\mu \triangleleft \Gamma$). This shows that the groups $\mbox{AR}_{C_\Gamma (H_\alpha )}$ and $H_{\alpha +1}$ commute, and so $\mbox{AR}_{C_\Gamma (H_\alpha )}$ is a subgroup of $C_\Gamma (H_{\alpha +1})$. As $C_\Gamma (H_{\alpha +1})$ is contained in $C_\Gamma (H_\alpha )$ we conclude that $\mbox{AR}_{C_\Gamma (H_\alpha )}$ is normal in $C_\Gamma (H_{\alpha +1})$ and therefore $\mbox{AR}_{C_\Gamma (H_\alpha )} \leq \mbox{AR}_{C_\Gamma (H_{\alpha +1})}$.

Now suppose $\beta$ is a limit ordinal. The induction hypothesis tells us that $\{ \mbox{AR}_{C_\Gamma (H_\xi )} \} _{\xi <\beta}$ is increasing with $\xi <\beta$ and we must show that $\mbox{AR}_{C_\Gamma (H_\xi )} \leq \mbox{AR}_{C_\Gamma (H_{\beta})}$ for all $\xi < \beta$. Fix $\xi < \beta$. For each $\alpha$ with $\xi\leq \alpha < \beta$ we have that $\mbox{AR}_{C_\Gamma (H _\xi )} \leq \mbox{AR}_{C_\Gamma (H_\alpha)} \leq C_\Gamma (H_\alpha )$. Intersecting this over all such $\alpha$ shows
\begin{align*}
\mbox{AR}_{C_{\Gamma}(H_\xi )} &\leq \bigcap _{\xi \leq \alpha < \beta} C_\Gamma (H_\alpha ) = C_\Gamma \big( \bigcup _{\xi\leq \alpha < \beta}H_\alpha \big) = C_\Gamma (H_\beta ) .
\end{align*}
Since $C_\Gamma (H_\beta ) \leq C_\Gamma (H_\xi )$ we actually have $\mbox{AR}_{C_{\Gamma}(H_\xi )} \triangleleft C_\Gamma (H_\beta )$ and so $\mbox{AR}_{C_\Gamma (H_\xi )}  \leq \mbox{AR}_{C_\Gamma (H_\beta )} $, which finishes the proof of the claim.
\qedhere[Claim \ref{claim:alpha+1}]
\end{proof}

Given now any $\alpha < \lambda$ we have shown that $\mbox{AR}_{C_{\Gamma}(H_\alpha )} \leq \mbox{AR}_{C_\Gamma (H_\mu )}$. But $H_\mu$ is normal in $\Gamma$ and $\mbox{AR}_\Gamma = \{ e \}$, so Lemma \ref{lem:normAR} shows that $\mbox{AR}_{C_\Gamma (H_\mu )} = \{ e \}$ and therefore $\mbox{AR}_{C_\Gamma (H_\alpha )} = \{ e \}$ as was to be shown.
\qedhere[Lemma \ref{lem:downward}]
\end{proof}

\begin{lemma}\label{lem:upward}
Let $\{ H_\alpha \} _{\alpha \leq \lambda}$ be an ascendant series of length $\lambda$ with $H_0=H$ and $H_\lambda = \Gamma$. Suppose that $\mbox{\emph{AR}}_{C_\Gamma (H)}=\mbox{\emph{AR}}_{H} = \{ e \}$. Then $\mbox{\emph{AR}}_\Gamma = \{ e \}$.
\end{lemma}

\begin{proof}
We proceed by transfinite induction on the length $\lambda$ of the series.

{\bf Limit stages:} Suppose first that $\lambda$ is a limit ordinal. By intersecting each group in the series $\{ H_\alpha \} _{\alpha \leq \lambda}$ with $C_\Gamma (H)$ we obtain the series $\{ C_{H_\alpha}(H) \} _{\alpha \leq \lambda}$, which is ascendant in $C_\Gamma (H)$. %Then $\Gamma = \bigcup _{\alpha < \lambda} H_\alpha$ and $C_\Gamma (H_0) = \bigcup _{\alpha <\lambda}C_{H_\alpha}(H_0)$ where both unions are increasing.
%
%\begin{claim}
%The sequence $\{ C_{H_\alpha} (H_0) \} _{\alpha \leq \lambda}$ is an ascendant series in $C_\Gamma (H_0)$.
%\end{claim}
%
%\begin{proof}[Proof of Claim]
%We show by transfinite induction on ordinals $\beta$ with $\beta \leq \lambda$ that $\{ %C_{H_\xi} (H_0) \} _{\xi \leq \beta}$ is an ascendant series in $C_{H_\beta}(H_0)$.  Suppose %that $\beta = \alpha +1$ is a successor ordinal. The group $H_\alpha$ is normalized by %$H_{\alpha +1}$, so it is also normalized by $C_{H_{\alpha +1}}(H_0)$. It follows that
%\[
%C_{H_\alpha} (H_0) = C_{H_{\alpha +1}}(H_0) \cap H_\alpha \triangleleft C_{H_{\alpha +1}}(H_0),
%\]
%so by the induction hypothesis $\{ C_{H_\xi} (H_0) \} _{\xi \leq \alpha +1}$ is an ascendant %series in $C_{H_{\alpha +1}}(H_0)$.
%
%When $\beta$ is a limit ordinal we have $H_\beta = \bigcup _{\xi <\beta}H_\xi$ where the union %is increasing. Thus
%\[
%C_{H_\beta}(H_0) = \bigcup _{\xi <\beta}(C_{H_\beta}(H_0)\cap H_\xi ) = \bigcup _{\xi %<\beta}C_{H_\xi}(H_0)
%\]
%which is an increasing union. This proves the claim.
%\end{proof}
%
Applying Lemma \ref{lem:aaseries} to the series $\{ C_{H_\alpha} (H) \} _{\alpha \leq \lambda}$ we obtain
\[
\bigcup _{\alpha <\lambda} \mbox{AR}_{C_{H_\alpha}(H)} = \mbox{AR}_{C_\Gamma (H)} .
\]
Since $\mbox{AR}_{C_\Gamma (H)} = \{ e \}$ we conclude that $\mbox{AR}_{C_{H_\alpha}(H)} = \{ e \}$ for all $\alpha < \lambda$. In addition we have $\mbox{AR}_{H} = \{ e \}$ so it follows from the induction hypothesis (applied to each series $\{ H_\xi \} _{\xi <\alpha}$ for $\alpha <\lambda$) that $\mbox{AR}_{H_\alpha} = \{ e \}$ for all $\alpha$. Another application of Lemma \ref{lem:aaseries} now shows that $\mbox{AR}_\Gamma = \bigcup _{\alpha <\lambda}\mbox{AR}_{H_\alpha} = \{ e \}$.

{\bf Successor stages:} Now assume that $\lambda = \mu +1$ is a successor ordinal. Since $H_\mu$ is normal in $H_{\mu +1} =\Gamma$ we have $C_{H_\mu}(H) \triangleleft C_\Gamma (H)$. It follows that $\mbox{AR}_{C_{H_\mu}(H)} \leq \mbox{AR}_{C_\Gamma (H)} = \{ e\}$ and so
\[
\mbox{AR}_{C_{H_\mu}(H)} = \{ e\} .
\]
By assumption $\mbox{AR}_{H} = \{ e \}$ so the induction hypothesis applied to $\{ H_\alpha \} _{\alpha \leq \mu}$ implies that
\begin{equation}\label{eqn:ARHmu}
\mbox{AR}_{H_\mu} = \{ e \} .
\end{equation}
Since $H_\mu$ is normal in $\Gamma$, $C_\Gamma (H_\mu )$ is normal in $\Gamma$ as well. In addition, $C_\Gamma (H_\mu )$ is contained in $C_\Gamma (H)$, so in fact $C_\Gamma (H_\mu ) \triangleleft C_\Gamma (H)$. It follows that
\begin{equation}\label{eqn:CARHmu}
\mbox{AR}_{C_\Gamma (H_\mu)} \leq \mbox{AR}_{C_\Gamma (H)} = \{ e \} .
\end{equation}
We see from (\ref{eqn:ARHmu}) and (\ref{eqn:CARHmu}) that the normal subgroup $H_\mu$ of $\Gamma$ satisfies the hypotheses of Lemma \ref{lem:normAR} and so $\mbox{AR}_\Gamma = \{ e \}$. This completes the induction.
\end{proof}

\begin{theorem}\label{thm:ascsubgrp}
Let $H$ be an ascendant subgroup of a countable group $\Gamma$. Then $\mbox{\emph{AR}}_\Gamma = \{ e \}$ if and only if $\mbox{\emph{AR}}_H = \{ e \}$ and $\mbox{\emph{AR}}_{C_\Gamma (H)} = \{ e \}$.
\end{theorem}
\end{appendices}

$\ $\\
\noindent Department of Mathematics \\
California Institute of Technology \\
Pasadena, CA 91125 \\
\texttt{rtuckerd@caltech.edu}
\end{document}